\documentclass[11pt,reqno]{amsproc}
\usepackage[english]{babel}
\usepackage[margin=1in]{geometry}
\usepackage[utf8]{inputenc}
\usepackage{amsmath, amsthm, amssymb}
\usepackage{hyperref}
\hypersetup{colorlinks=true, pdfstartview=FitV, linkcolor=BrickRed,citecolor=BrickRed, urlcolor=BrickRed}
\usepackage[abbrev,lite,nobysame]{amsrefs}
\usepackage{times}
\usepackage[usenames,dvipsnames]{color}
\usepackage{mathtools}
\usepackage{bm}
\usepackage{verbatim}

\renewcommand{\bar}[1]{\underline{#1}}
\newcommand{\cH}{\eta} 

\def\nh {{\rm nh}}
\def\h {{\rm h}}
\def\d {{\rm d}}
\def\eq {{\rm eq}}

\DeclareMathOperator{\dd}{d\!}
\DeclareMathOperator{\Id}{Id}
\DeclareMathOperator{\esssup}{ess\,sup}
\DeclareMathOperator{\di}{div}

\def\NN {\mathbb{N}}
\def\RR {\mathbb{R}}

\def\balpha{{\bm{\alpha}}}
\def\bbeta{{\bm{\beta}}}
\def\bgamma{{\bm{\gamma}}}

\newcommand{\be}{{\bm{e}}}
\newcommand{\bx}{{\bm{x}}}
\newcommand{\bu}{{\bm{u}}}
\newcommand{\bv}{{\bm{v}}}
\newcommand{\br}{{\bm{r}}}

\newcommand{\bL}{{\bm{L}}}
\newcommand{\bP}{{\bm{P}}}
\newcommand{\bR}{{\bm{R}}}
\newcommand{\bU}{{\bm{U}}}

\newcommand{\cE}{\mathcal{E}}
\newcommand{\cF}{\mathcal{F}}

\newcommand{\cC}{\mathcal{C}}

\newcommand{\Peq}{\bar P_\eq}
\newcommand{\Ph}{P_\h}
\newcommand{\Pnh}{P_\nh}

\def\i{\imath }

\def\de{{\partial}}

\DeclarePairedDelimiter\norm{\big\lvert}{\big\rvert}
\DeclarePairedDelimiter\Norm{\big\lVert}{\big\rVert}

\numberwithin{equation}{section}
\theoremstyle{definition}

\newtheorem{proposition}{Proposition}[section]
\newtheorem{theorem}[proposition]{Theorem}
\newtheorem{corollary}[proposition]{Corollary}
\newtheorem{lemma}[proposition]{Lemma}
\newtheorem{remark}[proposition]{Remark}

\title[On the hydrostatic limit of stably stratified fluids]{On the hydrostatic limit of stably stratified fluids with isopycnal diffusivity}

\author[R.\ Bianchini]{Roberta Bianchini}
\address{IAC, Consiglio Nazionale delle Ricerche, 00185 Rome, Italy}
\email{roberta.bianchini@cnr.it}
\author[V.\ Duchêne]{Vincent Duchêne}
\address{IRMAR, CNRS, Univ. Rennes, F-35000 Rennes, France}
\email{vincent.duchene@univ-rennes.fr}


\begin{document}
\begin{abstract}
This article is concerned with rigorously justifying the hydrostatic limit for continuously stratified incompressible fluids under the influence of gravity.

The main distinction of this work compared to previous studies is the absence of any (regularizing) viscosity contribution added to the fluid-dynamics equations; only thickness diffusivity effects are considered. Motivated by applications to oceanography, the diffusivity effects in this work arise from an additional advection term, the specific form of which was proposed by Gent and McWilliams in the 1990s to model the effective contributions of geostrophic eddy correlations in non-eddy-resolving systems.

The results of this paper heavily rely on the assumption of stable stratification. We establish the well-posedness of the hydrostatic equations and the original (non-hydrostatic) equations for stably stratified fluids, along with their convergence in the limit of vanishing shallow-water parameter. These results are obtained in high but finite Sobolev regularity and carefully account for the various parameters involved.

A key element of our analysis is the reformulation of the systems using isopycnal coordinates, enabling us to provide meticulous energy estimates that are not readily apparent in the original Eulerian coordinate system.
\end{abstract}

\maketitle

\section{Introduction}
The following system describes the evolution of heterogeneous incompressible flows under the influence of gravity,
\begin{equation}\label{eq:non-hydrostatic}
	\begin{aligned}
		\de_t \rho + (\bu+\bu_\star) \cdot \nabla_\bx \rho + (w+w_\star) \de_z \rho&=0, \\
		\rho\big( \de_t \bu + ((\bu+\bu_\star) \cdot \nabla_\bx) \bu + (w+w_\star) \de_z \bu\big) + \nabla_\bx P&=0, \\
		\rho\big(\de_t w +(\bu+ \bu_\star) \cdot \nabla_\bx w + (w+w_\star) \de_z w \big)+ \de_z P +g\,\rho&=0, \\
		\nabla_\bx \cdot\bu +\de_z w&=0,\\
		P|_{z=\zeta}-P_{\rm atm}&=0,\\
		\de_t\zeta+(\bu+\bu_\star)|_{z=\zeta}\cdot\nabla_\bx \zeta-(w+w_\star)|_{z=\zeta}&=0,\\
		w|_{z=-H}&=0.
	\end{aligned}
\end{equation}
Here, $t$ and $(\bx,z)$ are the time, and horizontal-vertical space variables, and we denote by $\nabla_\bx,\nabla_\bx\cdot, \Delta_\bx$ the gradient, divergence and Laplacian with respect to $\bx$. The vector field $(\bu, w) \in \RR^d \times \RR$ is the (horizontal and vertical) velocity, $\rho>0 $ is the density, $P\in\RR$ is the incompressible pressure, all being defined in the spatial domain 
\[
\Omega_t=\{(\bx,z) \ : \ \bx\in\RR^d ,\ -H<z<\zeta(t,\bx)\}, 
\]
where $\zeta(t, \bx)$ describes the location of a free surface, and $H$ is the depth of the layer at rest. The gravity field is assumed to be constant and vertical, and $g>0$ is the gravity acceleration constant.  Finally, the advection terms associated with the 'bolus velocity' $(\bu_\star, w_\star) \in \RR^d \times \RR$ were proposed by Gent and McWilliams~\cite{GMCW90}. These terms are introduced to account for the contribution of geostrophic eddy correlations to the effective transport velocities in non-eddy-resolving (large-scale) models. Their specific forms in the simplest case of constant diffusivity parameter $\kappa$ read as follows
\begin{equation}\label{eq:McWilliams}
	\bu_\star=\kappa \de_z\left(\frac{\nabla_\bx\rho}{\de_z\rho}\right)\ , \quad w_\star = -\kappa \nabla_\bx\cdot\left(\frac{\nabla_\bx\rho}{\de_z\rho}\right)\,, \qquad \kappa>0\,.
\end{equation}
Discarding the effective advection terms  (i.e. setting $\kappa=0$), one recovers the Euler equations for heterogeneous incompressible fluids under the influence of vertical gravity forces, where the last two lines of~\eqref{eq:non-hydrostatic} model the kinematic equation at the free surface and the impermeability condition of the rigid bottom respectively.

In~\eqref{eq:non-hydrostatic}, the pressure $P$ can be recovered from its (atmospheric) value at the surface, $P_{\rm atm}$, by solving the elliptic boundary-value problem induced by the incompressibility constraint of divergence-free velocity fields. Yet in
the shallow-water regime, where the horizontal scale of the perturbation is large compared with the depth of the layer $H$, formal computations (see below) suggest that vertical accelerations can be neglected and that the pressure $P$ approximately satisfies the hydrostatic balance law, that is
\begin{equation}\label{eq:hydr-balance}
	\de_z P+g\,\rho=0.
\end{equation}
Replacing the equation for the vertical velocity in ~\eqref{eq:non-hydrostatic} by the identity in \eqref{eq:hydr-balance} yields the so-called \emph{hydrostatic equations}:
\begin{equation}\label{eq:hydrostatic}
	\begin{aligned}
		\de_t \rho + (\bu+\bu_\star)  \cdot \nabla_\bx \rho + (w+w_\star)  \de_z \rho&=0, \\
		\rho\big( \de_t \bu + ((\bu+\bu_\star) \cdot \nabla_\bx) \bu + (w+w_\star) \de_z \bu\big) + \nabla_\bx P&= 0, \\
		\de_t\zeta+(\bu+\bu_\star)|_{z=\zeta}\cdot\nabla_\bx \zeta-(w+w_\star)|_{z=\zeta}&=0,\\
		P&=P_{\rm atm}+g\int_z^\zeta \rho(z',\cdot) \dd z',\\
		w&=-\int_{-H}^z\nabla_\bx \cdot\bu(z',\cdot) \dd z'.
	\end{aligned}
\end{equation}
{\em Our aim in this work is to rigorously justify the hydrostatic equations~\eqref{eq:hydrostatic} as an asymptotic model for the non-hydrostatic equations~\eqref{eq:non-hydrostatic}-\eqref{eq:McWilliams} in the shallow-water regime, for regular and stably stratified flows.}

	\subsection{The ``bolus'' velocities in geophysical flows}
	Let us now delve into the physical motivation behind introducing additional transport velocities, $\bu_\star$ and $w_\star$, defined in~\eqref{eq:McWilliams}, into the systems of equations~\eqref{eq:non-hydrostatic} and~\eqref{eq:hydrostatic}. As mentioned earlier, these ``bolus'' velocities serve as modifications to the standard incompressible Euler equations, whether with or without the hydrostatic assumption, derived from fundamental principles of fluid mechanics. In this regard, they play a role analogous to the viscosity contributions in the Navier--Stokes equations. Concerning the latter, it is worth pointing out that in theoretical and laboratory studies on density-stratified geophysical flows, viscosity effects do not model molecular viscosity but rather ``turbulent'' or ``eddy'' viscosities.
	It is important to mention that in the shallow-water regime, where horizontal scales are larger than vertical scales, the non-dimensionalization of standard molecular viscosity contributions results in anisotropic viscosity terms that are predominant in the vertical direction. In contrast, turbulent viscosity is widely reported to be predominant in the horizontal (or more precisely isopycnal) direction; see, for example, \cite{Griffies2003}*{Section 17.6}. In this work, we choose to neglect viscosity effects entirely and instead focus on eddy-induced diffusivity.

	The deterministic modeling of effective diffusivity induced by eddy correlation that we adopt in this work takes its roots in the 1990s and is due to Gent \& McWilliams~\cite{GMCW90}, see also~\cites{GMCW95, GMCW96}. 
	Motivated by the need to model the averaged dissipative effects of mesoscale eddies, which were not computationally feasible to resolve at that time, on the large-scale flow at the macroscopic level, they consider unknowns $(\rho, \bu, w)$ as the large-scale components of the density and velocity field, respectively. The approach involves introducing suitable correctors to the mass conservation equation and equations for tracers (such as salinity and temperature). Specifically, they suggest incorporating bolus velocity fields $\bu_\star$ and $w_\star$, defined in~\eqref{eq:McWilliams}.
	Adding the bolus velocity contribution in the momentum conservation equation was suggested in \cites{GMCW96,Gent01} as well.\footnote{Let us point out that our analysis would hold (and would in fact be simpler) without the contributions $\bu_\star$ and $w_\star$ in the evolution equations for the velocity. We add these terms because we believe they are important from a modeling point of view, and would play a crucial role in a refined analysis of the large-time behavior and/or less regular (weak) solutions.}  One of the main ingredient leading to the specific form of the bolus velocity fields is that the averaged dissipative effect of mesoscale eddies should act predominantly along isopycnals sheets (that is along and not across surfaces of equal densities).
	It should be mentioned that the Gent \& McWilliams eddy-induced advection is only part of the subgridscale parametrization of mesoscale eddies used in global ocean general circulation models, which also include Redi's eddy-induced diffusion~\cite{Redi82}. Specifically, the equations for tracers including the so-called Gent--McWilliams--Redi eddy parametrization read (see \cite{KornTiti} and references therein)
	\begin{equation}\label{eq.tracers}
		\partial_t C+ (\bU_3\cdot \nabla_3) C= \nabla_3\cdot \big(K_{\rm R} \nabla_3 C\big)+ \nabla_3\cdot \big(K_{\rm GM} \nabla_3 C\big),
	\end{equation}
	where $C$ is a tracer (typically the temperature $\theta$ and the salinity $S$), $\bU_3:=(\bu,w)$ denotes the three-dimensional velocity field, and $\nabla_3$ the three-dimensional gradient. Above, we define $K_{\rm R} $ and $K_{\rm GM}$ by
	\[ K_{\rm R}:=\frac{K_I}{1+|\bL|^2}
	\begin{pmatrix} 1+L_y^2 & -L_x L_y & L_x \\
		-L_x L_y & 1+L_x^2 & L_y \\
		L_x & L_y & |\bL|^2
	\end{pmatrix}
	+\frac{K_D}{1+|\bL|^2}
	\begin{pmatrix} L_x^2 & L_x L_y & -L_x \\
		L_x L_y & L_y^2 &- L_y \\
		-L_x & -L_y & 1
	\end{pmatrix}
	\]
	\[ K_{\rm GM}:=\kappa
	\begin{pmatrix} 0 & 0 & -L_x \\
		0 & 0 & -L_y \\
		L_x & L_y & 0
	\end{pmatrix}
	\]
	where we used the notation $\bL:=(L_x,L_y)=\frac{-\nabla_{\bx} \rho}{\partial_z\rho}$, and the non-negative constants $K_I$ and $K_D$ are the isoneutral and dianeutral diffusivity coefficients. Notice that the tensor $K_{\rm R}$ is symmetric, while the tensor $K_{\rm GM}$ is skew-symmetric, so that the dissipative effect due to Gent \& McWilliams eddy-advection is not obvious. In fact, there has been some debate about the adiabatic nature of the Gent \& McWilliams parametrization \cite{Gent11}. As it is evident from the formulation of the equations in isopycnal coordinates (see equations \eqref{eq:hydro-iso-intro0} below), the Gent \& McWilliams eddy-induced advection terms introduce a diffusive contribution to the thickness variable. This is why the parameter $\kappa$ is often referred to as the {\em thickness diffusivity} parameter.
	
	Another important aspect is that, when applied to the density variable, the isoneutral contribution of Redi's eddy-induced diffusion vanishes exactly, while the dianeutral contribution becomes an isotropic three-dimensional Laplacian, and the Gent \& McWilliams contribution can be interpreted as advection. Consequently, in situations where the equation of state $\rho=\rho(\theta,S)$ is a linear combination of the tracers (with equal diffusivity parameters for all tracers), the equations \eqref{eq.tracers} yield
	\[ \partial_t \rho+ (\bU_3\cdot \nabla_3) \rho +(\bU_3^\star\cdot \nabla_3 )\rho = K_D\Delta_3\rho,\]
	where we denote $\bU_3^\star:=(\bu^\star,w^\star)$ defined in~\eqref{eq:McWilliams}.
	Since the dianeutral diffusivity coefficient is typically set much smaller than the isoneutral diffusivity coefficient, $0\leq K_D\ll K_I$, it makes sense to neglect the right-hand side of the above equation. This results in the mass conservation equations in~\eqref{eq:non-hydrostatic} and~\eqref{eq:hydrostatic}.

	Let us emphasize that our framework leaves aside many important ingredients which are usually considered in the so-called primitive equations modeling large-scale flows (see {\em e.g.} \cite{Griffies2003}), and which are in fact responsible for the emergence of mesoscale eddies that led to the Gent--McWilliams--Redi eddy parametrization: typically the rotational effects, vertical boundaries, bathymetry and nonlinear equation of state. 
	As such neither the non-hydrostatic system \eqref{eq:non-hydrostatic} nor the hydrostatic system~\eqref{eq:hydrostatic} can be considered as relevant global ocean models. 
	This work is motivated by studying on a theoretical grounds the interplay between (stable) stratification, hydrostatic (shallow water) limits, and the eddy-induced thickness diffusivity contributions.
	Many of these aforementioned constituents which are essential to modeling aspects are not necessarily important in the mathematical analysis, and could be easily incorporated in our study at the price of blurring the main mechanisms that we would like to pinpoint, while interesting singular limits (geostrophic balance, boundary layers, {\em etc.}) would of course require and deserve a specific treatment.
	
	Let us finally mention that there exists a huge mathematical literature dedicated to the investigation of fluid-dynamics equations in the probabilistic setting, where the cumulative effect of mesoscale eddies on the large-scale flow is modeled by means of suitable (additive or multiplicative) noises. For that context, we refer to~\cites{flandoliGL2021,flandoliGL2021-2,ChapronCrisanHolmEtAl23,ChapronCrisanHolmEtAl24}, while our setting will be completely deterministic.

	\subsection{Two-velocity hydrodynamics in other contexts} \label{S.two-velocities} Interestingly, equations involving two distinct velocities have been proposed and studied independently, in different contexts. We refer, for instance, to a series of works initiated by Brenner in \cites{Brenner04,Brenner05} (see also \cite{Brenner12}) on compressible barotropic flows, which are described by the following equations:
	\begin{equation}\label{eq:Brenner}
		\begin{aligned}
			\partial_t \varrho+\nabla_\bx\cdot(\varrho\bu)&=\nabla_\bx\cdot(\kappa\nabla_\bx  \varrho),\\
			\partial_t(\varrho\bu)+\di_\bx\big(\varrho\bu\otimes (\bu-\kappa\tfrac{\nabla_\bx  \varrho}{ \varrho}) \Big)+ \nabla_\bx p&=\di_\bx\mathbb S(\nabla_\bx\bu),
		\end{aligned}
	\end{equation}
	where the convention for the divergence $\di_\bx$ is such that $\di_\bx(\bu\otimes\bv)=(\bv\cdot\nabla_\bx)\bu+(\nabla_\bx\cdot\bv)\bu$, and $\mathbb S$ is the standard Newtonian viscous stress:
	\begin{equation}\label{eq:viscous-stress}
		\mathbb S(\nabla_\bx\bu)=\mu\big(\nabla_\bx\bu+\nabla_\bx^{\rm T}\bu \big)+\lambda(\nabla_\bx\cdot\bu) \Id.
	\end{equation}
	Moreover, $\varrho$ (resp. $\bu$) represents the density (resp. velocity), and $p=p(\varrho)$ is the given pressure function.
	The system \eqref{eq:Brenner} proposes additional diffusive contributions to the standard barotropic Navier--Stokes equation. Remarkably, these additional terms bear a striking resemblance to the Gent \& McWilliams ``bolus velocity'' advection when expressed in isopycnal coordinates. Indeed, anticipating the derivation in Section \ref{S.isopycnal}, the hydrostatic system \eqref{eq:hydrostatic} may be reformulated as
	\begin{equation}\label{eq:hydro-iso-intro0}
		\begin{aligned}
			\partial_t  h+\nabla_\bx\cdot(h\bu)&=\nabla_\bx\cdot(\kappa\nabla_\bx  h),\\
			\partial_t(h\bu)+\di_\bx\big(\varrho\bu\otimes(\bu-\kappa\tfrac{\nabla_\bx  h}{ h})\Big)+ h\nabla_\bx \psi&=0,
		\end{aligned}
	\end{equation}
	where $h$ represents the infinitesimal thickness of pycnoclines, and the explicit expression of the Montgomery potential $\psi$ is given below. Despite their apparent similarity, there are important differences between \eqref{eq:Brenner} and \eqref{eq:hydro-iso-intro0}, in addition from the obvious fact that the viscous stress is discarded in the latter ($\mu=\lambda=0$).
	In \eqref{eq:Brenner}, the space variable $\bx$ is typically three-dimensional, whereas in analogous situations, it is only two-dimensional in \eqref{eq:hydro-iso-intro0}. However, the variables $h$ and $\bu$ therein depend on an additional ``density'' variable, which labels pycnoclines based on their vertical distribution. Importantly, the contribution $\psi$ is not expressed as a function of $h$ but rather as a linear operator within this infinite-dimensional framework:
	\begin{equation}\label{eq:Montgomery-intro}
		\psi(t,\bx,r)
		=g\int_{\rho_0}^r \min(1,r'/r) h(t,\bx,r')\dd r' .
	\end{equation}
	In particular, a crucial difference between \eqref{eq:Brenner} and \eqref{eq:hydro-iso-intro0} is that, in the absence of any diffusive and viscous effect ($\kappa=\mu=\lambda=0$), the initial-value problem for \eqref{eq:Brenner} is locally-in-time well-posed in Sobolev spaces $H^s(\RR^d)$ when $s>d/2+1$ for initial data satisfying the non-vacuum condition $\varrho\geq c_0>0$ and provided that the pressure function is sufficiently regular and satisfies $p'(\varrho)>0$, while  the well-posedness of the initial-value problem for \eqref{eq:hydro-iso-intro0} in finite-regularity spaces is an open problem, as we discuss below.
	
	Notwithstanding this fact we can remark that, from a mathematical point of view, the diffusive contributions ($\kappa>0$) provide a very valuable regularization mechanism which remedies issues encountered in the mathematical theory (in particular the global-in-time existence of weak solutions) of the standard Navier--Stokes system, that is the absence of uniform bounds on the density $\varrho$, the possibility of the  appearance of vacuum regions, and the possible development of uncontrollable oscillations for the density due to the low regularity of the velocity field; see \cites{FeireislVasseur10,FeireislGwiazdaSwierczewska-GwiazdaEtAl16}. In fact, motivated by the construction of suitable regulariza\-tions of the compressible Euler equations (in the baroclinic framework), Guermond and Popov derived independently a generalization of \eqref{eq:Brenner} without invoking  phenomenological assumptions \cite{GuermondPopov14}.

	Interestingly, systems similar to \eqref{eq:Brenner} arise also in relation with compressible barotropic flows with degenerate viscosities (which coincides when $p(\rho)\propto \rho^2$ with the viscous shallow water equations advocated by Gent in \cite{Gent93} and derived from the Navier--Stokes equations in \cites{GerbeauPerthame01,Marche07,BreschNoble07}, and to a variant of the the Aw--Rascle model for traffic flows studied in \cite{ChaudhuriGwiazdaZatorska23} in the pressureless case; similar models also appear in quantum hydrodynamics, see \cite{Juengel10}). Indeed, restricting to spatial dimension $d=1$, discarding the Newtonian viscous stress ($\mu=\lambda=0$), and denoting $v =u-\kappa\frac{\de_x\varrho}{\varrho}$, we find after straightforward computations that \eqref{eq:Brenner} reads
	\begin{equation}\label{eq:Bresch}
		\begin{aligned}
			\partial_t \varrho+\de_x(\varrho v)&=0,\\
			\partial_t(\varrho v)+\de_x\big(\varrho v^2 \big)+ \nabla_\bx p&=\de_x(\kappa\varrho\de_x v).
		\end{aligned}
	\end{equation}
	In the above, the diffusive contributions act as a degenerate viscosity, in the sense that the variable viscosity coefficient vanishes when the density $\varrho$ vanishes. In their analysis of such systems (and generalizations thereof), Bresch and Desjardins introduced the so-called BD entropy in \cites{BreschDesjardins03,BreschDesjardins04}. In fact this BD entropy stems from a reformulation of systems of the form \eqref{eq:Bresch} (in higher dimension) with viscous stresses as in \eqref{eq:viscous-stress} but with variable coefficients $\lambda(\varrho)$ and $\mu(\varrho)$, under a form analogous to \eqref{eq:Brenner}. This is made possible thanks to a key cancellation when a special algebraic relation between $\lambda(\varrho)$ and $\mu(\varrho)$ holds, namely $\lambda(\varrho)=2(\varrho\mu'(\varrho)-\mu(\varrho))$. The system \eqref{eq:Bresch} corresponds to the special case $\mu(\rho)=\frac12\kappa\varrho$ and $\lambda(\varrho)=0$ initially considered in \cite{BreschDesjardinsLin03}. This discovery triggered a series of work, including \cites{MelletVasseur07,VasseurYu16,BreschDesjardinsZatorska15,BreschVasseurYu22} (see also references therein) which address the global existence of weak solutions to the barotropic Navier--Stokes equations with degenerate viscosities. In \cites{BreschDesjardinsZatorska15,BreschVasseurYu22}, the authors extend the BD entropy to a so-called $\kappa$-entropy, considering the equation for the velocity $v+\theta\kappa\frac{\de_x \varrho}{\varrho}$ together with the equation for $v$  (where $\theta\in(0,1)$ is a free parameter, which is denoted $\kappa$ in \cite{BreschDesjardinsZatorska15}). This leads to an entropy depending on the parameter $\theta$, where the choice $\theta=0$ corresponds to the standard energy associated with \eqref{eq:Bresch}, the choice $\theta=1$ corresponds to the BD entropy, while the choice $\theta=1/2$ was considered in \cite{GisclonLacroix-Violet15}. In all these works, the crucial matter consists in exploiting some compactness on the density variable $\varrho$ gained from the parabolic nature of the evolution equation for the density when written with suitable modified velocities so as to prevent, in particular, the appearance of vacuum regions. Let us clarify again that our analysis concerns local-in-time regular solutions, and that for this matter the diffusivity-induced regularizing mechanism that we exploit is much less subtle than in the previously mentioned works, while the key difficulties in our analysis stem from the pressure contributions, namely \eqref{eq:Montgomery-intro} in \eqref{eq:hydro-iso-intro0}.
	
	Finally, we would like to highlight the works~\cites{CoudercDuranVila17,DuranVilaBaraille17}, in which the authors introduce (artificially) diffusive contributions analogous to those in \eqref{eq:Brenner} at the discrete level for the multilayer shallow water system. This is done with the aim of developing numerical schemes that ensure control over the discrete total energy.
	\footnote{The effective velocities are defined slightly differently in \cites{CoudercDuranVila17,DuranVilaBaraille17} compared with \eqref{eq:hydro-iso-intro0}, 
			since the analogous fully continuous equations read
			\[
			\begin{aligned}
				\partial_t  h+\nabla_\bx\cdot(h\bu)&=\nabla_\bx\cdot(\kappa\nabla_\bx  \psi),\\
				\partial_t(h\bu)+\di_\bx\big(\varrho\bu\otimes(\bu-\kappa\tfrac{\nabla_\bx  \psi}{ h})\Big)+ h\nabla_\bx \psi&=0,
			\end{aligned}
			\]
			where $\psi$ is defined in \eqref{eq:Montgomery-intro}. The multilayer shallow water system can be viewed as a system of several equations of the form \eqref{eq:Brenner}, coupled through the pressure contributions. Alternatively it may be interpreted as a semi-discretized (with respect to the density variable) version of the system \eqref{eq:hydro-iso-intro0}, as rigorously shown in \cite{Adim}. }

\subsection{Results on the hydrostatic equations and hydrostatic limit}
The initial-value problem for the hydrostatic equations \eqref{eq:hydrostatic} (and \eqref{eq:hydro-iso-intro0}-\eqref{eq:Montgomery-intro}) without diffusivity or viscosity contributions is not well-posed in finite-regularity functional spaces in general. In fact, restricting to homogeneous flows (that is $\rho$ being constant), ill-posedness was established by Renardy~\cite{Renardy09} at the linear level, and by Han-Kwan and Nguyen~\cite{HanKwanNguyen16} at the nonlinear level. Yet if we additionally assume that the initial data satisfy the Rayleigh condition of (strict) convexity/concavity in the vertical direction, well-posedness is restored~\cites{Brenier99,Grenier99,MasmoudiWong12}. 
The picture is very different in the framework of stably stratified flows. The celebrated Miles and Howard criterion~\cites{Miles61,Howard61} states that the linearized equations about equilibria $(\bar\rho(z),\bar\bu'(z))$ do not exhibit unstable modes (in dimension $d = 1$, see~\cite{Gallay}*{Remark 1.3} when $d = 2$) provided that the local Richardson number is greater than 1/4 everywhere, that is
\[\forall z\in [-H,0],\qquad  |\bar\bu'(z)|^2\leq 4 g \left(\frac{-\bar\rho'(z)}{\bar\rho(z)}\right).\]
Notice that the stabilizing (resp. destabilizing) effect of the stable stratification (resp. shear velocity) is clearly encoded by the above criterion.
However, let us underline again that the local\footnote{Incidentally, notice that it was proven that smooth solutions to the hydrostatic equations in the homogeneous framework can develop singularities in finite time; see~\cites{CaoIbrahimNakanishiEtAl15,Wong15}, and~\cite{IbrahimLinTiti21} in the presence of rotation. Again, such a result is not known in the stably stratified setting. Our study will be limited to local-in-time solutions.} well-posedness of the (nonlinear) hydrostatic equations in finite regularity spaces, even for initial data (strictly) satisfying the above inequality, is an open problem. 

This is in sharp contrast with the available results on the non-hydrostatic equations. In this context, we mention the recent work by Desjardins, Lannes and Saut~\cite{Desjardins-Lannes-Saut},
which is the closest to our framework and provides the well-posedness of the (inviscid and non-diffusive) non-hydrostatic equations in Sobolev spaces (using the rigid-lid assumption). Even though the stabilizing effect of the stable stratification is also a key ingredient of that work, it is not powerful enough to prove that the lifespan of the solutions to the non-hydrostatic equations 
is uniform with respect to the shallow-water parameter measuring the ratio of vertical to horizontal lengths, without additional smallness conditions on the initial data. A more detailed comparison between~\cite{Desjardins-Lannes-Saut} and our results is provided in Section \ref{S.strategy} below. 

From the technical viewpoint, the reason of this discrepancy -- in terms of the available results -- between the non-hydrostatic and  hydrostatic equations is that the vertical velocity variable $w$ changes its role passing from prognostic (when it belongs to the set of unknowns) to diagnostic (when it is reconstructed from the unknowns), whence losing one order of regularity; see the fourth equation in~\eqref{eq:hydrostatic}. In an analogous way, the available control on the pressure contributions lose one derivative with respect to the horizontal space variable between the non-hydrostatic and hydrostatic systems. One of the main observations in~\cite{Desjardins-Lannes-Saut} and in our work is that for stably stratified flows the equations benefit from a symmetric structure which can be exploited to partially (but not fully) overcome the difficulties related to these loss of derivatives. In this way we are able to pinpoint contributions which may generate high-frequency instabilities, revealing clearly the destabilizing influence of shear velocities. 

In order to deal with loss of derivatives without restricting the analysis to the analytic setting as in~\cites{KukavicaTemamVicolEtAl11,PaicuZhangZhang20}, one natural approach is the introduction of (regularizing) viscosity and diffusivity contributions.
This is the framework of most of the theoretical studies concerning the hydrostatic equations and/or the hydrostatic limit, starting with the work 
of Azérad and Guillén~\cite{AzeradGuillen2001}. A landmark in the theory is the work of Cao and Titi~\cite{CaoTiti07} where the 
global well-posedness of the initial-value problem for the hydrostatic equations was proved in dimension $d+1=3$: this striking result should be compared with the state of the art concerning the Navier-Stokes equations.
Several mathematical studies were later established to investigate situations involving partial viscosities and diffusivities, as well as more physically relevant boundary conditions.
Rather than providing an extensive bibliography for this huge set of results, we limit ourselves to point out the works~\cites{CT2014,CT2016}, which extended the previous results to the case where only horizontal viscosity and diffusivity are added to the equations. 
We also mention~\cites{LiTiti2019,FurukawaGigaHieberEtAl20,LiTitiYuan21} (in the homogeneous case) and~\cites{PuZhou21,PuZhou22} (in the heterogeneous case) for recent results on the hydrostatic limit and an extended list of references (therein).

A specificity of our analysis with respect to the previous ones (with the exception of~\cite{Desjardins-Lannes-Saut}) is that we shall crucially {\em use} the (stable) density stratification assumption, but we completely neglect viscosity-induced regularization and only allow for thickness diffusivity effects. We shall also keep track of all relevant parameters in our estimates, and in particular we use diffusivity-induced regularization only when crucially needed. This allows to characterize the relevant convergence rates and timescales, and to exhibit a balance between the destabilizing effect of shear velocities and the stabilizing result of thickness diffusivity.
Moreover, to the best of our knowledge,
this is the first rigorous mathematical study where the specific form of the diffusivity contributions, due to Gent and McWilliams~\cite{GMCW90} and modeling effective diffusivity induced by eddy correlation, is taken into account, with the exception of the work of Korn and Titi \cite{KornTiti} which appeared after our work was completed. In~\cite{KornTiti}, the authors extend the work~\cite{CaoTiti07} to a framework with the full Redi-Gent-McWilliams parametrization and nonlinear equation of state, pointing out the necessity of using a suitably regularized density in the definitions of the diffusion operators in~\eqref{eq.tracers}. Such regularization is not necessary in our analysis, which differs by two main ingredients. Firstly, as previously mentioned, we restrict ourselves to stably stratified flows and local-in-time regular solutions. Secondly, we will consider the equations~\eqref{eq:non-hydrostatic} and~\eqref{eq:hydrostatic} in the isopycnal coordinate system.

	\subsection{Plan of the paper} 
	
	The paper is organized as follows. In Section \ref{S.Main-results} we first rewrite the systems of equations~\eqref{eq:non-hydrostatic} and~\eqref{eq:hydrostatic} with bolus velocities \eqref{eq:McWilliams} in isopycnal coordinates. This yields systems~\eqref{eq:nonhydro-iso-intro} and~\eqref{eq:hydro-iso-intro}, respectively. We then present the main results of our work, namely Theorem~\ref{thm-well-posedness} on the initial-value problem for the hydrostatic equations~\eqref{eq:hydro-iso-intro}, and Theorem~\ref{thm-convergence} on the strong convergence of solutions to the non-hydrostatic equations~\eqref{eq:nonhydro-iso-intro} towards corresponding solutions to the hydrostatic equations~\eqref{eq:hydro-iso-intro}, as the shallow-water parameter vanishes.
	
	Section~\ref{S.Hydro} is dedicated to the proof of Theorem~\ref{thm-well-posedness}.
	
	In Section~\ref{S.NONHydro}, we analyze the non-hydrostatic equations. We first provide elliptic estimates for the boundary-value problem of the pressure reconstruction (Lemma~\ref{L.Poisson} and Corollary~\ref{C.Poisson}), and use them to infer two partial results concerning the initial-value problem: Proposition~\ref{P.NONHydro-small-time} (restricted to small time) and Proposition~\ref{P.NONHydro-large-time} (restricted to small data). 
	
	In Section~\ref{S.Convergence}, we bypass these restrictions assuming sufficiently small values of the shallow-water parameter, and complete the proof of Theorem~\ref{thm-convergence}. 
	
	Finally, in Appendix~\ref{S.Appendix} we provide product, commutator and composition estimates in anisotropic Sobolev spaces which are of independent interest.

	\section{Main results}\label{S.Main-results}

\subsection{The model in isopycnal coordinates and non-dimensionalization}\label{S.isopycnal}
Let us consider smooth solutions to~\eqref{eq:non-hydrostatic} defined on a time interval $I_t$. Assuming that the flow is \emph{stably stratified}, i.e.
\[
\inf(-\de_z \rho ) > 0,
\]
the density $\rho: z \rightarrow \rho(\cdot, \cdot, z)$ is an invertible function of $z$. We denote its inverse $\cH: \varrho \rightarrow \cH(\cdot, \cdot, \varrho)$, so that
\[
\rho(t, \bx, \cH(t, \bx, \varrho))=\varrho, \quad \cH(t, \bx, \rho(t, \bx, z))=z.
\]
We also assume that $\rho(t, \bx, -H)=\rho_1, \; \rho(t, \bx, \zeta(t, \bx))=\rho_0$ for $(t, \bx) \in I_t \times \RR^d$, where $\rho_0 < \rho_1$ are two fixed and positive constant reference densities. Then we have
\begin{equation}\label{def:eta}
	\cH :I_t\times\Omega\to \RR \quad \text{ with } \quad \Omega := \RR^d\times (\rho_0, \rho_1) \quad \text{ and } \quad  h:=- \partial_\varrho \cH>0 ,
\end{equation}
the latter inequality accounting for the stable stratification assumption.
We now introduce 
\[
\check\bu(t,\bx,\varrho)=\bu(t,\bx,\cH(t,\bx,\varrho)), \quad \check w(t,\bx,\varrho)=w(t,\bx, \cH(t,\bx,\varrho)),  \quad \check P(t,\bx,\varrho)=P(t,\bx, \cH(t,\bx,\varrho)).
\]
From the chain rule, we infer that
system~\eqref{eq:non-hydrostatic} in isopycnal coordinates reads 
\begin{equation}\label{eq:nonhydro-iso}
	\begin{aligned}
		\partial_t \cH+\check\bu \cdot\nabla_\bx \cH-\check w&=\kappa \Delta_\bx \cH,\\
		\varrho\Big( \partial_t\check\bu+\big((\check\bu-\kappa\tfrac{\nabla_\bx  h}{ h}\big)\cdot\nabla_\bx \big)\check\bu\Big)+ \nabla_\bx \check P+\frac{\nabla_\bx \cH}{ h} \partial_\varrho \check P&=0,\\
		\varrho\Big( \partial_t\check w+\big(\check \bu-\kappa\tfrac{\nabla_\bx  h}{ h}\big)\cdot\nabla_\bx \check w\Big)- \frac{\de_\varrho \check P}{ h} + g \varrho&=0,\\
		-h\nabla_\bx\cdot \check\bu-(\nabla_\bx \cH)\cdot (\partial_\varrho\check\bu) +\partial_\varrho \check w&=0,\\
		\check P\big|_{\varrho=\rho_0} =P_{\rm atm}, \qquad   \check w\big|_{\varrho=\rho_1}&=0.
	\end{aligned}
\end{equation}
Notice that differentiating with respect to $\varrho$ the first equation and using the fourth equation (stemming from the incompressibility constraint), the mass conservation reads
\begin{equation}\label{eq:h}
	\partial_t h+\nabla_\bx\cdot(h\check\bu)=\kappa\Delta_\bx h.
\end{equation}

At this point, we are ready to introduce a dimensionless version of the previous system. 
We are interested in deviations from steady solutions to the incompressible Euler equations with variable density:
\[
({h_{\rm eq}},\bu_{\rm eq}, w_{\rm eq}, P_{\rm eq})=({\bar h(\varrho)},\bar\bu(\varrho), 0, \bar P(\varrho)), 
\]
which satisfy the equilibrium condition
\[
\de_\varrho \bar P(\varrho) = g\varrho \bar h(\varrho).
\]
Therefore, we consider (non-necessarily small) fluctuations of that steady solution, so that our unknowns admit the following decomposition: 
\begin{align*}
	h(t,\bx,\varrho)&{=\bar h(\varrho)+h_{\rm pert}(t,\bx,\varrho) },& \quad \check\bu(t,\bx,\varrho)&=\bar{\bu}(\varrho)+ \bu_{\rm pert}(t,\bx,\varrho),  \\  
	\check w(t,\bx,\varrho)&=0+ w_{\rm pert}(t,\bx,\varrho), &\quad \check P(t,\bx,\varrho)&=\bar P(\varrho)+ P_{\rm pert}(t,\bx,\varrho).
\end{align*}
Furthermore, we non-dimensionalize the equations through the following scaled variables: we set
\[
\frac{h (t, \bx, \varrho)}{H}=\tilde {\bar h}(\varrho) +  \tilde h (\tilde t, \tilde\bx, \varrho)  \quad \text{ and } \quad \frac{\check \bu(t,\bx,\varrho)}{\sqrt{gH}}=\tilde{\bar\bu}( \varrho)+ \tilde \bu(\tilde t,\tilde\bx, \varrho),
\]
and \footnote{Notice the different scaling between the horizontal and vertical velocity fields. There, $\lambda$ is a reference horizontal length.}
\[   \frac{\lambda}{H}\frac{\check w(t,\bx,\varrho)}{\sqrt{gH}}= \tilde w(\tilde t,\tilde\bx, \varrho), \quad \frac{\check P(t,\bx,\varrho)}{gH}=\frac{P_{\rm atm}}{ g H} +\int_{\rho_0}^\varrho \varrho' \tilde{\bar h}(\varrho') \, d\varrho'+ \tilde P(\tilde t,\tilde\bx, \varrho),\]
where we use the following scaled coordinates\footnote{\label{f.rho}We could scale also the $\varrho$-coordinate. Adjusting accordingly the other variables, we can set without loss of generality $\rho_0=1$. In the following we shall not discuss the dependency with respect to $\rho_1$, and in particular the physically relevant limit of small density contrast, $\frac{\rho_1-\rho_0}{\rho_0}\ll1$; see~\cite{Duchene16} and references therein. } 
\[ \tilde \bx=\frac{\bx}{\lambda}  \quad \text{ and } \quad  \tilde t=\frac{\sqrt{gH}}{\lambda} t.\]
Introducing the dimensionless (thickness) diffusion parameter, $\tilde\kappa$ and the shallowness parameter, $\mu$, through
\[ \tilde \kappa = \frac{\kappa}{ \lambda\sqrt{gH}}  \quad \text{ and } \quad  \mu=\frac{H^2}{\lambda^2},\]
substituting the scaled coordinates/variables in system~\eqref{eq:nonhydro-iso} and the subsequent equation {\em and dropping the tildes for the sake of readability} yields 
\begin{align}\label{eq:nonhydro-iso-intro}
	\partial_{ t}   h+\nabla_{ \bx} \cdot\big(({\bar h}+ h)({\bar \bu}+   \bu)\big)&= \kappa \Delta_{ \bx}  h,\notag\\
	\varrho\Big( \partial_{ t}  \bu+\big(({\bar \bu} +   \bu -  \kappa\tfrac{\nabla_{ \bx}  h}{ {\bar h}+  h})\cdot\nabla_\bx\big)  \bu\Big)+ \nabla_{ \bx}  P+\frac{\nabla_{ \bx}  \cH}{ {\bar h}+  h}(\varrho {\bar h} + \partial_\varrho   P ) &=0,\\
	\mu \varrho\Big( \partial_{ t}  w+(\bar \bu +   \bu -  \kappa\tfrac{\nabla_{ \bx}   h}{ \bar h+  h})\cdot\nabla_{ \bx}   w\Big)- \frac{ \de_\varrho  P}{\bar h+  h} + \frac{\varrho  h}{\bar h +   h}&=0,\notag\\
	-(\bar h +   h) \nabla_{\bx} \cdot  \bu-\nabla_{\bx}  \cH  \cdot({\bar \bu}'+\partial_\varrho\bu) +\partial_\varrho  w&=0, \notag\quad \text{(div.-free cond.)} \\
	\cH(\cdot, \varrho)=\int_{\varrho}^{\rho_1}  h(\cdot, \varrho')\dd\varrho', \qquad   P\big|_{\varrho=\rho_0} =0, \qquad   w\big|_{\varrho=\rho_1}&=0.  \notag \quad \text{(bound. cond.)}
\end{align}

The hydrostatic system is obtained by setting $\mu=0$ in~\eqref{eq:nonhydro-iso-intro}. Specifically, plugging the hydrostatic balance
\[ \frac{ \de_\varrho  P}{\bar h+  h}  = \frac{\varrho  h}{\bar h +   h} \quad \text{ and } \quad   P\big|_{\varrho=\rho_0} =0\]
into the second equation of~\eqref{eq:nonhydro-iso-intro} yields
\begin{subequations}\label{eq:hydro-iso-intro}
	\begin{equation}\label{eq:hydro-iso-intro-eq}
		\begin{aligned}
			\partial_t  h+\nabla_\bx\cdot((\bar h+ h)({\bar \bu}+\bu))&=\kappa\Delta_\bx  h,\\
			\varrho\Big( \partial_t\bu+\big(({\bar \bu}+\bu-\kappa\tfrac{\nabla_\bx  h}{\bar h+ h})\cdot\nabla_\bx \big)\bu\Big)+ \nabla_\bx \psi&=0,
		\end{aligned}
	\end{equation}
	with 
	\begin{align} 
		\psi(t,\bx,\varrho)
		&=\int_{\rho_0}^\varrho \varrho' h(t,\bx,\varrho')\dd\varrho' + \varrho \int_{\varrho}^{\rho_1} h(t, \bx, \varrho')\dd\varrho' \nonumber\\
		&=\rho_0\int_{\rho_0}^{\rho_1} h(t,\bx,\varrho')\dd\varrho'   +\int_{\rho_0}^\varrho \int_{\varrho'}^{\rho_1}  h(t,\bx,\varrho'')\dd\varrho''\dd\varrho'. \label{eq:nablaphi-intro}
	\end{align}
\end{subequations}
We shall provide a rigorous proof of the convergence of (smooth) solutions to~\eqref{eq:nonhydro-iso-intro} towards (smooth) solutions to~\eqref{eq:hydro-iso-intro} as $\mu \searrow 0$, under the stable stratification assumption, $\bar h+h>0$.

\subsection{Our main results}
Our main results are stated and commented below. Some notations, and in particular the Sobolev spaces $H^{s,k}(\Omega)$, are introduced right after. First, we prove the existence, uniqueness and control of the solutions to the hydrostatic system~\eqref{eq:hydro-iso-intro} for sufficiently smooth initial data. Let us point out that the existence time of our solutions encodes the aforementioned stabilizing (resp. destabilizing) effect of the stable stratification and thickness diffusivity (resp. shear velocity).

\begin{theorem}\label{thm-well-posedness} 
	Let $s,k\in\NN$ be such that $s> 2 +\frac d 2$, $2\leq k\leq s$, and $\bar M,M,h_\star,h^\star>0$ and $0<\rho_0<\rho_1$ be fixed. Then there exists $C>0$ such that for any $\kappa\in(0,1]$, any $\bar h,\bar \bu\in W^{k,\infty}((\rho_0,\rho_1)) $ satisfying
	\[  \norm{\bar h}_{W^{k,\infty}_\varrho } + \norm{\bar \bu'}_{W^{k-1,\infty}_\varrho }\leq \bar M\]
	and any initial data $(h_0, \bu_0) \in H^{s,k}(\Omega)$, with $\cH_0(\cdot,\varrho)=\int_{\varrho}^{\rho_1} h_0(\cdot,\varrho')\dd\varrho'$, satisfying 
	the following estimate
	\[
	M_0:=\Norm{\cH_0}_{H^{s,k}}+\Norm{\bu_0}_{H^{s,k}}+\norm{\cH_0\big\vert_{\varrho=\rho_0}}_{H^s_\bx}+\kappa^{1/2}\Norm{h_0}_{H^{s,k}}
	\le M;
	\]
	and the stable stratification assumption
	\[ \forall (\bx,\varrho)\in  \Omega , \qquad h_\star \leq  \bar h(\varrho)+h_0(\bx,\varrho) \leq h^\star , \]
	there exists a unique $(h^\h,\bu^\h)\in   \cC^0([0,T];H^{s,k}(\Omega)^{1+d})$ solution to~\eqref{eq:hydro-iso-intro} and $(h^\h,\bu^\h)\big\vert_{t=0}=(h_0,\bu_0)$, where
	\begin{equation}\label{def:time}
		T^{-1}= C\, \big(1+ \kappa^{-1} \big(\norm{\bar \bu'}_{L^2_\varrho}^2+M_0^2\big)  \big).
	\end{equation}
	Moreover, $h^\h\in  L^2(0,T;H^{s+1,k}(\Omega))$ and  one has, for any $t\in[0,T]$,
	\[ \forall (\bx,\varrho)\in  \Omega , \qquad h_\star/2 \leq  \bar h(\varrho)+h(t,\bx,\varrho) \leq 2\,h^\star , \]
	and, denoting $\cH^\h(\cdot,\varrho)=\int_{\varrho}^{\rho_1} h^\h(\cdot,\varrho')\dd\varrho'$,
	\begin{multline*}
		\Norm{\cH^\h(t,\cdot)}_{H^{s,k}}+\Norm{\bu^\h(t,\cdot)}_{H^{s,k}} +\norm{\cH^\h\big\vert_{\varrho=\rho_0}(t,\cdot)}_{H^s_\bx} +\kappa^{1/2}\Norm{h^\h(t,\cdot)}_{H^{s,k}} \\+ \kappa^{1/2} \Norm{\nabla_\bx \cH^\h}_{L^2(0,T;H^{s,k})} +  \kappa^{1/2} \norm{\nabla_\bx \cH^\h \big\vert_{\varrho=\rho_0}  }_{ L^2(0,T;H^s_\bx)} +\kappa \Norm{\nabla_\bx h^\h}_{L^2(0,T;H^{s,k})} \leq CM_0.
	\end{multline*}
\end{theorem}

In our second main result, we prove that within the timescale defined by~\eqref{def:time}, there exist solutions to the non-hydrostatic equations~\eqref{eq:nonhydro-iso-intro} for $\mu$ sufficiently small, and they converge towards the corresponding solutions of the hydrostatic equations, with the expected $\mathcal O(\mu)$ convergence rate.

\begin{theorem}\label{thm-convergence}
	There exists $p\in\NN$ such that for any  $ k=s \in \NN$, $\bar M,M,h_\star,h^\star>0$ and $0<\rho_0<\rho_1$,  there exists $C>0$ such that the following holds. For any $ 0< M_0\leq M$, $0<\kappa\leq 1$, and $\mu>0$ such that 
	\[\mu \leq \kappa/(C M_0^2 ),\] 
	for any for any $(\bar h, \bar \bu) \in W^{k+p,\infty}((\rho_0,\rho_1))^2 $ satisfying
	\[  \norm{\bar h}_{W^{k+p,\infty}_\varrho } + \norm{\bar \bu'}_{W^{k+p-1,\infty}_\varrho }\leq \bar M \,;\]
	for any initial data $(h_0, \bu_0,w_0)\in H^{s+p,k+p}(\Omega)^{2+d}$ satisfying the boundary condition $w_0|_{\varrho=\rho_1}=0$ and the incompressibility condition
	\[-(\bar h+h_0)\nabla_\bx\cdot \bu_0-(\nabla_\bx \cH_0)\cdot({\bar \bu}'+\partial_\varrho\bu_0)+\partial_\varrho  w_0=0,\] 
	(denoting $\cH_0(\cdot,\varrho)=\int_{\varrho}^{\rho_1} h_0(\cdot,\varrho')\dd\varrho'$), the inequality
	\[
	\Norm{\cH_0}_{H^{s+p,k+p}}+\Norm{\bu_0}_{H^{s+p,k+p}}+\norm{\cH_0\big\vert_{\varrho=\rho_0}}_{H^{s+p}_\bx}+\kappa^{1/2}\Norm{h_0}_{H^{s+p,k+p}}  =M_0\leq M
	\]
	and the stable stratification assumption
	\[ \forall (\bx,\varrho)\in  \Omega , \qquad h_\star \leq  \bar h(\varrho)+h_0(\bx,\varrho) \leq h^\star , \]
	the following holds. Denoting $(h^\h,\bu^\h)\in   \cC^0([0,T^\h];H^{s+p,k+p}(\Omega)^{1+d})$ the solution to the hydrostatic equations~\eqref{eq:hydro-iso-intro} with initial data $(h^\h,\bu^\h)\big\vert_{t=0}=(h_0,\bu_0)$ provided by Theorem~\ref{thm-well-posedness},
	there exists a unique strong solution $(h^\nh,\bu^\nh,w^\nh)\in  \cC([0,T^\h];H^{s,k}(\Omega)^{2+d})$ to the non-hydrostatic equations~\eqref{eq:nonhydro-iso-intro} with initial data $(h^\nh,\bu^\nh, w^\nh)\big\vert_{t=0}=(h_0,\bu_0, w_0)$.
	Moreover, one has
	\[	\Norm{h^\nh-h^\h}_{L^\infty(0,T^\h;H^{s,k})}+\Norm{\bu^\nh-\bu^\h}_{L^\infty(0,T^\h;H^{s,k})}\leq C\,\mu.\]
\end{theorem}

\subsection{Strategy of the proofs} \label{S.strategy}
Our results rely mainly on the energy method, exhibiting the structure of the systems of equations trough well-chosen energy functionals and making use of product, commutator and composition estimates in the $L^2$-based Sobolev spaces $H^{s,k}(\Omega)$ that are summarized in the Appendix.

	Based on the structure of hydrostatic equations without thickness diffusivity ($\kappa=0$)~\eqref{eq:hydro-iso-intro}, namely
	\begin{equation*}
		\begin{aligned}
			\partial_t  h+\nabla_\bx\cdot((\bar h+ h)({\bar \bu}+\bu))&=0,\\
			\varrho\Big( \partial_t\bu+\big(({\bar \bu}+\bu)\cdot\nabla_\bx \big)\bu\Big)+ \mathcal M\nabla_\bx \psi&=0,
		\end{aligned}
	\end{equation*}
	where the operator $\mathcal M$ defining the Montgomery potential in \eqref{eq:nablaphi-intro} is self-adjoint for the $L^2_\varrho$ inner-product and satisfies for all $h_1,h_2\in L^2_\varrho:=L^2((\rho_0,\rho_1))$
	\[ \big(\mathcal M h_1\ ,\ h_2\big)_{L^2_\varrho}=\rho_0 \Big(\int_{\rho_0}^{\rho_1} h_1\dd\varrho\Big)\Big(\int_{\rho_0}^{\rho_1} h_2\dd\varrho\Big) + \Big(\int_{\cdot}^{\rho_1} h_1\dd\varrho\ ,\ \int_{\cdot}^{\rho_1} h_2\dd\varrho\Big)_{L^2_\varrho},\]
	the natural energy functional associated with the hydrostatic equations involves $\bu$, $\cH:=\int_{\cdot}^{\rho_1} h\dd\varrho$, and $\cH\vert_{\varrho=\rho_0}=\int_{\rho_0}^{\rho_1} h\dd\varrho$ (that represents the free surface deformation), and their derivatives. Notice we do {\em not} control $h=-\de_\varrho \cH$ (see \ref{def:eta}) in the same regularity class as $\bu,\eta$ unless it is multiplied by the prefactor $\kappa^{1/2}$. In our energy estimates, problematic contributions arise from the commutator between advection operator and density integration in the equation of mass conservation, {\em i.e.} the first equation in~\eqref{eq:hydro-iso-intro}. In order to control these contributions we crucially use the diffusivity-induced regularization, which explains why the lower bound on the time of existence of our solution stated in \eqref{def:time} vanishes as $\kappa\searrow 0$, yet with a prefactor involving the shear velocity, $\bar\bu'(\varrho)$ (since advection by a $\varrho$-independent velocity commutes with density integration). It is interesting to notice that the index of regularity with respect to the space variable, $s$, and the one with respect to the density variable, $k$, are decoupled (yet only in the hydrostatic framework). This is due to the fact that the isopycnal change of coordinate is semi-Lagrangian: the advection in isopycnal coordinates occurs only in the horizontal space directions. It would be of utmost interest (but outside of the scope of the present work) to decrease the regularity assumption with respect to the density variable, so as to admit discontinuities, representing density interfaces.

Concerning the non-hydrostatic system,~\eqref{eq:nonhydro-iso-intro}, the natural energy space involves additionally $\sqrt\mu w$ and its derivatives (hence the control vanishes as $\mu\searrow 0$). In order to obtain suitable energy estimates, we decompose the pressure as the sum of the hydrostatic contribution and the non-hydrostatic contribution, the latter being of lower order in terms of regularity and/or smallness with respect to $\mu\ll 1$. Then we use the structure of the hydrostatic equations, which we complement with an additional symmetric structure for the non-hydrostatic contributions. There, the difficulty consists in providing controls of the energy norms that are uniform with respect to the vanishing parameter $\mu\ll1$. Our estimates concerning the non-hydrostatic contribution of the pressure stem from elliptic estimates on a boundary-value problem. This strategy is hea\-vily inspired by the work of Desjardins, Lannes and Saut, and it is interesting to compare our results with the analogous ``large-time'' well-posedness result in~\cite{Desjardins-Lannes-Saut}*{Theorem~2}. 
Firstly, due to the choice of isopycnal coordinates, our boundary-value problem is no longer an anisotropic Poisson equation but involves a fully nonlinear elliptic operator. Since this operator involves $h$ that is not controlled in the energy space, we use again the diffusivity-induced regularization at this stage. On the positive side, using isopycnal coordinates rather than  Eulerian coordinates allows us to consider the free-surface framework (since isopycnal coordinates readily set the domain as a flat strip, thanks to our assumption that the density is constant at the surface and at the bottom) rather than the rigid-lid setting.  We believe that our study can be extended to the rigid-lid framework with small adjustments. Incidentally, we do not employ the often-used Boussinesq approximation, since it is not useful in our context. 
Additionally, we do not rely on the use of strong boundary conditions on the initial density and velocities and their derivatives at the surface and the bottom, which instead are assumed in~\cite{Desjardins-Lannes-Saut} (and in most of the other works, often put in a periodic framework) and rather use only the natural no-slip boundary condition at the bottom; the former allow to cancel the trace contributions resulting from vertical integration by parts. We also consider the general situation where the velocity field is a perturbation of a non-zero background current, $\bar\bu$. In turn, the price to pay to handle this general framework manifests in terms of some restrictions on the length of the time of existence of our solutions, which is inversely proportional with respect to the size of the fluctuations in~\cite{Desjardins-Lannes-Saut}*{Theorem~2}.

Now we describe our strategy to prove the convergence of the solutions to the non-hydrostatic system towards the hydrostatic one.
First, we point out that a direct use of energy estimates as previously allows to obtain the existence of solutions of the non-hydrostatic equations in a timescale uniform with respect to $\mu$ but not necessarily the same as the existence time of the corresponding solution to the hydrostatic equations, and (for technical reasons) restricted to sufficiently small data. To overcome these issues, we look at the hydrostatic solution as an approximate solution to the non-hydrostatic system (in the sense of consistency, as it approximately solves the non-hydrostatic equations), and deduce, using the aforementioned structure of the non-hydrostatic equations, an energy inequality that controls the difference between the solution to the non-hydrostatic system and the (respective) solution to the hydrostatic one. This stratey allows to bootstrap the control of the difference of the two solutions (and hence the control of the solution to the non-hydrostatic equations) within the timescale of existence (in a higher-regularity space) of the hydrostatic solution, provided that the parameter $\mu$ is sufficiently small. However, the rate of convergence obtained by this method is not optimal. Therefore, in a second step we implement another strategy to obtain the expected (optimal) convergence rate. It simply consists in taking the opposite viewpoint: to look at the solution to the non-hydrostatic equations as an approximate solution to the hydrostatic equations (again in the sense of consistency) and
use the structure of the hydrostatic equations to infer the $\mathcal{O}(\mu)$ convergence rate. Both steps involve loss of derivatives, described by the parameter $p$ in Theorem~\ref{thm-convergence}.

\subsection{Notation and conventions}
We highlight the following conventions used throughout the paper.
\begin{itemize}
	\item $\rho_0$ and $\rho_1$ are fixed constants such that $0<\rho_0<\rho_1$, and the dependency with respect to these constants is never explicitly displayed.
	\item For $k,s \in \NN$ and $k \leq s$, and $\Omega=\RR^d\times(\rho_0,\rho_1)$, we define the functional space
	\begin{align}\label{def:space}
		H^{s,k}(\Omega)=\left\{ f \ : \ \forall (\balpha,j)\in \NN^{d+1}, \ |\balpha|+ j\leq s,\ j\leq k,\  \partial_\bx^\balpha\partial_\varrho^j f  \in L^2(\Omega) \right\},
	\end{align}
	endowed with the topology of the norm
	\begin{equation}\label{def:Hsk}
		\Norm{f}_{H^{s,k}}^2:=\sum_{j=0}^k \sum_{|\balpha|=0}^{s-j} \Norm{\partial_\bx^\balpha\partial_\varrho^j f}_{L^2(\Omega)}^2 .
	\end{equation}
	When $s'\in\RR$ (and $k\in\NN$) we define $H^{s',k}(\Omega)=\left\{ f \ : \ \forall j\in \NN,\ j\leq k,\  \Lambda^{s'-j}\partial_\varrho^j f  \in L^2(\Omega) \right\}$ and
	\[
	\Norm{f}_{H^{s',k}}^2:=\sum_{j=0}^k  \Norm{\Lambda^{s'-j}\partial_\varrho^j f}_{L^2(\Omega)}^2.
	\]
	where $\Lambda=(\Id-\Delta_\bx)^{1/2}$. Of course the two notations are consistent when $s'=s \in \NN$, up to harmless factors in the definition of the norm.
	\item We use both the equivalent notations $H^s(\RR^d)=H^s_\bx$ (the usual $L^2$-based Sobolev space on $\RR^d$) and $W^{k,\infty}(\RR^d)=W^{k,\infty}_\bx$ (the $L^\infty$-based Sobolev space on $\RR^d$), and similarly $L^2((\rho_0,\rho_1))=L^2_\varrho$ and $W^{k,\infty}((\rho_0,\rho_1))=W^{k,\infty}_\varrho$. For functions with variables in $\Omega$ we denote for instance 
	\[L^2_\varrho L^\infty_\bx=L^2(\rho_0,\rho_1;L^\infty(\RR^d))=\{f \ : \ \esssup_{\bx\in\RR^d} |f(\cdot,\bx)|\in L^2((\rho_0,\rho_1))\}.\]
	Notice $L^2_\varrho L^2_\bx=L^2_\bx  L^2_\varrho=L^2(\Omega)$ and $L^\infty_\varrho L^\infty_\bx=L^\infty_\bx  L^\infty_\varrho=L^\infty(\Omega)$. We use similar notations for functions also depending on time. For instance, for $k\in\NN$, and $X$ a Banach space as above, $\cC^k([0,T];X)$ is the space of functions with values in $X$ which are continuously differentiable up to order $k$, and $L^p(0,T;X)$ the $p$-integrable $X$-valued functions. All these spaces are endowed with their natural norms. 
	\item For any operator $A: f \rightarrow Af$, we denote by $[A,f]g=A(fg)-f(Ag)$ the usual commutator, while $[A;f,g]=A(fg)-f(Ag)-g(Af)$ is the symmetric commutator,
	\item $C(\lambda_1,\lambda_2,\dots)$ denotes a constant which depends continuously on its parameters.
	\item For any $a, b \in \RR$,  we use the notation $a \lesssim b$ (resp. $ a \gtrsim b$) if there exists $C>0$, independent of relevant parameters, such that $a \le C b$ (resp. $a \ge C b$). We write $a\approx b$ if $a\lesssim b$ and $a\gtrsim b$.
	\item We put $a\vee b:=\max(a,b)$. Finally,
	\[ \big\langle B_a\big\rangle_{a> b} =\begin{cases}
		0 &\text{ if } a\leq b\,,\\
		B_a& \text{otherwise,}
	\end{cases} \quad \text{ and } \quad \big\langle B_a\big\rangle_{a= b} =\begin{cases}
		0 &\text{ if } a\neq  b\,,\\
		B_a& \text{otherwise.}
	\end{cases} \]
\end{itemize}

\section{The hydrostatic system}\label{S.Hydro}
\interfootnotelinepenalty=100000000

In this section we study the hydrostatic system in isopycnal coordinates. Specifically, we provide in this section a well-posedness result on the initial-value problem, namely Theorem~\ref{thm-well-posedness}. The result follows from careful {\em a priori} energy estimates, and the standard method of parabolic regularization. Therefore we will first study the system%
\begin{equation}\label{eq:hydro-iso-nu}
	\begin{aligned}
		\partial_t  h+\nabla_\bx\cdot((\bar h+ h)({\bar \bu}+\bu))&=\kappa\Delta_\bx  h,\\[1ex]
		\partial_t\bu+\big(({\bar \bu}+\bu-\kappa\tfrac{\nabla_\bx  h}{\bar h+ h})\cdot\nabla_\bx \big)\bu+ \tfrac{1}{\varrho}\nabla_\bx \psi&=\nu \Delta_\bx \bu,
	\end{aligned}
\end{equation}
with 
\[
\nabla_\bx\psi(t,\bx,\varrho):= \rho_0\int_{\rho_0}^{\rho_1} \nabla_\bx h(t,\bx,\varrho')\dd\varrho'   +\int_{\rho_0}^\varrho \int_{\rho'}^{\rho_1} \nabla_\bx  h(t,\bx,\varrho'')\dd\varrho''\dd\varrho',
\]
and $\nu>0$, and will rigorously establish the limit $\nu \rightarrow 0$.%
\begin{samepage}
	\footnote{As highlighted by the anonymous referee, another reasonable approach, inspired by the specific form of the viscosity contributions in the shallow water equations advocated in \cite{Gent93} and analyzed using the energy method in \cite{GustafssonSundstroem78}, would be to consider the following system
			\begin{equation}\label{eq:hydro-iso-nu-Gent}
				\begin{aligned}
					\partial_t  h+\nabla_\bx\cdot((\bar h+ h)({\bar \bu}+\bu))&=\kappa\Delta_\bx  h,\\[1ex]
					\partial_t\bu+\big(({\bar \bu}+\bu-\kappa\tfrac{\nabla_\bx  h}{\bar h+ h})\cdot\nabla_\bx \big)\bu+ \tfrac{1}{\varrho}\nabla_\bx \psi&=\nu \frac{\di_\bx((\bar h+h)\nabla_\bx \bu)}{\bar h+h}.
				\end{aligned}
			\end{equation}
			Notice that after straightforward computations the system \eqref{eq:hydro-iso-nu-Gent} can be reformulated as
			\begin{equation}\label{eq:hydro-iso-nu-Gent-2}
				\begin{aligned}
					\partial_t  h+\nabla_\bx\cdot((\bar h+ h)({\bar \bu}+\bu))&=\kappa\Delta_\bx  h,\\[1ex]
					\partial_t\bu+\big(({\bar \bu}+\bu-(\kappa+\nu)\tfrac{\nabla_\bx  h}{\bar h+ h})\cdot\nabla_\bx \big)\bu+ \tfrac{1}{\varrho}\nabla_\bx \psi&=\nu\Delta_\bx \bu,
				\end{aligned}
			\end{equation}so that the difference with respect to \eqref{eq:hydro-iso-nu} is minor, and our analysis can easily be extend to \eqref{eq:hydro-iso-nu-Gent-2}.
				Specifically, the only significant difference arises in the energy estimates of Lemma~\ref{lem:estimate-system}. Because the first contribution of (v) in the proof vanishes, the assumption on the control of $\nu^{1/2}\Norm{ \nabla_\bx h(t,\cdot) }_{L^\infty(\Omega)}$ is no longer necessary.
				As a consequence, Proposition \ref{P.regularized-large-time-WP} holds without assuming $\nu\leq \kappa$. This leads to the the interesting problem of the large-time well-posedness of the initial-value problem for \eqref{eq:hydro-iso-nu-Gent-2} when $\kappa=0$, $\nu>0$, which we believe can be addressed in the spirit of the BD and $\kappa$ entropies described in Section \ref{S.two-velocities}, but leave open for future studies.
	}
\end{samepage}

\subsection{Well-posedness of the regularized hydrostatic system}

We start with proving the well-posedness of the initial value problem.

\begin{proposition}\label{P.WP-nu}
	Let $s>\frac32 + \frac d 2$, $k\in\NN$ with $1\leq k\leq s$, and $\bar M$, $M_0 $, $h_\star$, $\nu$, $\kappa>0$  and $C_0>1$. Then there exists $T=T(s,k, \bar M,M_0, h_\star, \nu, \kappa,C_0)$ such that 
	for any $(\bar h,\bar\bu) =(\bar h(\varrho),\bar\bu(\varrho)) \in W^{k,\infty}((\rho_0, \rho_1))$ and for any $(h_0, \bu_0)=(h_0(\bx, \varrho), \bu_0(\bx, \varrho)) \in H^{s,k}(\Omega)$ such that
	\[ \inf_{(\bx,\varrho)\in\Omega}(\bar h(\varrho) + h_0(\bx,\varrho))\geq h_\star , \quad \norm{(\bar h,\bar\bu)}_{W^{k,\infty}_\varrho}\leq \bar M, \quad  \Norm{(h_0,\bu_0)}_{H^{s,k}} \leq M_0,\]
	there exists a unique solution $(h, \bu)\in \cC^0([0,T];H^{s,k}(\Omega))$ to system~\eqref{eq:hydro-iso-nu} with $(h, \bu)\big\vert_{t=0}=(h_0,\bu_0)$. Moreover, $(h, \bu)\in L^2(0,T;H^{s+1,k}(\Omega))$ and, for a universal constant $c_0>0$, the following estimates hold
	\begin{equation}\label{eq:bound-parabolic}
		\begin{cases}
			\Norm{h}_{L^\infty(0,T;H^{s,k})}+c_0\kappa^{1/2} \Norm{\nabla_\bx h}_{L^2(0,T;H^{s,k})}  < C_0 \Norm{h_0}_{H^{s,k}};\\
			\Norm{\bu}_{L^\infty(0,T;H^{s,k})}+c_0\nu^{1/2} \Norm{\nabla_\bx \bu }_{L^2(0,T;H^{s,k})}  < C_0 \Norm{\bu_0}_{H^{s,k}};\\
			\inf_{(t,\bx,\varrho)\in(0,T)\times \Omega}(\bar h(\varrho)+h(t,\bx,\varrho))> h_\star/C_0.
		\end{cases}
	\end{equation}
\end{proposition}

\begin{proof}
	We will construct the solution as the  fixed point of the Duhamel formula
	\begin{align*} 
		h(t,\cdot) &= e^{\kappa t \Delta_\bx } h_0+\int_0^t e^{\kappa(t-\tau)\Delta_\bx } f(h(\tau,\cdot),\bu(\tau,\cdot)) \dd\tau , \\
		\bu(t,\cdot) &= e^{\nu t \Delta_\bx } \bu_0+\int_0^t e^{\nu(t-\tau)\Delta_\bx } (\bm{f}_1+\bm{f}_2)(h(\tau,\cdot),\bu(\tau,\cdot)) \dd\tau 
	\end{align*}
	where $e^{\alpha t \Delta_\bx } $ with $\alpha >0$ is the heat semigroup defined by $\mathcal F[e^{\alpha t \Delta_\bx } f](\bm{\xi})=e^{-\alpha t |\bm{\xi}|^2} \mathcal F[ f](\bm{\xi})$
	where $\mathcal F$ is the Fourier transform with respect to the variable $\bx$, and
	\begin{align*} f(h,\bu)&= -\nabla_\bx\cdot((\bar h+ h)({\bar \bu}+\bu)),\\
		\bm{f}_1(h,\bu)&=  -\big(({\bar \bu}+\bu-\kappa\tfrac{\nabla_\bx  h}{\bar h+ h})\cdot\nabla_\bx\big) \bu\\
		\bm{f}_2(h,\bu)&= - \frac1\varrho\left(\rho_0\int_{\rho_0}^{\rho_1} \nabla_\bx h(\cdot,\varrho')\dd\varrho'   +\int_{\rho_0}^\varrho \int_{\rho'}^{\rho_1} \nabla_\bx  h(\cdot,\varrho'')\dd\varrho''\dd\varrho'\right).
	\end{align*}
	Let us first recall the standard regularization properties of the heat flow. For any $\nu>0$, $T>0$ and for any $u_0\in H^{s,k}(\Omega)$ and $ g\in L^1(0,T;H^{s,k}(\Omega))$, there exists a unique $u \in \cC^0([0,T];H^{s,k}(\Omega)) \cap L^2(0,T;H^{s+1,k}(\Omega))$ solution to $\partial_t u-\nu\Delta_\bx u=g$ with $u(0,\cdot)=u_0$ which reads by definition
	\[u= e^{\nu t \Delta_\bx } u_0+\int_0^t e^{\nu(t-\tau)\Delta_\bx }  g(\tau,\cdot)\dd\tau,\]
	and we have 
	\[\Norm{u}_{L^\infty(0,T;H^{s,k})}+c_0\nu^{1/2} \Norm{\nabla_\bx u}_{L^2(0,T;H^{s,k})} \leq  \Norm{u_0}_{H^{s,k}} + \Norm{g}_{L^1(0,T;H^{s,k})},\]
	where $c_0>0$ is a universal constant.
	The existence and uniqueness of the solution as well as the above estimate easily follow from solving the equation (for almost every $\varrho\in(\rho_0,\rho_1)$) in Fourier space and using Plancherel's formula, then using that $\partial_\varrho$ commutes with $\partial_t$ and $\Delta_\bx$, and invoking Minkowski's integral inequality (resp. Fubini's theorem) to exchange the order of integration in the variables $(t,\varrho)$ (resp. $(\bx,\varrho)$). We also remark that, by the positivity of the heat kernel and the continuous embedding $H^{s-1,1}(\Omega)\subset L^\infty(\Omega)$ for $s>\frac32+\frac{d}{2}$  (see Lemma~\ref{L.embedding}),
	\[\inf_{\Omega} u \geq \inf_\Omega u_0 - \Norm{g}_{L^1(0,T;H^{s-1,1})}.\]
	
	Now we consider $(f,\bm{f}_1+\bm{f}_2)$ as a bounded operator from $H^{s+1,k}(\Omega)^{1+d}$ to $H^{s,k}(\Omega)^{1+d}$. Indeed, there exists $C_{s,k}>0$ such that for any $(h,\bu)\in H^{s+1,k}(\Omega)^{1+d}$, 
	\begin{align*}
		\Norm{f(h,\bu)}_{H^{s,k}} &\leq \Norm{\nabla_\bx\cdot(\bar h \bu + h{\bar \bu}+h \bu)}_{H^{s,k}} \\
		&\leq C_{s,k}\times\left( \norm{\bar h}_{W^{k,\infty}_\varrho} \Norm{\bu}_{H^{s+1,k}}+\norm{\bar \bu}_{W^{k,\infty}_\varrho} \Norm{h}_{H^{s+1,k}}+  \Norm{ h}_{H^{s,k}} \Norm{ \bu}_{H^{s+1,k}}+  \Norm{ h}_{H^{s+1,k}} \Norm{ \bu}_{H^{s,k}}\right),
	\end{align*}
	where we used straightforward product estimates for the first two terms, and Lemma~\ref{L.product-Hsk} for the last ones. Similarly, we have
	\begin{align*}
		\Norm{\bm{f}_1(h,\bu)}_{H^{s,k}} &\leq \Norm{\big(({\bar \bu}+\bu-\kappa\tfrac{\nabla_\bx  h}{\bar h+ h})\cdot\nabla_\bx \big)\bu }_{H^{s,k}} \\
		&\leq C_{s,k}\left( \norm{\bar \bu}_{W^{k,\infty}_\varrho} +\Norm{\bu}_{H^{s,k}} \right) \Norm{\bu}_{H^{s+1,k}}\\
		&\quad + \kappa C_{s,k} \big( \norm{\bar h^{-1}}_{W^{k,\infty}_\varrho} + \Norm{\tfrac{h}{\bar h(\bar h+ h)}}_{H^{s,k}} \big) \Norm{(\nabla_\bx h\cdot\nabla_\bx) \bu }_{H^{s,k}}.
	\end{align*}
	Using the constraint $\inf_{(\rho_0,\rho_1)} \bar h\geq \inf_{\Omega}(\bar h + h_0)\geq h_\star>0$ and Lemma~\ref{L.composition-Hsk-ex}, we find that for any $h_\star>0$ and $M_0,\bar M\geq 0$ there exists $C_{s,k}(h_\star,\bar M,M_0,C_0)$ such that for any $h\in H^{s,k}(\Omega)$ bounded by $\Norm{h}_{H^{s,k}}\leq  C_0M_0$ and satisfying $\inf_{(\bx,\varrho)\in\Omega}(\bar h(\varrho) + h(\bx,\varrho))\geq h_\star/C_0$, one has
	\[ \norm{\bar h^{-1}}_{W^{k,\infty}_\varrho} + \Norm{\tfrac{h}{\bar h(\bar h+ h)}}_{H^{s,k}} \leq C_{s,k}(h_\star,\bar M,M_0,C_0) .\]
	Using the last estimates in Lemma~\ref{L.product-Hsk}, since  $s>\frac32+\frac{d}{2}$, we have
	\[\Norm{(\nabla_\bx h\cdot\nabla_\bx) \bu }_{H^{s,k}} \leq \Norm{h}_{H^{s,k}} \Norm{ \bu }_{H^{s+1,k}}+\Norm{ h}_{H^{s+1,k}} \Norm{ \bu }_{H^{s,k}}.\]
	Finally, from the continuous embedding $L^\infty((\rho_0,\rho_1))\subset   L^2((\rho_0,\rho_1)) \subset L^1((\rho_0,\rho_1))$ we immediately infer
	\[\Norm{\bm{f}_2(h,\bu)}_{H^{s,k}}  \leq C_{s,k}\Norm{ h}_{H^{s+1,k}} .\]
	Altogether, we find that  for any $h_\star,C_0>0$ and $\bar M,M_0\geq 0$ there exists $C_{s,k}(h_\star,\bar M,M_0,C_0)$ such that for any $(h,\bu)\in H^{s+1,k}(\Omega)^{1+d}$ satisfying $\Norm{(h,\bu)}_{H^{s,k}}\leq  C_0M_0$ and $\inf_{(\bx,\varrho)\in\Omega}(\bar h(\varrho) + h(\bx,\varrho))\geq h_\star/C_0$, we have
	\[\Norm{\big( f(h,\bu),\bm{f}_1(h,\bu),\bm{f}_2(h,\bu)\big)}_{H^{s,k}}  \leq C_{s,k}(h_\star,\bar M,M_0,C_0)\, (1+\kappa)\, \Norm{ (h,\bu)}_{H^{s+1,k}} .\]
	
	By similar considerations, we find that for any $h_\star,C_0>0$ and $\bar M,M_0\geq 0$ there exists $C_{s,k}(h_\star,\bar M,M_0,C_0)$ such that for any $(h_1,\bu_1,h_2,\bu_2)\in H^{s+1,k}(\Omega)^{2(1+d)}$ satisfying the bound $\Norm{(h_i,\bu_i)}_{H^{s,k}}\leq C_0M_0$ as well as $\inf_{(\bx,\varrho)\in\Omega}(\bar h(\varrho) + h_i(\bx,\varrho))\geq h_\star/C_0$ (with $i\in\{1,2\}$),
	one has
	\begin{multline*}\Norm{ \big( f(h_2,\bu_2)-f(h_1,\bu_1), \bm f_1(h_2,\bu_2)- \bm f_1(h_1,\bu_1) , \bm f_2(h_2,\bu_2)-\bm f_2(h_1,\bu_1) \big)}_{H^{s,k}} \\
		\leq C_{s,k}(h_\star,\bar M,M_0,C_0) (1+\kappa)\Big( \Norm{ ( h_2-h_1, \bu_2-\bu_1)}_{H^{s+1,k}}\\
		+   \Norm{(h_1,h_2,\bu_1,\bu_2)}_{H^{s+1,k}} \Norm{ ( h_2-h_1, \bu_2-\bu_1)}_{H^{s,k}}    \Big).
	\end{multline*}
	
	From the above estimates, we easily infer that for $T>0$ sufficiently small (uniquely depending on $s,k,M_0,\bar M,h_\star,\nu,\kappa,C_0$),
	\[ \mathcal T:\begin{pmatrix}h\\ \bu
	\end{pmatrix}\mapsto \begin{pmatrix}  e^{\kappa t \Delta_\bx } h_0+\int_0^t e^{\kappa(t-\tau)\Delta_\bx } f(h(\tau,\cdot),\bu(\tau,\cdot)) \dd\tau \\  e^{\nu t \Delta_\bx } \bu_0+\int_0^t e^{\nu(t-\tau)\Delta_\bx } (\bm{f}_1+\bm{f}_2)(h(\tau,\cdot),\bu(\tau,\cdot)) \dd\tau 
	\end{pmatrix}
	\]
	is a contraction mapping on 
	\[ X=\left\{ (h,\bu) \in  \cC^0([0,T];H^{s,k}(\Omega)) \cap L^2(0,T;H^{s+1,k}(\Omega)) \ : \  \text{\eqref{eq:bound-parabolic} holds}\right\}.\]
	The Banach fixed point theorem provides the existence and uniqueness of a fixed point (solution to~\eqref{eq:hydro-iso-nu}) in $X$, and uniqueness in $\cC^0([0,T];H^{s,k}(\Omega))$ is easily checked (for instance by the energy method).
\end{proof}
\begin{remark}
	It should be emphasized that the time of existence provided by Proposition~\ref{P.WP-nu} is not uniform with respect to the parameters $\kappa,\nu>0$. More precisely, the proof provides a lower bound as 
	\[T\gtrsim   \min(\{1,\kappa,\nu\}) , \qquad \text{\em i.e.} \quad  T^{-1}\lesssim 1+\kappa^{-1}+\nu^{-1}.\]
\end{remark}

\subsection{Quasilinearization}\label{S.Hydro:quasi}

In the result below, we apply spatial derivatives to system~\eqref{eq:hydro-iso-nu} and rewrite it in such a way that the linearized equations satisfied by the highest-order terms exhibit a skew-symmetric structure, which will allow us to obtain improved energy estimates in the subsequent section.

\begin{lemma}\label{lem:quasilinearization}
	Let  $s, k \in \NN$ such that $s>2+\frac d 2$ and $2\leq k\leq s$, and $\bar M,M,h_\star>0$. Then there exists $C=C(s,k,\bar M,M,h_\star)>0$ such that for any $\kappa\in[0,1]$, $\nu\geq0$, for any
	$(\bar h, \bar \bu) \in W^{k,\infty}((\rho_0,\rho_1))$ such that 
	\[  \norm{\bar h}_{W^{k,\infty}_\varrho } + \norm{\bar \bu'}_{W^{k-1,\infty}_\varrho }\leq \bar M \,;\]
	and any $(h, \bu)  \in L^\infty(0,T;H^{s,k}(\Omega))$ solution to~\eqref{eq:hydro-iso-nu} with some $T>0$ and satisfying  for almost every $t\in[0,T]$
	\[
	\Norm{h(t,\cdot)}_{H^{s-1, k-1}} + \Norm{\cH(t,\cdot)}_{H^{s,k}}+\Norm{\bu(t,\cdot)}_{H^{s,k}}+\norm{\cH(t,\cdot)\big\vert_{\varrho=\rho_0}}_{H^s_\bx}+\kappa^{1/2}\Norm{h(t,\cdot)}_{H^{s,k}} \le M 
	\]
	(where $\cH(t,\bx,\varrho):=\int_{\varrho}^{\rho_1} h(t,\bx,\varrho')\dd\varrho'$) and 
	\[ \inf_{(\bx,\varrho)\in \Omega } \bar h(\varrho)+h(t, \bx,\varrho) \geq h_\star,\]
	the following holds. Denote, for any multi-index $\balpha \in \mathbb{N}^d$, $ \cH^{(\balpha)}=\de_\bx^\balpha  \cH, \, \bu^{(\balpha)}=\de_\bx^\balpha \bu$.
	\begin{itemize}
		\item For any $\balpha \in \mathbb{N}^d$ with $0 \le |\balpha| \le s$, we have that
		\begin{subequations}
			\begin{equation}\label{eq.quasilin}
				\begin{aligned}
					\partial_t  \cH^{(\balpha)}+ (\bar \bu + \bu) \cdot\nabla_\bx \cH^{(\balpha)} + \int_\varrho^{\rho_1} (\bar \bu' + \de_{\varrho}\bu) \cdot\nabla_\bx \cH^{(\balpha)} \, d{\varrho'} &\\
					+\int_{\varrho}^{\rho_1} (\bar h+h)\nabla_\bx \cdot\bu^{(\balpha)} \dd\varrho'&=\kappa\Delta_\bx   \cH^{(\balpha)}+R_{\balpha,0},\\
					\partial_t\bu^{(\balpha)}+\big(({\bar \bu}+\bu-\kappa\tfrac{\nabla_\bx h}{\bar h+ h})\cdot\nabla_\bx \big)\bu^{(\balpha)} \hspace{3.5cm} &\\
					+ \frac{\rho_0}{\varrho}\nabla_\bx \cH^{(\balpha)}\big\vert_{\varrho=\rho_0} +\frac1\varrho\int_{\rho_0}^\varrho  \nabla_\bx  \cH^{(\balpha)} \dd\varrho'&=\nu\Delta_\bx   \bu^{(\balpha)}+\bR_{\balpha,0} ,
				\end{aligned}
			\end{equation}
			where for almost every $t\in[0,T]$, $(R_{\balpha,0}(t,\cdot),\bR_{\balpha,0}(t,\cdot))\in \cC^0([\rho_0,\rho_1];L^2(\RR^d))\times L^2(\Omega)^d$ and 
			\begin{align} \label{eq.est-quasilin}
				\Norm{R_{\balpha,0}}_{L^2(\Omega) }+\Norm{ \bR_{\balpha,0}}_{L^2(\Omega)}  +\norm{ R_{\balpha,0}\big\vert_{\varrho=\rho_0}}_{L^2_\bx}  &\leq 
				C\, M \, \big(1 +\kappa\Norm{ \nabla_\bx h}_{H^{s,k}} \big).
			\end{align}
		\end{subequations}
		\item For any $j\in\NN$, $1\leq j\leq k$ 
		and any $\balpha \in \mathbb{N}^d$, $0 \le |\balpha| \le s-j$, it holds
		\begin{subequations}
			\begin{equation}\label{eq.quasilin-j}
				\begin{aligned}
					\partial_t  \partial_\varrho^{j} \cH^{(\balpha)}+(\bar\bu+\bu)\cdot \nabla_\bx \partial_\varrho^{j} \cH^{(\balpha)}
					&=\kappa\Delta_\bx  \partial_\varrho^{j} \cH^{(\balpha)}+R_{\balpha,j},\\
					\partial_t\partial_\varrho^j \bu^{(\balpha)}+\big(({\bar \bu}+\bu-\kappa\tfrac{\nabla_\bx h}{\bar h+ h})\cdot\nabla_\bx \big)\partial_\varrho^j \bu^{(\balpha)}
					&=\nu \Delta_\bx  \partial_\varrho^j  \bu^{(\balpha)}+\bR_{\balpha,j},
				\end{aligned}
			\end{equation}
			where for almost every $t\in[0,T]$, $(R_{\balpha,j}(t,\cdot),\bR_{\balpha,j}(t,\cdot))\in L^2(\Omega)\times L^2(\Omega)^{d}$ and
			\begin{equation}\label{eq.est-quasilin-j}
				\Norm{R_{\balpha,j}}_{L^2(\Omega) }+\Norm{ \bR_{\balpha,j}}_{L^2(\Omega)}  \leq 
				C\, M \, \big(1  +\kappa\Norm{ \nabla_\bx  h}_{H^{s,k}} \big).
			\end{equation}
		\end{subequations}
		\item For any $j\in\NN$, $0\leq j\leq k$ and any multi-index $\balpha \in \mathbb{N}^d$, $0 \le |\balpha| \le s-j$, it holds
		\begin{subequations}
			\begin{equation}\label{eq.quasilin-j-h}
				\begin{aligned}
					\partial_t  \partial_\varrho^{j} h^{(\balpha)}+(\bar\bu+\bu)\cdot \nabla_\bx \partial_\varrho^{j} h^{(\balpha)}
					&=\kappa\Delta_\bx  \partial_\varrho^{j} h^{(\balpha)}+r_{\balpha,j}+\nabla_{\bx} \cdot \br_{\balpha,j},
				\end{aligned}
			\end{equation}
			where for almost every $t\in[0,T]$, $(r_{\balpha,j}(t,\cdot),\br_{\balpha,j}(t,\cdot))\in L^2(\Omega)^{1+d}$ and
			\begin{equation}\label{eq.est-quasilin-j-h}
				\kappa^{1/2} \Norm{r_{\balpha,j}}_{L^2(\Omega) }  + \Norm{\br_{\balpha,j}}_{L^2(\Omega) }  \leq C\,M.
			\end{equation}
		\end{subequations}
	\end{itemize}
\end{lemma}
\begin{proof} In this proof, we denote $s_0=s-2>\frac d 2$.
	
	\noindent {\em Estimate of $R_{\balpha,0}$.}  First we notice the identity by integration by parts in $\varrho$,
	\[ (\bar \bu + \bu) \cdot\nabla_\bx \cH^{(\balpha)} + \int_\varrho^{\rho_1} (\bar \bu' + \de_{\varrho}\bu) \cdot\nabla_\bx \cH^{(\balpha)} \, d{\varrho'}=\int_{\varrho}^{\rho_1}(\bar\bu+\bu)\cdot \nabla_\bx h^{(\balpha)}\dd\varrho'.\]
	Hence, recalling the notation $[P;u,v]=P(uv)-u(Pv)-v(Pu)$ and integrating by parts in $\varrho$, we get
	\begin{align*}
		R_{\balpha,0}&:=-\int_\varrho^{\rho_1} [\de_\bx^\balpha, \bu ]\cdot \nabla_\bx h  + [\de_\bx^\balpha, h] \nabla_\bx \cdot \bu \, \dd\varrho' \\
		&=-\int_\varrho^{\rho_1}  [\de_\bx^\balpha; \bu , \nabla_\bx h ]+ \bu^{(\balpha)} \cdot \nabla_\bx h + [\de_\bx^\balpha; h,\nabla_\bx \cdot \bu]+   h^{(\balpha)} (\nabla_\bx \cdot \bu)  \, \dd\varrho'\\
		&=-[\de_\bx^\balpha; \bu , \nabla_\bx \cH]   -  \cH^{(\balpha)} \nabla_\bx \cdot  \bu  \\
		&\qquad-\int_\varrho^{\rho_1} [\de_\bx^\balpha; \de_\varrho\bu , \nabla_\bx \cH ] + \bu^{(\balpha)} \cdot \nabla_\bx h   + [\de_\bx^\balpha; h,\nabla_\bx \cdot \bu]+   \cH^{(\balpha)} \nabla_\bx \cdot \de_\varrho \bu  \, \dd\varrho'.
	\end{align*}
	By the standard Sobolev embedding $H^{s_0}(\RR^d)\subset L^\infty(\RR^d)$ and  Lemma~\ref{L.embedding}, one gets
	\begin{align*} 
		\Norm{ \cH^{(\balpha)}  \nabla_\bx \cdot  \bu    }_{L^2(\Omega)} &\leq  
		\Norm{\cH^{(\balpha)} }_{ L^2(\Omega)} \Norm{  \nabla_\bx \cdot  \bu  }_{ L^\infty(\Omega)} 
		\lesssim  \Norm{\cH}_{ H^{s,0}}\Norm{ \bu }_{ H^{s_0+\frac 3 2,1}}.
	\end{align*}
	and
	\begin{align*} 
		\norm{ \big(\cH^{(\balpha)} \nabla_\bx \cdot  \bu \big)\big\vert_{\varrho=\rho_0}   }_{L^2_\bx} &\lesssim 
		\norm{\cH\big\vert_{\varrho=\rho_0} }_{  H^s_\bx } \Norm{ \bu }_{ H^{s_0+\frac32,1}}.
	\end{align*}
	By Lemma~\ref{L.commutator-Hs}(\ref{L.commutator-Hs-3}), and Lemma~\ref{L.embedding}, we have
	\begin{align*} 
		\Norm{ [\de_\bx^\balpha; \bu,\nabla_\bx \cH]  }_{L^2(\Omega)} &\lesssim 
		\Norm{\bu }_{ L^\infty_\varrho H^{s-1}_\bx} \Norm{\nabla_\bx \cH}_{ L^2_\varrho H^{s_0+1}_\bx} + \Norm{\bu }_{L^\infty_\varrho H^{s_0+1}_\bx } \Norm{\nabla_\bx \cH}_{ L^2_\varrho H^{s-1}_\bx}\\
		&\lesssim  \Norm{\bu }_{ H^{s-\frac 12,1}}  \Norm{\cH}_{ H^{s_0+2,0} } + \Norm{\bu }_{H^{s_0+\frac 3 2,1} } \Norm{\cH}_{H^{s,0}},
	\end{align*}
	\begin{align*} 
		\norm{ [\de_\bx^\balpha; \bu\big\vert_{\varrho=\rho_0},\nabla_\bx \cH\big\vert_{\varrho=\rho_0}   ]  }_{ L^2_\bx} &\lesssim 
		\norm{\bu\big\vert_{\varrho=\rho_0} }_{ H^{s-1}_\bx} \norm{\nabla_\bx \cH\big\vert_{\varrho=\rho_0} }_{  H^{s_0+1}_\bx} + \norm{\bu\big\vert_{\varrho=\rho_0} }_{ H^{s_0+1}_\bx } \norm{\nabla_\bx \cH\big\vert_{\varrho=\rho_0} }_{  H^{s-1}_\bx}\\
		&\lesssim \Norm{\bu }_{ H^{s-\frac 12,1}} \norm{ \cH \big\vert_{\varrho=\rho_0} }_{  H^{s_0+2}_\bx}+\Norm{\bu }_{H^{s_0+\frac 32,1} } \norm{ \cH \big\vert_{\varrho=\rho_0} }_{  H^{s}_\bx},
	\end{align*}
	and using additionally the Cauchy-Schwarz inequality,
	\begin{align*}
		\Norm{  [\de_\bx^\balpha; \de_\varrho \bu,\nabla_\bx \cH ]  }_{ L^1_\varrho L^2_\bx } &\lesssim \Norm{\de_\varrho \bu  }_{  L^2_\varrho H^{s-1}_\bx } \Norm{ \nabla_\bx \cH}_{ L^2_\varrho H^{s_0+1}_\bx }+\Norm{\de_\varrho \bu  }_{  L^2_\varrho H^{s_0+1}_\bx } \Norm{ \nabla_\bx \cH}_{ L^2_\varrho H^{s-1}_\bx }   \\
		&\lesssim \Norm{ \bu  }_{  H^{s,1} }\Norm{ \cH}_{  H^{s_0+2 ,0} }  + \Norm{ \bu  }_{   H^{s_0+2 ,1} }\Norm{ \cH}_{ H^{s,0} } ,
	\end{align*}
	\begin{align*}
		\Norm{ [\de_\bx^\balpha; h,\nabla_\bx \cdot \bu   ]  }_{ L^1_\varrho L^2_\bx }  &\lesssim\Norm{h }_{  L^2_\varrho H^{s-1}_\bx } \Norm{ \nabla_\bx \cdot\bu}_{ L^2_\varrho H^{s_0+1}_\bx } + \Norm{h  }_{  L^2_\varrho H^{s_0+1}_\bx } \Norm{ \nabla_\bx \cdot\bu}_{ L^2_\varrho H^{s-1}_\bx }  \\
		&\lesssim \Norm{ h }_{  H^{s-1,0} } {\Norm{  \bu }_{  H^{s_0+2,0} }}+ \Norm{h }_{  H^{s_0+1,0}} \Norm{  \bu }_{ H^{s,0} } ,
	\end{align*}
	and
	\begin{align*}
		\Norm{   \bu^{(\balpha)} \cdot \nabla_\bx h }_{L^1_\varrho L^2_\bx } &\lesssim \Norm{\bu}_{H^{s,0}}\Norm{h}_{  H^{s_0+1,0}},\\
		\Norm{   \cH^{(\balpha)} (\nabla_\bx \cdot \de_\varrho \bu) }_{L^1_\varrho L^2_\bx }  
		&\lesssim \Norm{\cH}_{H^{s,0}}\Norm{ \bu}_{   H^{s_0+2,1} }.
	\end{align*}
	
	Altogether, using  the continuous embedding $L^\infty((\rho_0,\rho_1))\subset   L^2((\rho_0,\rho_1)) \subset L^1((\rho_0,\rho_1))$, the Minkowski and triangle inequalities and $s\geq s_0+2$, we get
	\begin{equation}\label{eq:est-R0}
		\norm{R_{\balpha,0}\big\vert_{\varrho=\rho_0}}_{L^2_\bx} +\Norm{R_{\balpha,0}}_{L^2(\Omega)}    \lesssim 
		(\Norm{\cH}_{H^{s,0}}+\Norm{h}_{H^{s-1,0}}+\norm{ \cH \big\vert_{\varrho=\rho_0} }_{  H^{s}_\bx}) \Norm{\bu}_{H^{s, 1}}.
	\end{equation}
	
	\noindent {\em Estimate of $R_{\balpha,j}$ for $1\leq j\leq k$.} We have
	\begin{align*}
		R_{\balpha,j}&:= - [\de_\bx^{\balpha} \de_\varrho^{j-1} \nabla_\bx \cdot, \bar \bu + \bu]   h - \de_\bx^{\balpha} \de_\varrho^{j-1} \nabla_\bx \cdot (\bar h (\bar\bu+\bu)) \\
		&=-\sum_{i=1}^d[\de_\bx^{\balpha}\partial_{x_i}\de_\varrho^{j-1} , \bu_i]h-[\de_\varrho^{j-1}, \bar \bu]\cdot \de_\bx^\balpha \nabla_\bx h-  \de_\varrho^{j-1}  \cdot (\bar h \de_\bx^{\balpha}\nabla_\bx \bu),
	\end{align*}
	where $\bu_i$ is the $i^{\rm th}$ component of $\bu$.
	By Lemma~\ref{L.commutator-Hsk} and since $(|\balpha|+1)+(j-1)\leq s$ and $ j-1\leq k-1$, and $s\geq s_0+\frac32$, we find for $2\leq k-1\leq s$
	\[
	\Norm{[\de_\bx^{\balpha}\partial_{x_i}\de_\varrho^{j-1} ,\bu_i ] h}_{ L^2(\Omega) }  \lesssim  \Norm{h}_{H^{s-1,k-1}}\Norm{\bu}_{H^{s,k}}.
	\]
	There remains to consider $1\leq j\leq k\leq 2$. If $j=1$ we have by Lemma~\ref{L.commutator-Hs}(\ref{L.commutator-Hs-2}) and since $|\balpha|\leq s-1$ and $s\geq s_0+\frac32$
	\[\Norm{[\de_\bx^{\balpha}\partial_{x_i} ,\bu_i ] h}_{ L^2(\Omega) }  \lesssim \Norm{h  }_{  L^\infty_\varrho H^{s_0}_\bx } \Norm{ \bu }_{ L^2_\varrho H^{s}_\bx } +\Norm{h  }_{  L^2_\varrho H^{s-1}_\bx } \Norm{ \bu }_{ L^\infty_\varrho H^{s_0+1}_\bx }   \lesssim \Norm{h}_{H^{s-1,1}}\Norm{\bu}_{H^{s,1}}.\]
	If $j=k=2$, and since $|\balpha|\leq s-2$ and $s\geq s_0+\frac32$,
	\begin{align*}
		\Norm{[\de_\bx^{\balpha}\partial_{x_i}\de_\varrho^{j-1} ,\bu_i ] h}_{ L^2(\Omega) } &\leq \Norm{[\de_\bx^{\balpha}\partial_{x_i},\bu_i]  \partial_\varrho h}_{ L^2(\Omega) } + \Norm{\partial_\bx^{\balpha}\partial_{x_i}  ( h\partial_\varrho\bu_i)}_{ L^2(\Omega) }  \\
		&\lesssim \Norm{ \partial_\varrho h  }_{  L^2_\varrho H^{s_0}_\bx } \Norm{ \bu }_{ L^\infty_\varrho H^{s-1}_\bx } +\Norm{\partial_\varrho h  }_{  L^2_\varrho H^{s-2}_\bx } \Norm{ \bu }_{ L^\infty_\varrho H^{s_0+1}_\bx } \\
		&\qquad + \Norm{h  }_{  L^\infty_\varrho H^{s_0}_\bx } \Norm{\partial_\varrho \bu }_{ L^2_\varrho H^{s-1}_\bx } +\Norm{h  }_{  L^2_\varrho H^{s-1}_\bx } \Norm{ \partial_\varrho\bu }_{ L^\infty_\varrho H^{s_0}_\bx }  \\
		& \lesssim \Norm{h}_{H^{s-1,1}}\Norm{\bu}_{H^{s,2}}.
	\end{align*}
	Finally, we have immediately
	\begin{align*}
		\Norm{[\de_\varrho^{j-1}, \bar \bu]\cdot\de_\bx^\balpha \nabla_\bx h }_{ L^2(\Omega) } & \lesssim  \norm{\bar \bu'}_{W^{j-2,\infty}_\varrho }\Norm{h}_{H^{s-1,j-2}},\\
		\Norm{\de_\varrho^{j-1}   (\bar h \de_\bx^{\balpha}\nabla_\bx \cdot\bu) }_{ L^2(\Omega) } & \lesssim  \norm{\bar h}_{W^{j-1,\infty}_\varrho }\Norm{\bu}_{H^{s,j}}.
	\end{align*}
	Altogether, we find that for any $1\leq j\leq k$
	\begin{equation}\label{eq:est-Rj}
		\Norm{R_{\balpha,j}}_{L^2(\Omega)}    \lesssim 
		( \norm{\bar h}_{W^{k-1,\infty}_\varrho} +  \norm{\bar \bu'}_{W^{k-2,\infty}_\varrho} +\Norm{h}_{H^{s-1,k-1}}) \big( \Norm{\bu}_{H^{s, k}}+\Norm{h}_{H^{s-1,k-2}}\big).
	\end{equation}
	
	\noindent {\em Estimate of $\br_{\balpha,j}$ and $r_{\balpha,j}$ for  $0\leq j\leq k$.} We have~\eqref{eq.quasilin-j-h} with
	\begin{align*}
		\br_{\balpha,j}&:= - [\de_\varrho^{j} , \bar \bu ]  \de_\bx^{\balpha}  h - \de_\bx^{\balpha} \de_\varrho^{j}  (\bar h \bu)- (\partial_\bx^\balpha\de_\varrho^{j}\bu) h, \\
		r_{\balpha,j}&:=-[\partial^\balpha\de_\varrho^{j}\nabla_\bx\cdot;\bu,h] +(\partial_\bx^\balpha\de_\varrho^{j}\bu)\cdot\nabla_\bx h  .
	\end{align*}
	We have immediately (since $|\balpha|+j\leq s$, $ j\leq k$, and using Lemma~\ref{L.embedding})
	\begin{align*}
		\Norm{[\de_\varrho^{j}, \bar \bu]\de_\bx^\balpha  h }_{ L^2(\Omega) } & \lesssim  \norm{\bar \bu'}_{W^{k-1,\infty}_\varrho }\Norm{h}_{H^{s-1,k-1}},\\
		\Norm{\partial_\bx^\balpha\de_\varrho^{j}   (\bar h \bu) }_{ L^2(\Omega) } & \lesssim  \norm{\bar h}_{W^{k,\infty}_\varrho }\Norm{\bu}_{H^{s,k}},\\
		\Norm{ (\partial_\bx^\balpha\de_\varrho^{j}\bu) h}_{L^2(\Omega)}  & \lesssim \Norm{\bu}_{H^{s,k}} \Norm{h}_{H^{s_0+\frac12,1}}.
	\end{align*}
	By Lemma~\ref{L.commutator-Hsk-sym} and since $|\balpha|+j+1\leq s+1$, $ j\leq k\leq s$, $s+1\geq s_0+\frac52$, we find for $2\leq k\leq s$
	\[\Norm{[\partial^\balpha\de_\varrho^{j}\nabla_\bx\cdot;\bu,h] }_{ L^2(\Omega) } \lesssim \Norm{h}_{H^{s,k}}\Norm{\bu}_{H^{s,k}},\]
	and we have by Lemma~\ref{L.embedding}
	\[\Norm{ (\partial_\bx^\balpha\de_\varrho^{j}\bu)\cdot\nabla_\bx h}_{L^2(\Omega)}   \lesssim \Norm{\bu}_{H^{s,k}} \Norm{h}_{H^{s_0+\frac32,1}}.\]
	Altogether, we find that for any $0\leq j\leq k$
	\begin{align}\label{eq:est-brj}
		\Norm{\br_{\balpha,j}}_{L^2(\Omega)}   &\lesssim 
		( \norm{\bar h}_{W^{k,\infty}_\varrho} +  \norm{\bar \bu'}_{W^{k-1,\infty}_\varrho} +\Norm{h}_{H^{s-1,k-1}}) \big( \Norm{\bu}_{H^{s, k}}+\Norm{h}_{H^{s-1,k-1}}\big),\\
		\label{eq:est-rj}
		\Norm{r_{\balpha,j}}_{L^2(\Omega)}   & \lesssim 
		\Norm{\bu}_{H^{s, k}}\Norm{h}_{H^{s,k}}.
	\end{align}

	\noindent {\em Estimate of $\bR_{\balpha,0}$.} The precise expression of the second remainder in~\eqref{eq.quasilin} is the following:
	\[
	\bR_{\balpha,0}:=- \big( [\de_\bx^\balpha,  \bu] \cdot \nabla_\bx \big)\bu + \kappa  [\de_\bx^\balpha, \tfrac{1}{\bar h+h}] (\nabla_\bx h\cdot \nabla_\bx) \bu + \tfrac{\kappa }{\bar h+h} \big( [\de_\bx^\balpha, \nabla_\bx h] \cdot\nabla_\bx \big)\bu.\]
	By Lemma~\ref{L.commutator-Hs}(\ref{L.commutator-Hs-2}) and Lemma~\ref{L.embedding}
	we have
	\begin{align*}
		\Norm{\big([\de_\bx^\balpha, \bu]\cdot \nabla_\bx\big) \bu}_{L^2(\Omega)} &\lesssim \Norm{\bu}_{L^\infty_\varrho H^{s_0+1}_\bx}\Norm{\bu}_{L^2_\varrho H^s_\bx}\lesssim \Norm{\bu}_{H^{s_0+\frac32,1}}\Norm{\bu}_{H^{s,0}}.
	\end{align*}
	Next, appealing again to~\ref{L.commutator-Hs}(\ref{L.commutator-Hs-2}), we have
	\begin{align*}
		\kappa \Norm{ [\de_\bx^\balpha, \tfrac{1}{\bar h+h}] (\nabla_\bx h \cdot \nabla_\bx) \bu }_{L^2(\Omega)}  & \lesssim\kappa \Norm{\nabla_\bx \tfrac1{\bar h+h}}_{L^\infty_\varrho H^{s_0}_\bx}\Norm{(\nabla_\bx h \cdot \nabla_\bx) \bu}_{L^2_\varrho H^{s-1}_\bx} \\
		&\quad + \kappa \Norm{\nabla_\bx \tfrac1{\bar h+h}}_{L^2_\varrho H^{s-1}_\bx}\Norm{(\nabla_\bx h \cdot \nabla_\bx) \bu}_{L^\infty_\varrho H^{s_0}_\bx}.
	\end{align*}
	Now, by Lemma~\ref{L.product-Hs}(\ref{L.product-Hs-2}) and Lemma~\ref{L.composition-Hs}, one has for any $t\geq 0$ that
	\begin{align}
		\norm{\nabla_\bx \tfrac1{\bar h+h}}_{ H^{t}_\bx} &= \norm{ \tfrac{\nabla_\bx h}{(\bar h+h)^2}}_{ H^{t}_\bx}  \leq \norm{ \tfrac{\nabla_\bx h}{\bar h^2}}_{ H^{t}_\bx} +\norm{  (\tfrac1{\bar h^2}-\tfrac{1}{(\bar h+h)^2})\nabla_\bx h}_{ H^{t}_\bx} \notag \\
		& \lesssim (h_\star)^{-2} \norm{\nabla_\bx h}_{H^{t}_\bx }+ \norm{\tfrac1{\bar h^2}-\tfrac{1}{(\bar h+h)^2}}_{H^{s_0}_\bx}\norm{\nabla_\bx h}_{H^t_\bx}
		+\left\langle \norm{\tfrac1{\bar h^2}-\tfrac{1}{(\bar h+h)^2}}_{H^{t}_\bx}\norm{\nabla_\bx h}_{H^{s_0}_\bx}\right\rangle_{t>s_0} \notag\\
		&\leq  C(h_\star,\norm{h}_{H^{s_0}_\bx})  \norm{\nabla_\bx h}_{H^{t}_\bx }, \label{eq:last-fraction}
	\end{align}
	where in the last step we used that, by  Lemma~\ref{L.composition-Hs},
	\[ \norm{\tfrac1{\bar h^2}-\tfrac{1}{(\bar h+h)^2}}_{H^{s_0}_\bx}\leq  C(h_\star,\norm{h}_{H^{s_0}_\bx}) \]
	and, provided that $t>s_0$,
	\begin{align*}
		\norm{\tfrac1{\bar h^2}-\tfrac{1}{(\bar h+h)^2}}_{H^{t}_\bx}\leq \norm{\tfrac1{\bar h^2}-\tfrac{1}{(\bar h+h)^2}}_{H^{s_0}_\bx}+\norm{\nabla_\bx \tfrac{1}{(\bar h+h)^2}}_{H^{t-1}_\bx},
	\end{align*}
	and a finite induction on $t$, until $\norm{ \nabla_\bx \tfrac{1}{(\bar h+h)^2} }_{L^2_\bx} =  \norm{\tfrac{\nabla_\bx h}{\bar h+h}}_{L^2_\bx} \le  h_\star^{-2} \norm{\nabla_\bx h}_{L^2_\bx}$.
	Then, by Lemma~\ref{L.product-Hs}(\ref{L.product-Hs-2}) and Lemma~\ref{L.embedding}, we have
	\begin{align*}
		\Norm{(\nabla_\bx h \cdot \nabla_\bx) \bu}_{L^2_\varrho H^{s-1}_\bx} &\lesssim \Norm{\nabla_\bx h}_{L^2_\varrho H^{s-1}_\bx} \Norm{\bu}_{L^\infty_\varrho H^{s_0+1}_\bx} +\Norm{\nabla_\bx h}_{L^\infty_\varrho H^{s_0}_\bx} \Norm{\bu}_{L^2_\varrho H^{s}_\bx}\\ &\lesssim  \Norm{h}_{H^{s,0}} \Norm{\bu}_{H^{s_0+\frac32,1}}+ \Norm{h}_{H^{s_0+\frac32,1}} \Norm{\bu}_{H^{s,0}} 
	\end{align*}
	and
	\[\Norm{(\nabla_\bx h \cdot \nabla_\bx) \bu}_{L^\infty_\varrho H^{s_0}_\bx}\lesssim \Norm{\nabla_\bx h}_{L^\infty_\varrho H^{s_0}_\bx} \Norm{\bu}_{L^\infty_\varrho H^{s_0+1}_\bx} \lesssim \Norm{h}_{H^{s_0+\frac32,1}} \Norm{\bu}_{H^{s_0+\frac32,1}} .\]
	Finally, we have by Lemma~\ref{L.commutator-Hs}(\ref{L.commutator-Hs-2}) and Lemma~\ref{L.embedding}
	\begin{align*}
		\Norm{ \big([\de_\bx^\balpha, \nabla_\bx h] \cdot\nabla_\bx\big) \bu}_{L^2(\Omega)} & \lesssim \Norm{\nabla_\bx h}_{L^\infty_\varrho H^{s_0+1}_\bx}\Norm{\nabla_\bx \bu}_{L^2_\varrho H^{s-1}_\bx}+ \Norm{\nabla_\bx h}_{L^2_\varrho H^{s}_\bx}\Norm{\nabla_\bx \bu}_{L^\infty_\varrho H^{s_0}_\bx}\\
		& \lesssim \Norm{ \nabla_\bx h}_{ H^{s_0+\frac 32, 1}}\Norm{ \bu}_{H^{s,0}} + \Norm{ \nabla_\bx h}_{ H^{s,0}}\Norm{ \bu}_{ H^{s_0+\frac32,1}}.
	\end{align*}
	Collecting the estimates above and using that $s \ge s_0+\frac32$, we obtain 
	\begin{equation}\label{eq:est-bR0}
		\Norm{\bR_{\balpha,0}}_{L^2(\Omega)}    \lesssim   \Norm{ \bu}_{H^{s,0}} \Norm{ \bu}_{H^{s,1}}  + \kappa C(h_\star,\Norm{h}_{H^{s_0+\frac12,1}}) \big(  \Norm{  h}_{ H^{s,1}}^2+\Norm{  \nabla_\bx h}_{ H^{s,1}}  \big) \Norm{ \bu}_{H^{s,1}} .
	\end{equation}

	\noindent {\em Estimate of $\bR_{\balpha,j}$ for $1\leq j\leq k$.} The explicit expression of the second remainder in~\eqref{eq.quasilin-j} is the following
	\begin{multline*}
		\bR_{\balpha,j}:= -  \big([\de_\bx^\balpha \de_\varrho^j, \bar \bu + \bu] \cdot\nabla_\bx\big)  \bu + \kappa  [\de_\bx^\balpha \de_\varrho^j, \tfrac{1}{\bar h+h}] \big( (\nabla_\bx h\cdot\nabla_\bx)  \bu \big) + \tfrac{\kappa }{\bar h+h} \big( [\de_\bx^\balpha \de_\varrho^j, \nabla_\bx h] \cdot\nabla_\bx\big) \bu\\
		+  \de_\varrho^j\de_\bx^\balpha \left( \frac{\rho_0}{\varrho} \int_{\rho_0}^{\rho_1} \nabla_\bx h\, \dd \varrho' +\frac1\varrho\int_{\rho_0}^\varrho  \nabla_\bx  \cH \dd\varrho'\right).
	\end{multline*}
	By Lemma~\ref{L.commutator-Hsk} we have for $s\geq s_0+\frac32$ and since $0 \le |\balpha|\leq s-j$ and  $j\leq k$ with $k  \ge 2$, that
	\[
	\Norm{ \big( [\de_\bx^\balpha \de_\varrho^j,   \bu] \cdot\nabla_\bx \big) \bu}_{L^2(\Omega)} \lesssim \Norm{\bu}_{H^{s,k}} \Norm{\nabla_\bx \bu}_{H^{s-1,k}}.
	\]
	Then,
	\[
	\Norm{ \big([\de_\bx^\balpha \de_\varrho^j,  \bar \bu] \cdot\nabla_\bx\big)  \bu}_{L^2(\Omega)}
	=\Norm{ [\de_\varrho^j,  \bar \bu] \cdot\nabla_\bx \de_\bx^\balpha   \bu}_{L^2(\Omega)} \\
	\lesssim \norm{\bar\bu'}_{W^{j-1,\infty}_\varrho}\Norm{\bu}_{H^{s, j-1}}.
	\]
	Next, using Lemma~\ref{L.product-Hsk},
	\begin{align*}
		\Norm{[\de_\varrho^j,\tfrac1{\bar h}] \de_\bx^\balpha \big( (\nabla_\bx h\cdot\nabla_\bx)  \bu \big)}_{L^2(\Omega)}
		&\lesssim C(h_\star)\norm{\bar h'}_{W^{j-1,\infty}_\varrho} \Norm{ (\nabla_\bx h\cdot\nabla_\bx)  \bu}_{H^{s-1, j-1}}\\
		&\lesssim  C(h_\star)\norm{\bar h'}_{W^{j-1,\infty}_\varrho} \Norm{  h}_{H^{s,j}} \Norm{  \bu}_{H^{s,j}},
	\end{align*}
	and by Lemma~\ref{L.commutator-Hsk}, since $s\geq s_0+\frac 32$ and $2\leq j\leq s$, $|\balpha|+j\leq s$, Lemma~\ref{L.composition-Hsk-ex} and Lemma~\ref{L.product-Hsk},
	\begin{align*} \Norm{ [\de_\bx^\balpha \de_\varrho^j, \tfrac{1}{\bar h+h} -\tfrac1{\bar h}] \big( (\nabla_\bx h\cdot\nabla_\bx)  \bu \big)}_{L^2(\Omega)}
		&\lesssim \Norm{\tfrac{1}{\bar h+h} -\tfrac1{\bar h} }_{H^{s,k}} \Norm{(\nabla_\bx h\cdot\nabla_\bx)  \bu }_{H^{s-1,\min(\{k,s-1\})}}\\
		&\lesssim C(h_\star, \norm{\bar h'}_{W^{k-1,\infty}_\varrho},\Norm{h}_{H^{s-1,k-1}}) \Norm{  h}_{H^{s,k}}^2   \Norm{  \bu}_{H^{s, k}}.
	\end{align*}
	By Lemma~\ref{L.commutator-Hsk} we have for $s\geq s_0+\frac32$ and since $|\balpha|+j\leq s$ and  $2\leq j\leq s$
	\[
	\Norm{\big([\de_\bx^\balpha \de_\varrho^j, \nabla_\bx h] \cdot\nabla_\bx\big) \bu }_{L^2(\Omega)} \lesssim  \Norm{ \nabla_\bx h}_{H^{s,k}}\Norm{\bu }_{H^{s,k}}.
	\]
	We have immediately since $|\balpha|\leq s-j\leq s-1$ ,
	\[\Norm{ \de_\varrho^j\de_\bx^\balpha\Big( \frac{\rho_0}{\varrho} \int_{\rho_0}^{\rho_1} \nabla_\bx h \, \dd \varrho'\Big) }_{L^2(\Omega)} \lesssim \norm{\de_\bx^\balpha \nabla_\bx \cH\big\vert_{\varrho=\rho_0} }_{L^2_\bx} \lesssim \norm{\cH\big\vert_{\varrho=\rho_0}}_{H^{s}_\bx} \]
	and  since $(|\balpha|+1)+(j-1)\leq s$,
	\[\Norm{   \de_\varrho^j\Big( \frac1\varrho\int_{\rho_0}^\varrho \de_\bx^\balpha \nabla_\bx  \cH \dd\varrho'\Big)}_{L^2(\Omega)}\lesssim \sum_{i=0}^{j-1}\Norm{\partial_\varrho^i \de_\bx^\balpha \nabla_\bx \cH }_{L^2(\Omega)} \ \lesssim \Norm{\cH}_{H^{s,j-1}} .\]
	
	Collecting the estimates above we obtain for $1\leq j\leq k$
	\begin{multline}\label{eq:est-bRj}
		\Norm{\bR_{\balpha,j}}_{L^2(\Omega)}    \lesssim   \norm{\cH\big\vert_{\varrho=\rho_0}}_{H^{s}_\bx}+ \Norm{\cH}_{H^{s,k-1}} + \big(\norm{\bar\bu'}_{W^{k-1,\infty}_\varrho}+\Norm{ \bu}_{H^{s,k}} \big)\Norm{ \bu}_{H^{s,k}}  \\
		+ \kappa C(h_\star, \norm{\bar h'}_{W^{k-1,\infty}_\varrho},\Norm{h}_{H^{s-1,k-1}}) \big(  \Norm{  h}_{ H^{s,k}}^2+\Norm{  \nabla_\bx h}_{ H^{s,k}}  \big) \Norm{ \bu}_{H^{s,k}} .
	\end{multline}
	
	We infer the bound~\eqref{eq.est-quasilin} from~\eqref{eq:est-R0} and~\eqref{eq:est-bR0}, the bound~\eqref{eq.est-quasilin-j} from~\eqref{eq:est-Rj} and~\eqref{eq:est-bRj}, and the bound~\eqref{eq.est-quasilin-j-h} from~\eqref{eq:est-brj} and~\eqref{eq:est-rj}, and the proof is complete.
\end{proof}

\subsection{A priori energy estimates}

In this section we provide {\em a priori} energy estimates associated with the equations featured in Lemma~\ref{lem:quasilinearization}.
We start with the transport-diffusion equations in~\eqref{eq.quasilin-j} and~\eqref{eq.quasilin-j-h}, which we rewrite as
\begin{equation}\label{eq:transport-diffusion}
	\partial_t \dot h+\bu\cdot\nabla_\bx \dot h=\kappa\Delta_\bx \dot h+r+\nabla_\bx\cdot\br.
\end{equation}
\begin{lemma}\label{lem:estimate-transport-diffusion}
	There exists a universal constant $C_0>0$ such that for any $\kappa>0$ and $T>0$, for any $\bu\in L^\infty(0,T;L^\infty (\Omega))$ with $\nabla_\bx\cdot \bu \in L^1(0,T;L^\infty(\Omega))$, for any $(r,\br)\in L^2(0,T;L^2(\Omega))$ and for any $\dot h\in L^\infty(0,T;L^2(\Omega))$ with $\nabla_\bx \dot h\in  L^2(0,T;L^2(\Omega))$, such that~\eqref{eq:transport-diffusion} holds in $L^2(0,T;H^{1,0}(\Omega)')$,
	we have
	\begin{multline}
		\Norm{\dot h}_{L^\infty(0,T;L^2(\Omega))}+\kappa^{1/2}\Norm{\nabla_\bx \dot h}_{L^2(0,T;L^2(\Omega))}
		\\
		\leq C_0\big(\Norm{ \dot h\big|_{t=0}}_{L^2(\Omega)} + \Norm{r}_{L^1(0,T;L^2(\Omega))}+\kappa^{-1/2}\Norm{\br}_{L^2(0,T;L^2(\Omega))}\big)\\
		\times\exp\Big(C_0\int_0^T \Norm{\nabla_\bx\cdot\bu(t,\cdot)}_{L^\infty(\Omega)}\dd t\Big).
	\end{multline}
\end{lemma}
\begin{proof}
	Testing the equation against $\dot h$ and integrating by parts (with respect to the variable $\bx$) yields
	\[ \frac12\frac{\dd}{\dd t}\Norm{\dot h}_{L^2(\Omega)}^2 + \kappa \Norm{\nabla_\bx \dot h}_{L^2(\Omega)}^2 =\frac12\iint_\Omega (\nabla_\bx\cdot\bu) \dot h^2\dd\bx\dd\varrho + \iint_\Omega r\dot h\dd\bx\dd\varrho- \iint_\Omega \br\cdot\nabla_\bx\dot h\dd\bx\dd\varrho.\]
	The estimate follows from the Cauchy-Schwarz inequality and Gronwall's Lemma. 
\end{proof}

Next, we consider system~\eqref{eq.quasilin}, which we rewrite as
\begin{equation}\label{eq.system}
	\begin{aligned}
		\partial_t  \dot\cH+(\bar\bu+\bu)\cdot \nabla_\bx \dot \cH+\int_{\varrho}^{\rho_1}(\bar\bu'+\de_\varrho\bu)\cdot \nabla_\bx \dot \cH \dd\varrho'+\int_{\varrho}^{\rho_1} (\bar h+h)\nabla_\bx \cdot\dot\bu \dd\varrho'&=\kappa\Delta_\bx   \dot\cH+R,\\[1ex]
		\varrho \left(\partial_t\dot\bu+\big(({\bar \bu}+\bu-\kappa\tfrac{\nabla_\bx h}{\bar h+ h})\cdot\nabla_\bx \big)\dot\bu \right)
		+ {\rho_0}\nabla_\bx \dot\cH\big\vert_{\varrho=\rho_0} + \int_{\rho_0}^\varrho  \nabla_\bx  \dot\cH \dd\varrho'&=\varrho\nu\Delta_\bx   \dot\bu+\bR .
	\end{aligned}
\end{equation}
For the sake of readability, we introduce the following notations
\begin{equation}\label{eq:X01spaces}
	X^0:= \cC^0([\rho_0,\rho_1];L^2(\RR^d))\times L^2(\Omega)^d; \quad  X^1:= \cC^0([\rho_0,\rho_1];H^1(\RR^d))\times H^{1,0}(\Omega)^d.
\end{equation}
\begin{lemma}\label{lem:estimate-system}
	Let  $h_\star,h^\star, \bar M>0$ be fixed. There exists $C(h_\star,h^\star,M )>0$ such that for any $\kappa>0$ and $\nu\in [0,1]$, for any $(\bar h,\bar\bu)\in W^{1,\infty}((\rho_0,\rho_1))$,
	for any $T>0$ and
	$(h,\bu)\in L^\infty(0,T;W^{1,\infty}(\Omega))$ with $\Delta_\bx h\in L^1(0,T;L^\infty(\Omega))$
	satisfying~\eqref{eq:hydro-iso-nu} and, for almost every $t\in [0,T]$, the upper bound
	\[
	\Norm{ h(t,\cdot)}_{L^\infty(\Omega)}
	+\Norm{ \nabla_\bx h(t,\cdot)}_{L^\infty_\bx L^2_\varrho}
	+\nu^{1/2}\Norm{ \nabla_\bx h(t,\cdot) }_{L^\infty(\Omega)}
	+\Norm{ \nabla_\bx\cdot \bu(t,\cdot) }_{L^\infty(\Omega)}  \le M
	\]
	and the lower and upper bounds  
	\[ \forall (\bx,\varrho)\in  \Omega , \qquad h_\star \leq  \bar h(\varrho)+h(t,\bx,\varrho) \leq h^\star ;  \]
	and for any $(\dot\cH, \dot\bu) \in \cC^0([0,T];X^0)\cap L^2(0,T;X^1)$, with $X^0, X^1$ in~\eqref{eq:X01spaces}, and $(R,\bR)\in L^2(0,T;X^0)$ satisfying system~\eqref{eq.system} in $L^2(0,T;X^1)'$, the following estimate holds:
	\begin{multline*}\cE(\dot\cH(t,\cdot),\dot\bu(t,\cdot))^{1/2} + \kappa^{1/2} \Norm{\nabla_\bx \dot \cH}_{L^2(0,t;L^2(\Omega))} +  \kappa^{1/2} \norm{\nabla_\bx \dot \cH \big\vert_{\varrho=\rho_0}  }_{ L^2(0,t;L^2_\bx)}+\nu\Norm{\nabla_\bx \dot \bu}_{L^2(\Omega)}^2 \\
		\leq \left( \cE(\dot\cH(0,\cdot),\dot\bu(0,\cdot))^{1/2}+C\int_0^t \cE(R(\tau,\cdot),\bR(\tau,\cdot))^{1/2}\dd \tau \right)\\
		\times \exp\Big( C \int_0^t \big(1+ \kappa^{-1}\Norm{ \bar \bu'+\de_\varrho \bu (\tau,\cdot)}_{L^\infty_\bx L^2_\varrho}^2 \big)\dd\tau\Big),
	\end{multline*}
	where we denote
	\[ \cE(\dot\cH,\dot\bu):= \frac12 \int_{\rho_0}^{\rho_1}\int_{\RR^d} \dot\cH^2+ \varrho(\bar h+h)\big|\dot\bu\big|^2\dd\bx\dd \varrho\ + \ \frac{\rho_0}2\int_{\RR^d} \dot\cH^2\big\vert_{\varrho=\rho_0}\dd\bx.\]
\end{lemma}

\begin{proof}
	We test the first equation against $\dot\cH\in L^2(0,T;H^{1,0}(\Omega)) $, its trace on $\{(\bx,\rho_0 ),\bx\in\RR^d\}$ against $\rho_0\dot\cH\big\vert_{\varrho=\rho_0}\in L^2(0,T;H^1(\RR^d))$,  and the second equation against $(\bar h+h) \dot\bu\in L^2(0,T;H^{1,0}(\Omega)) $. This yields, after integration by parts
	\begin{align*}
		& \frac{\dd}{\dd t} \cE(\dot\cH, \dot\bu)+ \kappa \Norm{\nabla_\bx \dot\cH}_{L^2(\Omega)}^2 +  \rho_0\kappa \norm{\nabla_\bx \dot \cH \big\vert_{\varrho=\rho_0}  }_{ L^2_\bx}^2+\nu\sum_{i=1}^d\int_\Omega \varrho (\bar h+h) |\partial_{x_i}\dot\bu|^2\dd\bx\dd\varrho \\
		& = - \left((\bar \bu + \bu) \cdot\nabla_\bx \dot\cH , \dot\cH\right)_{L^2(\Omega)}-\left(\int_\varrho^{\rho_1} (\bar \bu'+\de_\varrho\bu) \cdot \nabla_\bx \dot \cH \, d\varrho', \dot\cH\right)_{L^2(\Omega)} & {\rm (i)}\\
		&\quad - \left(\int_\varrho^{\rho_1} (\bar h+h) \nabla_\bx \cdot \dot\bu\, d\varrho', \dot\cH\right)_{L^2(\Omega)}+ \big(R, \dot\cH\big)_{L^2(\Omega)} & {\rm (ii)}\\
		&\quad  - \left(\varrho (\bar \bu + \bu) \cdot \nabla_\bx \dot\bu, (\bar h+h) \dot\bu\right)_{L^2(\Omega)} + \kappa \left( \varrho (\nabla_\bx h \cdot \nabla_\bx) \dot\bu, \dot\bu\right)_{L^2(\Omega)}& {\rm (iii)} \\
		& \quad
		-\big( \rho_0\nabla_\bx \dot\cH\big\vert_{\varrho=\rho_0} , (\bar h+h) \dot\bu\big)_{L^2(\Omega)}
		- \big( \int_{\rho_0}^\varrho \nabla_\bx \dot\cH \, d\varrho', (\bar h+h) \dot\bu \big)_{L^2(\Omega)}& {\rm (iv)} \\
		& \quad -\nu \big( \varrho (\nabla_\bx h\cdot\nabla)\dot\bu , \dot\bu\big)_{L^2(\Omega)} + \big(\varrho\bR, (\bar h+h) \dot\bu\big)_{L^2(\Omega)}& {\rm (v)} \\
		&\quad -\rho_0\left(\big((\bar\bu+\bu)\cdot \nabla_\bx \dot \cH\big)\big\vert_{\varrho=\rho_0},\dot\cH\big\vert_{\varrho=\rho_0} \right)_{L^2_\bx} -\rho_0\left(\int_{\rho_0}^{\rho_1}(\bar\bu'+\de_\varrho\bu)\cdot \nabla_\bx \dot \cH\dd\varrho',\dot\cH\big\vert_{\varrho=\rho_0} \right)_{L^2_\bx} & {\rm (vi)}\\
		&\quad -\rho_0\left(\int_{\rho_0}^{\rho_1}(\bar h+h)\nabla_\bx \cdot\dot\bu \dd\varrho',\dot\cH\big\vert_{\varrho=\rho_0} \right)_{L^2_\bx} +\rho_0\left(R\big\vert_{\varrho=\rho_0} ,\dot\cH\big\vert_{\varrho=\rho_0} \right)_{L^2_\bx}& {\rm (vii)}\\
		& \quad + \tfrac 12 (\varrho (\de_t h) \dot \bu, \dot \bu)_{L^2(\Omega)} . & {\rm (viii)}\\
	\end{align*}
	
	We consider first the second terms in (i) and (vi). We have by an immediate application of Cauchy-Schwarz inequality and  the continuous embedding $L^\infty((\rho_0,\rho_1))\subset   L^2((\rho_0,\rho_1))$
	\begin{align}
		& \left| \left( \int_\varrho^{\rho_1} (\bar \bu' + \de_\varrho \bu) \cdot \nabla_\bx \dot\cH  \, d\varrho' , \dot\cH \right)_{L^2(\Omega)} \right| 
		+\rho_0\left| \left( \int_{\rho_0}^{\rho_1} (\bar \bu' + \de_\varrho \bu) \cdot \nabla_\bx \dot\cH  \, d\varrho' , \dot\cH\big\vert_{\varrho=\rho_0}  \right)_{ L^2_\bx} \right| \nonumber\\
		&\qquad \lesssim \Norm{ \bar \bu'+\de_\varrho \bu }_{L^\infty_\bx L^2_\varrho} \Norm{\nabla_\bx \dot\cH }_{L^2(\Omega)} \Big(\Norm{\dot\cH }_{L^2(\Omega)}+ \norm{ \dot\cH\big\vert_{\varrho=\rho_0}  }_{ L^2_\bx} \Big).
		\label{eq:badterms}
	\end{align}
	Notice that the right-hand side~\eqref{eq:badterms} cannot be bounded by the energy functional $\cE(\dot\cH,\dot\bu)$, and this is exactly the point where we use the assumption $\kappa>0$. Let us now estimate all other terms.
	
	Using integration by parts in the variable $\bx$, we estimate the first addend of (i) and (vi)  as follows:
	\begin{multline*}
		\left| \left(  (\bar \bu + \bu) \cdot \nabla_\bx \dot\cH  , \dot\cH \right)_{L^2(\Omega)} \right| 
		+\rho_0\left| \left(  \big((\bar \bu + \bu)  \cdot \nabla_\bx \dot\cH\big)\big\vert_{\varrho=\rho_0}   , \dot\cH\big\vert_{\varrho=\rho_0}  \right)_{ L^2_\bx} \right| \nonumber\\
		\lesssim 
		\Norm{ \nabla_\bx\cdot \bu }_{L^\infty(\Omega)} \Norm{ \dot\cH }_{L^2(\Omega)}^2 + \norm{ \nabla_\bx \cdot\bu\big\vert_{\varrho=\rho_0} }_{ L^\infty_\bx} \norm{ \dot\cH\big\vert_{\varrho=\rho_0}  }_{ L^2_\bx}^2 .
	\end{multline*}
	The contributions in (iii) and (viii) compensate after integration by parts in $\bx$, using the first equation in~\eqref{eq:hydro-iso-nu}.
	Now consider the first addend of (ii)  together with the second addend of (iv). By application of Fubini's theorem 
	we have
	\[ \int_{\mathbb{R}^d} \int_{\rho_0}^{\rho_1} \left(\int_{\rho_0}^\varrho \nabla_\bx \dot\cH(\varrho') \, \dd\varrho'\,\right)  \cdot (\bar h+h)(\varrho) \dot\bu(\varrho)  \, \dd\varrho \, \dd\bx= \int_{\mathbb{R}^d} \int_{\rho_0}^{\rho_1}\left(\int_{\varrho'}^{\rho_1}  (\bar h+h)(\varrho) \dot\bu(\varrho) \, \dd \varrho \right)\cdot \nabla_\bx  \dot\cH(\varrho')  \, \dd\varrho'\, \dd\bx\]
	and hence, integrating by parts in $\bx$, we infer
	\begin{align*}
		&\Bigg|  \int_{\mathbb{R}^d} \int_{\rho_0}^{\rho_1} \int_\varrho^{\rho_1} (\bar h+h)(\varrho') \nabla_\bx \cdot \dot\bu(\varrho')\, \dd\varrho'\, \dot\cH(\varrho) \, \dd\varrho\, \dd\bx \\
		& \qquad + \int_{\mathbb{R}^d} \int_{\rho_0}^{\rho_1} \left(\int_{\rho_0}^\varrho \nabla_\bx \dot\cH(\varrho') \, \dd\varrho'\,\right)  \cdot (\bar h+h)(\varrho) \dot\bu(\varrho)  \, \dd\varrho \, \dd\bx \Bigg| \\
		& \quad = \Bigg|\int_{\mathbb{R}^d} \int_{\rho_0}^{\rho_1} \int_\varrho^{\rho_1} (\nabla_\bx h)(\varrho') \cdot\dot\bu(\varrho') \dot\cH(\varrho) \, \, \dd \varrho'\, \dd\varrho\, \dd\bx\Bigg|
		\lesssim \Norm{ \nabla_\bx h}_{L^\infty_\bx L^2_\varrho}  \Norm{\dot\bu}_{L^2(\Omega)}\Norm{\dot\cH}_{L^2(\Omega)}.
	\end{align*}
	Concerning first addend of (iv) and the first addend of (vii), we have after integrating by parts with respect to the $\bx$ variable and using Cauchy-Schwarz inequality
	\begin{align*}&\left|-\left(\rho_0\nabla_\bx \dot\cH\big\vert_{\varrho=\rho_0} , (\bar h+h) \dot\bu\right)_{L^2(\Omega)}
		- \rho_0\left(\int_{\rho_0}^{\rho_1}(\bar h+h)\nabla_\bx \cdot\dot\bu \dd\varrho,\dot\cH\big\vert_{\varrho=\rho_0} \right)_{L^2_\bx}  \right| \\
		&\quad =\rho_0\left| \left(\int_{\rho_0}^{\rho_1}(\nabla_\bx h) \cdot\dot\bu \dd\varrho,\dot\cH\big\vert_{\varrho=\rho_0} \right)_{L^2_\bx} \right|
		\lesssim\Norm{ \nabla_\bx h}_{L^\infty_\bx L^2_\varrho}  \Norm{\dot\bu}_{L^2(\Omega)}  \norm{\dot\cH\big\vert_{\varrho=\rho_0}}_{L^2_\bx}.
	\end{align*}
	Concerning the first addend of (v), we have for an arbitrarily large constant $K>0$,
	\[  \nu \left| \big( \varrho (\nabla_\bx h\cdot\nabla)\dot\bu , \dot\bu\big)_{L^2(\Omega)} \right| \leq \frac{1}{2K} \nu \Norm{ \nabla\dot\bu}_{L^2(\Omega)}^2 + \frac{K\rho_1^2}2 \nu\Norm{ \nabla_\bx h }_{L^\infty(\Omega)}^2 \Norm{\dot\bu }_{L^2(\Omega)}^2.\]
	The last contributions, namely 
	\[\Big|  \big(R, \cH\big)_{L^2(\Omega)} +  \big(\bR,\varrho (\bar h+h) \dot\bu\big)_{L^2(\Omega)} + {\rho_0\big(R\big\vert_{\varrho=\rho_0} ,\dot\cH\big\vert_{\varrho=\rho_0} \big)_{L^2_\bx}} \Big|,\]
	are easily controlled by means of Cauchy-Schwarz inequality.
	Collecting all of the above, and using that
	\[ \cE(\dot\cH,\dot\bu) \approx \Norm{\dot\cH}_{L^2(\Omega)}^2 + \Norm{\dot\bu}_{L^2(\Omega)}^2+ \norm{\dot\cH\big\vert_{\varrho=\rho_0}}_{L^2_\bx}^2\]
	and
	\[\nu\sum_{i=1}^d\int_\Omega \varrho (\bar h+h) |\partial_{x_i}\dot\bu|^2\dd\bx\dd\varrho \gtrsim \nu \Norm{\nabla_\bx\dot\bu}_{L^2(\Omega)}^2\]
	since $\rho_0 h_\star\leq \varrho(\bar h+h)\leq \rho_1 h^\star$, we obtain (choosing $K$ sufficiently large)
	\begin{multline*} \frac{\dd}{\dd t} \cE(\dot\cH, \dot\bu)+ \kappa \Norm{\nabla_\bx \dot \cH}_{L^2(\Omega)}^2 +  \rho_0\kappa \norm{\nabla_\bx \dot \cH \big\vert_{\varrho=\rho_0}  }_{ L^2_\bx}^2\\
		\leq C\,
		\cE(\dot\cH, \dot\bu) +C \Norm{ \bar \bu'+\de_\varrho \bu }_{L^\infty_\bx L^2_\varrho}   \cE(\dot\cH, \dot\bu)^{1/2} \Norm{\nabla_\bx \dot \cH}_{L^2(\Omega)}\\
		+ C \cE(\dot\cH, \dot\bu)^{1/2}  \cE(R, \bR)^{1/2} ,
	\end{multline*}
	with $C=C(h_\star,h^\star,M)$. We deduce (augmenting $C$ if necessary)
	\begin{multline*} \frac{\dd}{\dd t} \cE(\dot\cH, \dot\bu)+ \frac\kappa2 \Norm{\nabla_\bx \cH}_{L^2(\Omega)}^2 +  \rho_0\kappa \norm{\nabla_\bx \dot \cH \big\vert_{\varrho=\rho_0}  }_{ L^2_\bx}^2\\
		\leq C\, \big(1+ 
		\kappa^{-1}\Norm{ \bar \bu'+\de_\varrho \bu }_{L^\infty_\bx L^2_\varrho}^2 \big)\cE(\dot\cH, \dot\bu) + C \cE(\dot\cH, \dot\bu)^{1/2}  \cE(R, \bR)^{1/2} ,
	\end{multline*}
	and the desired estimate follows by Gronwall's inequality.
\end{proof}

\subsection{Large-time existence; proof of Theorem~\ref{thm-well-posedness}}

We prove the large-time existence and energy esti\-mates on solutions to the regularized system~\eqref{eq:hydro-iso-nu} in the following result. Compared with Proposition~\ref{P.WP-nu}, we provide an existence time which is uniformly bounded (from below) with respect to the artificial regula\-ri\-zation parameter $\nu>0$, and specify the dependency with respect to the diffusivity parameter $\kappa$, in relation with the size of the data. It is in this sense that the existence of strong solutions to the hydrostatic system holds for \emph{large} times. We then complete the proof of Theorem~\ref{thm-well-posedness} at the end of this section.

\begin{proposition}\label{P.regularized-large-time-WP}
	Let  $s, k \in \NN$ be such that $s> 2+\frac d 2$, $2\leq k\leq s$, and $\bar M,M^\star,h_\star,h^\star>0$. Then, there exists $C>0$ such that, for any $ 0< \nu \leq \kappa\leq 1$, and 
	\begin{itemize}
		\item for any $(\bar h, \bar \bu) \in W^{k,\infty}((\rho_0,\rho_1))$ such that 
		\[  \norm{\bar h}_{W^{k,\infty}_\varrho } + \norm{\bar \bu'}_{W^{k-1,\infty}_\varrho }\leq \bar M\,;\]
		\item for any initial data $(h_0, \bu_0)=(h_0(\bx, \varrho), \bu_0(\bx, \varrho)) \in H^{s,k}(\Omega)$ with
		\[
		M_0:=
		\Norm{\cH_0}_{H^{s,k}}+\Norm{\bu_0}_{H^{s,k}}+\norm{\cH_0\big\vert_{\varrho=\rho_0}}_{H^s_\bx}+\kappa^{1/2}\Norm{h_0}_{H^{s,k}} \leq M^\star,
		\]
		and 
		\[ \forall (\bx,\varrho)\in  \Omega , \qquad h_\star \leq  \bar h(\varrho)+h_0(\bx,\varrho) \leq h^\star , \]
	\end{itemize}
	the following holds. Denoting
	\[
	T^{-1}= C\, \big(1+ \kappa^{-1} \big(\norm{\bar \bu'}_{L^2_\varrho}^2+M_0^2\big)  \big),  
	\]
	there exists a unique strong solution $(h,\bu)\in  \cC([0,T];H^{s,k}(\Omega)^{1+d})$ to the Cauchy problem associated with~\eqref{eq:hydro-iso-nu} and initial data $(h,\bu)\big\vert_{t=0}=(h_0,\bu_0)$.
	Moreover, $h\in  L^2(0,T;H^{s+1,k}(\Omega))$ and one has, for any $t\in[0,T]$, the lower and the upper bounds
	\[ \forall (\bx,\varrho)\in  \Omega , \qquad h_\star/2 \leq  \bar h(\varrho)+h(t,\bx,\varrho) \leq 2\,h^\star , \]
	and the estimate
	\begin{multline*}\cF(t):=
		\Norm{\cH(t,\cdot)}_{H^{s,k}}+\Norm{\bu(t,\cdot)}_{H^{s,k}} +\norm{\cH\big\vert_{\varrho=\rho_0}(t,\cdot)}_{H^s_\bx} +\kappa^{1/2}\Norm{h(t,\cdot)}_{H^{s,k}} \\+ \kappa^{1/2} \Norm{\nabla_\bx \cH}_{L^2(0,t;H^{s,k})} +  \kappa^{1/2} \norm{\nabla_\bx \cH \big\vert_{\varrho=\rho_0}  }_{ L^2(0,t;H^s_\bx)} +\kappa \Norm{\nabla_\bx h}_{L^2(0,t;H^{s,k})} \leq C M_0.
	\end{multline*}
\end{proposition}
\begin{proof}Let us denote by $T^\star\in(0,+\infty]$ the maximal time of existence of $(h,\bu)\in \cC^0([0,T^\star);H^{s,k}(\Omega))$ as provided by Proposition~\ref{P.WP-nu}, and 
	\begin{multline*}T_\star=\sup\Big\{ 0<T< T^\star \ : \  \forall t\in (0,T), \qquad  
		h_\star/2\leq \bar h(\varrho)+h(t,\bx,\varrho) \leq 2\, h^\star  \quad \text{ and } \quad  \cF(t)\leq C_0 M_0\Big\} ,
	\end{multline*}
	where $C_0>1$ will be determined later on.
	By the continuity in time of the solution, and using that the linear operator $h\mapsto \cH:= \int_{\varrho}^{\rho_1}h(\cdot,\varrho' )\dd\varrho'$ (resp. $h\mapsto\cH\big\vert_{\varrho=\rho_0}$) is well-defined and bounded from $H^{s,k}(\Omega)$ to itself (resp. $H^s_\bx(\RR^d)$)  we have $T_\star>0$.
	Using Lemma~\ref{lem:quasilinearization},~\ref{lem:estimate-transport-diffusion} and~\ref{lem:estimate-system} and, therein, the inequalities $\Norm{h}_{H^{s-1,k-1}}=\Norm{\de_\varrho \cH}_{H^{s-1,k-1}}\leq \Norm{\cH}_{H^{s,k}}$, and (since $\nu\leq \kappa$) $\nu^{1/2}\Norm{ \nabla_\bx h }_{L^\infty(\Omega)}\leq \kappa^{1/2}\Norm{h}_{H^{s,k}}$, we find that there exists $c_0>1$ depending only on $\rho_0h_\star, \rho_1h^\star$; and $C>0$ depending on $\bar M, h_\star, h^\star, C_0M_0$ such that for any $0<t<T_\star$,
	\begin{multline}
		\Norm{\cH(t,\cdot)}_{H^{s,0}}+\Norm{\bu(t,\cdot)}_{H^{s,0}} + \norm{ \cH\big\vert_{\varrho=\rho_0}(t,\cdot) }_{H^s_\bx} + \kappa^{1/2} \Norm{\nabla_\bx \cH}_{L^2(0,t;H^{s,0})} +  \kappa^{1/2} \norm{\nabla_\bx \cH \big\vert_{\varrho=\rho_0}  }_{ L^2(0,t;H^s_\bx)}\\
		\leq c_0 \left( \Norm{\cH_0}_{H^{s,0}}+\Norm{\bu_0}_{H^{s,0}} + \norm{ \cH_0\big\vert_{\varrho=\rho_0} }_{H^s_\bx}  +C \,C_0M_0 \, \big(t+\sqrt t \big)\right)\\
		\times \exp\Big( C \int_0^t \big(1
		+\kappa^{-1}\Norm{ \bar \bu'+\de_\varrho \bu }_{L^\infty_\bx L^2_\varrho}^2 \big)\dd\tau\Big);
	\end{multline}
	and (using a slightly adapted version of Lemma~\ref{lem:estimate-system} which does not involve the trace of $\partial_\varrho^j\cH$ at the surface) for any $1\leq j\leq k$
	\begin{multline}
		\Norm{\partial_\varrho^j\cH(t,\cdot)}_{H^{s-j,0}} + \kappa^{1/2} \Norm{\nabla_\bx \partial_\varrho^j \cH}_{L^2(0,t;H^{s-j,0})} \\
		\leq  \left( \Norm{\partial_\varrho^j\cH(0,\cdot)}_{H^{s-j,0}}+C\, C_0 M_0 \, \big(t+\sqrt t\big)\right)\\
		\times \exp\Big( C \int_0^t \Norm{\nabla_\bx\cdot \bu(\tau,\cdot) }_{L^\infty(\Omega)} \dd \tau\Big),
	\end{multline}
	and
	\begin{multline}
		\Norm{\partial_\varrho^j\bu(t,\cdot)}_{H^{s-j,0}} + \nu^{1/2} \Norm{\nabla_\bx \partial_\varrho^j \bu}_{L^2(0,t;H^{s-j,0})} \\
		\leq  \left( \Norm{\partial_\varrho^j\bu(0,\cdot)}_{H^{s-j,0}}+C \, C_0M_0 \, \big(t+\sqrt t \big)\right)\\
		\times \exp\Big( C \int_0^t  \Norm{\nabla_\bx\cdot \big(\bu-\kappa\tfrac{\nabla_\bx h}{\bar h+h}\big)(\tau,\cdot) }_{L^\infty(\Omega)} \dd \tau\Big);
	\end{multline}
	and finally for any $0\leq j\leq k$
	\begin{multline}
		\kappa^{1/2}\Norm{\partial_\varrho^jh(t,\cdot)}_{H^{s-j,0}} + \kappa \Norm{\nabla_\bx \partial_\varrho^j h}_{L^2(0,t;H^{s-j,0})} \\
		\leq  \left(\kappa^{1/2} \Norm{\partial_\varrho^jh(0,\cdot)}_{H^{s-j,0}}+C\, C_0 M_0\, \big( t+\sqrt{t} \big)\right)\\
		\times \exp\Big( C \int_0^t \Norm{\nabla_\bx\cdot \bu(\tau,\cdot) }_{L^\infty(\Omega)} \dd \tau\Big).
	\end{multline}
	By the continuous embeddings $H^{s_0+\frac12,1}\subset L^\infty_\varrho H^{s_0}\subset L^\infty(\Omega)$ for any $s_0>d/2$ (see Lemma~\ref{L.embedding}) and since $k\geq 1$ and $s>\frac32+\frac{d}2$, we have
	\[\Norm{\nabla_\bx\cdot \bu}_{L^\infty(\Omega)}  + \Norm{\nabla_\bx\cdot \big(\bu-\kappa\tfrac{\nabla_\bx h}{\bar h+h}\big) }_{L^\infty(\Omega)} \leq C(h_\star)\big( \Norm{\bu}_{H^{s,k}} +\kappa\Norm{  h}_{H^{s,k}}^2 +\kappa\Norm{ \nabla_\bx h}_{H^{s,k}}\big) .\]
	We deduce that
	\[
	\cF(t)\leq c \Big( M_0
	+C \, C_0 M_0 \, \big(t +\sqrt t\big)\Big)
	\times \exp\Big( C\big( t+\sqrt t +  \kappa^{-1}\int_0^t\Norm{ \bar \bu'+\de_\varrho \bu(\tau,\cdot) }_{L^\infty_\bx L^2_\varrho}^2\dd\tau \big)\Big),
	\]
	where we recall that  $c_0>1$ depends only on $h_\star$ and $h^\star$; and $C>0$ depends on $\bar M, C_0M_0,  h_\star, h^\star$.
	Hence choosing $C_0=2c_0$ 
	and using that (by Lemma~\ref{L.embedding} and since $k\geq 2$ and $s>\frac32+\frac{d}2$)
	\[ \Norm{ \bar \bu'+\de_\varrho \bu }_{L^\infty_\bx L^2_\varrho}\leq \Norm{ \bar \bu'+\de_\varrho \bu }_{L^2_\varrho L^\infty_\bx}^2  \lesssim \norm{\bar \bu'}_{L^2_\varrho}^2 +\Norm{\bu}_{H^{s,k}}^2\leq \norm{\bar \bu'}_{L^2_\varrho}^2+(C_0M_0)^2,\]
	we find that there exists $C_0\geq 1$ depending only on $\bar M,M^\star, h_\star, h^\star$ such that 
	\[t \big(1+\kappa^{-1}\big(\norm{\bar \bu'}_{L^2_\varrho}^2+M_0^2\big)\big) \leq C_0^{-1} \quad \Longrightarrow \quad 
	\cF(t)\leq \frac34 C_0 M_0.\]
	Now we remark that since
	\[\partial_th+\bar\bu\cdot\nabla_\bx h=\kappa\Delta_\bx h+g \quad \text{ with } \quad  g=-\nabla_\bx\cdot(\bar h \bu+h\bu)\]
	and by the positivity of the heat kernel we have
	\[\inf_{\Omega} h(t,\cdot) \geq \inf_\Omega h_0 - \Norm{g}_{L^1(0,t;L^\infty(\Omega) )}, \qquad \sup_{\Omega} h(t,\cdot) \leq \sup_\Omega h_0 + \Norm{g}_{L^1(0,t;L^\infty(\Omega))}.\]
	Now, by the continuous embedding $H^{s-1,1}(\Omega)\subset L^\infty(\Omega)$ (since $s>\frac 3 2 + \frac d 2$), we have that
	\[ \Norm{g}_{L^\infty(\Omega)} \lesssim  \norm{\bar h}_{W^{1,\infty}_\varrho} \Norm{\bu}_{H^{s,1}}
	+\Norm{h}_{H^{s,1}}\Norm{\bu}_{H^{s,1}}\\
	\leq C(\bar M)  (1+\kappa^{-1}M_0^2) .
	\]
	Hence augmenting $C_0$ if necessary we find that 
	\[ t (1+\kappa^{-1}M_0^2)\leq C_0^{-1} \quad \Longrightarrow \quad  \forall(\bx,\varrho)\in\Omega,\quad \frac23 h_\star\leq \bar h(\varrho)+h(t,\bx,\varrho) \leq \frac32 h^\star.\]
	By a continuity argument we infer $T_\star\geq \Big( C \big(1+\kappa^{-1}\big(\norm{\bar \bu'}_{L^2_\varrho}^2+M_0^2\big)\big) \Big)^{-1}$, and the proof is complete.
\end{proof}
\bigskip

{\bf Completion of the proof of Theorem~\ref{thm-well-posedness}}
\medskip

In order to complete the proof of Theorem~\ref{thm-well-posedness}, there remains to consider vanishing viscosity limit, $\nu\searrow 0$, in Proposition~\ref{P.regularized-large-time-WP}. Let us briefly sketch the standard argument. 
By Proposition~\ref{P.regularized-large-time-WP}, we construct a family $(h_\nu,\bu_\nu)\in \cC^0([0,T];H^{s,k}(\Omega))$ of solutions to~\eqref{eq:hydro-iso-nu} with $(h_\nu,\bu_\nu)\big\vert_{t=0}=(h_0,\bu_0)$ indexed by the parameter $\nu>0$. Notice that the time of existence and associated bounds provided by Proposition~\ref{P.regularized-large-time-WP} are uniform with respect to the parameter $\nu>0$. Hence by the Banach-Alaoglu theorem there exists a subsequence which converges weakly towards $(h,\bu)\in L^\infty(0,T;H^{s,k}(\Omega)^{1+d})$, satisfying the estimates of Proposition~\ref{P.regularized-large-time-WP}. Using the equations, we find that $(\partial_t\zeta_\nu,\partial_t \bu_\nu) $ are uniformly bounded in $L^\infty(0,T;H^{s-2,k})$. 
The Aubin-Lions lemma (see~\cite{Simon87}) implies that, up to extracting a subsequence, the convergence holds strongly in $(h,\bu)\in \cC^0([0,T];H^{s',k}(B)^{1+d})$ for any $0\leq s'<s$ for any bounded $B\subset \RR^d\times(\rho_0,\rho_1)$. Choosing $s'>3/2+d/2$ and using Lemma~\ref{L.embedding} and Sobolev embedding, we can pass to the limit in the nonlinear terms of the equation and infer that that $(h,\bu)$ is a strong solution to~\eqref{eq:hydro-iso-nu} with $\nu=0$. Moreover, since $(h,\bu)\in \cC^0([0,T];H^{s-2,k}(\Omega)^{1+d})$, we have $(h,\bu)\in \cC^0([0,T];H^{s',k}(\Omega)^{1+d})$ for any $0\leq s'<s$.

Uniqueness of the solution $(h,\bu)\in L^\infty(0,T;H^{s,k}(\Omega)^{1+d})$ follows by using Lemma~\ref{lem:estimate-system} on the difference between two solutions, and Gronwall's Lemma.

There remains to prove that $(h,\bu)\in \cC^0([0,T];H^{s,k}(\Omega)^{1+d})$. 
We prove the equivalent statement that for any $\balpha\in\NN^d$ and $j\in\NN$ such that $0\leq |\balpha|+j\leq s$ and $0\leq j\leq k$, $(\de_\varrho^j\de_\bx^\balpha h, \de_\varrho^j\de_\bx^\balpha\bu)\in \cC^0([0,T];L^2(\Omega)^{1+d})$. By Lemma~\ref{lem:quasilinearization}, as long as $\kappa>0$, we can write
\begin{align*}
	\partial_t (\de_\varrho^j\de_\bx^\balpha h)-\kappa\Delta_\bx (\de_\varrho^j\de_\bx^\balpha h)
	&=r_{\balpha,j}+\nabla_{\bx} \cdot \br_{\balpha,j},\\
	\partial_t (\de_\varrho^j\de_\bx^\balpha \bu)+(\bv\cdot \nabla_\bx)(\de_\varrho^j\de_\bx^\balpha \bu)
	&=R_{\balpha,j},
\end{align*}
with $(r_{\balpha,j},\br_{\balpha,j},R_{\balpha,j})\in L^2(0,T;L^2(\Omega))^{1+2d}$ and $\bv(\cdot,\varrho):=\bar\bu(\varrho)+\big(\bu-\kappa\tfrac{\nabla_\bx h}{\bar h+h}\big)(\cdot,\varrho)\in W^{k,\infty}((\rho_0,\rho_1))^d+ L^2(0,T;H^{s,k}(\Omega))^d$.
In other words, $\de_\varrho^j\de_\bx^\balpha h$ satisfies a heat equation and continuity in time stems from the Duhamel formula, as already used in Proposition~\ref{P.WP-nu}; and $\de_\varrho^j\de_\bx^\balpha \bu$ satisfies a transport equation and continuity in time is standard, see {\em e.g.}~\cite{BCD11}*{Th.~3.19}. Let us acknowledge however that our situation is slightly different, since $\Omega$ is neither the Euclidean space or the torus, and advection occurs only in the direction $\bx$ (and not $\varrho$). It is however easy to adapt the proof of~\cite{BCD11}*{Th.~3.19} to infer $\de_\varrho^j\de_\bx^\balpha \bu \in L^2(\rho_0,\rho_1;\cC^0([0,T];L^2(\RR^d)))^d\subset \cC^0([0,T];L^2(\Omega))^d$ from the facts that $R_{\balpha,j} \in L^2(0,T;L^2(\Omega))^d\subset L^2(\rho_0,\rho_1; L^1(0,T;L^2(\RR^d)))^d $ and $\nabla_\bx \bv\in L^2(0,T;H^{s-1,k}(\Omega))^{d\times d} \subset L^1(0,T;L^\infty(\rho_0,\rho_1;H^{s-3/2}(\RR^d)))^{d\times d}$ and $s-3/2>d/2$, the continuous embeddings following from Minkowski inequality and Lemma~\ref{L.embedding}. 

This concludes the proof of Theorem~\ref{thm-well-posedness}.

\section{The non-hydrostatic system}\label{S.NONHydro}

In this section we study the  local well-posedness theory for the non-hydrostatic system in isopycnal coordinates, which we recall below.
\begin{equation}\label{eq:nonhydro-iso-recall}
	\begin{aligned}
		\partial_t  h+\nabla_{ \bx} \cdot\big((\bar h +h)(\bar\bu +\bu)\big)&= \kappa \Delta_{ \bx}  h,\\
		\varrho\Big( \partial_t  \bu+\big((\bar \bu +  \bu - \kappa\tfrac{\nabla_{ \bx}  h}{ \bar h+  h})\cdot\nabla_\bx\big)  \bu\Big)+  \nabla_{ \bx}  P+ \frac{\nabla_{ \bx}  \cH}{\bar h+  h}( \de_\varrho P + \varrho \bar h) &=0,\\
		\mu \varrho\Big( \partial_t  w+\big(\bar \bu +  \bu - \kappa\tfrac{\nabla_{ \bx}   h}{ \bar h+ h}\big)\cdot\nabla_{ \bx}   w\Big)- \frac{\de_\varrho  P}{\bar h+ h} + \frac{\varrho  h}{\bar h + h}&=0,\\
		-(\bar h +  h) \nabla_{\bx} \cdot  \bu-\nabla_{\bx} \cH \cdot ({\bar \bu}'+\partial_\varrho\bu) +\partial_\varrho  w&=0, \quad \text{(div.-free cond.)} \\
		\cH(\cdot, \varrho)=\int_{\varrho}^{\rho_1}  h(\cdot, \varrho')\dd\varrho', \qquad P\big|_{\varrho=\rho_0} = 0, \qquad  w\big|_{\varrho=\rho_1}&=0.  \quad \text{(bound. cond.)}
	\end{aligned}
\end{equation}

\subsection{The pressure reconstruction}\label{S.Nonhydro:Poisson}
The first step of our analysis consists in showing how the pressure variable, $P$, can be uniquely reconstructed (thanks to the ``divergence-free'' incompressibility constraint) from prognostic variables $\bu$, $w$ and $h$ (or, equivalently, $\cH$), through an elliptic boundary-value problem. 
Differentiating the ``divergence-free'' incompressibility constraint in~\eqref{eq:nonhydro-iso-recall} with respect to time yields
\[-(\bar h+h)\nabla_\bx\cdot \partial_t \bu-(\nabla_\bx \cH)\cdot(\partial_\varrho \partial_t \bu)+\partial_\varrho \partial_t  w=(\partial_t h  )(\nabla_\bx\cdot \bu) +(\nabla_\bx \partial_t \cH)\cdot({\bar \bu}'+\partial_\varrho\bu).\]
We plug the expressions for $\de_t \bu, \de_t w, \de_t h, \de_t \cH$ provided by~\eqref{eq:nonhydro-iso-recall} inside the above identity. Reorganizing terms, this yields the following
\begin{multline*}\textstyle
	(\bar h+h)\nabla_\bx\cdot \left( \frac1\varrho\nabla_\bx  P+\frac{  \nabla_\bx \cH (\partial_\varrho  P+\varrho \bar h)}{ \varrho (\bar h+h)}\right)+(\nabla_\bx \cH)\cdot\left(\partial_\varrho \left( \frac1\varrho\nabla_\bx  P+\frac{\nabla_\bx \cH(\partial_\varrho  P + \varrho \bar h)}{\varrho (\bar h+h)} \right)\right)
	+\partial_\varrho \left(  \frac{\de_\varrho  P + \varrho \bar h }{\mu\varrho (\bar h+h)} \right)\\
	\textstyle=-(\bar h + h) \nabla_\bx\cdot  \left(\big(({\bar \bu}+\bu-\kappa\tfrac{\nabla_\bx  h}{\bar h+ h})\cdot\nabla_\bx\big) \bu \right)
	-(\nabla_\bx \cH)\cdot\left(\partial_\varrho \left(\big(({\bar \bu}+\bu-\kappa\tfrac{\nabla_\bx  h}{\bar h+ h})\cdot\nabla_\bx \big)\bu \right) \right)\\
	+\partial_\varrho\left(\big({\bar \bu}+ \bu-\kappa\tfrac{\nabla_\bx  h}{\bar h+ h}\big)\cdot\nabla_\bx  w\right)
	+\Big(\kappa\Delta_\bx  h - \nabla_\bx\cdot\big((\bar h+ h)({\bar \bu}+\bu)\big) \Big)(\nabla_\bx\cdot \bu) \\
	\textstyle+\left( \kappa\nabla_\bx\Delta_\bx  \cH -\nabla_\bx\int_{\varrho}^{\rho_1}   \nabla_\bx\cdot((\bar h+ h)({\bar \bu}+\bu)) \dd\varrho'   \right)\cdot({\bar \bu}'+\partial_\varrho\bu).
\end{multline*}
Using that $\de_\varrho\nabla_\bx\cH =-\nabla_\bx h$ we can rewrite the left-hand side in a compact formulation as
\[{\rm (LHS)}=\frac1\mu \begin{pmatrix} \sqrt\mu\nabla_\bx \\ \de_\varrho  \end{pmatrix}\cdot \left(
\begin{pmatrix}
	\frac{\bar h+h}{\varrho}\Id& \frac{\sqrt\mu\nabla_\bx \cH}{\varrho} \\
	\frac{\sqrt\mu\nabla_\bx^\top \cH}{\varrho}& \frac{1+\mu |\nabla_\bx \cH|^2}{\varrho (\bar h+h)}
\end{pmatrix}
\begin{pmatrix} \sqrt\mu\nabla_\bx \\ \de_\varrho  \end{pmatrix} (P+ \Peq)\right),
\]
with $ \Peq:=\int_{\rho_0}^{\varrho}  \varrho' \bar h(\varrho') \, \dd\varrho'$.
As for the right-hand side, we denote 
\begin{equation}\label{def:ustar}
	\bu_\star:= -\kappa\frac{\nabla_\bx  h}{\bar h+ h},
\end{equation}
and we infer
\begin{multline*}
	{\rm (RHS)}=-(\bar h+h) \nabla_\bx\cdot  \left(\big(({\bar \bu}+\bu+\bu_\star)\cdot\nabla_\bx \big)\bu \right)
	-(\nabla_\bx \cH)\cdot\left(\partial_\varrho \left(\big({\bar \bu}+\bu+\bu_\star\big)\cdot\nabla_\bx \bu \right) \right)\\
	+\partial_\varrho\left(\big({\bar \bu}+ \bu+\bu_\star\big)\cdot\nabla_\bx  w\right)
	- \nabla_\bx\cdot((\bar h+ h)({\bar \bu}+\bu+\bu_\star)) (\nabla_\bx\cdot \bu) \\
	\textstyle\quad-\left(\nabla_\bx \int_{\varrho}^{\rho_1}   \nabla_\bx\cdot((\bar h+ h)({\bar \bu}+\bu+\bu_\star)) \dd\varrho'   \right)\cdot({\bar \bu}'+\partial_\varrho\bu).
\end{multline*}
Notice the identity (reminiscent of~\eqref{eq:h}) 
\begin{equation}\label{cancellation} \int_{\varrho}^{\rho_1}   \nabla_\bx\cdot((\bar h+ h)({\bar \bu}+\bu+\bu_\star)) \dd\varrho' = ({\bar \bu}+\bu+\bu_\star)\cdot\nabla_\bx \cH-w-w_\star,
\end{equation}
where 
\begin{equation}\label{def:wstar}
	w_\star:=\kappa\Delta_\bx \cH-\kappa\frac{\nabla_\bx h\cdot\nabla_\bx \cH}{\bar h+h},
\end{equation}
which is obtained by integrating with respect to $\varrho$ the divergence-free identities
\begin{align*}-(\bar h+h)\nabla_\bx\cdot \bu-(\nabla_\bx \cH)\cdot({\bar \bu}'+\partial_\varrho\bu)+\partial_\varrho  w&=0,\\
	-(\bar h+h)\nabla_\bx\cdot \bu_\star-(\nabla_\bx \cH)\cdot(\partial_\varrho\bu_\star)+\partial_\varrho  w_\star&=0,
\end{align*}
integrating by parts with respect to $\varrho$, using the boundary condition $w\vert_{\varrho=\rho_1}=0=w_\star\vert_{\varrho=\rho_1}$ and $h=-\de_\varrho \cH$. 
Hence the above can be equivalently written as
\begin{multline}\label{eq.RHS}
	{\rm (RHS)}=-(\bar h+h) \nabla_\bx\cdot  \left(\big({\bar \bu}+\bu+\bu_\star\big)\cdot\nabla_\bx \bu \right)
	-(\nabla_\bx \cH)\cdot\left(\partial_\varrho \left(\big({\bar \bu}+\bu+\bu_\star\big)\cdot\nabla_\bx \bu \right) \right)\\
	+\partial_\varrho\left(\big({\bar \bu}+ \bu+\bu_\star\big)\cdot\nabla_\bx  w\right)
	- \nabla_\bx\cdot((\bar h+ h)({\bar \bu}+\bu+\bu_\star)) (\nabla_\bx\cdot \bu) \\
	-\left(\nabla_\bx \left(({\bar \bu}+\bu+\bu_\star)\cdot\nabla_\bx \cH-w-w_\star \right)  \right)\cdot({\bar \bu}'+\partial_\varrho\bu).
\end{multline}
Taking into account the boundary conditions in~\eqref{eq:nonhydro-iso-recall}, we find that the pressure satisfies the following problem (recalling $ \Peq:=\int_{\rho_0}^{\varrho}  \varrho' \bar h(\varrho') \, \dd\varrho'$):
\begin{equation}\label{eq.Poisson}
	\left\{\begin{array}{l}
		\displaystyle\frac1\mu \begin{pmatrix} \sqrt\mu\nabla_\bx \\ \de_\varrho  \end{pmatrix}\cdot \left(
		\begin{pmatrix}
			\frac{\bar h+h}{\varrho}\Id& \frac{\sqrt\mu\nabla_\bx \cH}{\varrho} \\
			\frac{\sqrt\mu\nabla_\bx^\top \cH}{\varrho}& \frac{1+\mu |\nabla_\bx \cH|^2}{\varrho (\bar h+h)}
		\end{pmatrix}
		\begin{pmatrix} \sqrt\mu\nabla_\bx \\ \de_\varrho  \end{pmatrix}  (P+ \Peq) \right)= 
		{\rm (RHS)},\\
		P\big\vert_{\varrho=\rho_0}=0, \qquad (\de_\varrho P)\big\vert_{\varrho=\rho_1}=\rho_1 h\big\vert_{\varrho=\rho_1}.
	\end{array}\right.
\end{equation}
This boundary value problem corresponds to~\cite{Desjardins-Lannes-Saut}*{(7)} written in isopycnal coordinates, adapting the boundary conditions to the free-surface framework, and taking into account the effective transport velocities from eddy correlation. Following~\cite{Desjardins-Lannes-Saut}, we shall infer the existence and uniqueness as well as estimates on the pressure $P$ from the elliptic theory applied to the above boundary value problem, as stated below.

\begin{lemma}\label{L.Poisson}
	Let $s_0>d/2$, $s,k\in\NN$  such that $s\geq s_0+\frac52$ and $1\leq k\leq s$. Let $\bar M,M,h_\star>0$. There exists $C>0$ such that for any $\mu\in(0,1]$, and for any $\bar h,h,\cH$ satisfying
	the following bound \[ \norm{\bar h}_{W^{1\vee k-1,\infty}_\varrho}\leq \bar M,\quad \Norm{h}_{H^{s-1,1\vee k-1}}+\sqrt\mu\Norm{\nabla_\bx\cH}_{H^{s-1,1\vee k-1}} \leq M; \]
	(where we recall the notation $a\vee b=\max(a,b)$) and the stable stratification assumption
	\[ \inf_{(\bx,\varrho)\in\Omega} \bar h(\varrho)+h(\bx,\varrho) \geq {h_\star}; \]
	and for any $(Q_0,Q_1,\bR)\in   H^{s,k-1}(\Omega)^2\times H^{s,k}(\Omega)^{d+1}$, there exists a unique $P\in H^{s+1,k+1}(\Omega)$ solution to 
	\begin{equation}\label{eq.Poisson-simple}
		\left\{\begin{array}{l}
			\nabla^\mu_{\bx,\varrho}\cdot\big( A^\mu \nabla^\mu_{\bx,\varrho} P\big) = Q_0+\sqrt\mu \Lambda Q_1 +\nabla^\mu_{\bx,\varrho}\cdot \bR
			\\
			P\big\vert_{\varrho=\rho_0}=0, \qquad \be_{d+1}\cdot (A \nabla^\mu_{\bx,\varrho} P)\big\vert_{\varrho=\rho_1}=\be_{d+1}\cdot \bR\big\vert_{\varrho=\rho_1}
		\end{array}\right.
	\end{equation}
	where we denote $\Lambda:=(\Id-\Delta_\bx)^{1/2}$, 
	\[ \nabla^\mu_{\bx,\varrho}:=\begin{pmatrix} \sqrt\mu\nabla_\bx \\ \de_\varrho  \end{pmatrix} \qquad ; \qquad  A^\mu:= \begin{pmatrix}
		\frac{\bar h+h}{\varrho}\Id& \frac{\sqrt\mu\nabla_\bx \cH}{\varrho} \\
		\frac{\sqrt\mu\nabla_\bx^\top \cH}{\varrho}& \frac{1+ \mu|\nabla_\bx \cH|^2}{\varrho (\bar h+h)}
	\end{pmatrix},\]
	and one has, denoting $\Norm{(Q_0,Q_1,\bR)}_{r,j}:=\Norm{Q_0}_{H^{r,j-1}} +\Norm{Q_1}_{H^{r,j-1}} +\Norm{\bR}_{H^{r,j}}$,
	\begin{multline}\label{ineq:est-P}
		\Norm{P}_{L^2(\Omega)} +\Norm{\nabla^\mu_{\bx,\varrho} P}_{H^{s,k}} \leq C\,\times \Big(\Norm{(Q_0,Q_1,\bR)}_{s,k}\\
		+ \big(  \Norm{  h}_{H^{s,k}}+\sqrt\mu\Norm{\nabla_\bx\cH}_{H^{s,k}}\big)\,  \Norm{(Q_0,Q_1,\bR)}_{s-1,1\vee k-1} \Big)  
	\end{multline}
	and, when $k\geq 2$,
	\begin{equation}\label{ineq:est-Plow}
		\Norm{P}_{L^2(\Omega)} +  \Norm{\nabla^\mu_{\bx,\varrho}P}_{H^{s-1,k-1}} \leq C 
		\Norm{(Q_0,Q_1,\bR)}_{s-1,k-1}.
	\end{equation}
\end{lemma}
\begin{proof} Testing~\eqref{eq.Poisson-simple} with $ P$, using integration by parts and the boundary conditions, we find
	\[
	-\int_{\RR^d} \int_{\rho_0}^{\rho_1}  A^\mu \nabla^\mu_{\bx,\varrho} P \cdot \nabla^\mu_{\bx,\varrho}  P \, \dd \varrho\, \dd \bx = \int_{\RR^d} \int_{\rho_0}^{\rho_1}  Q_0 {P} +  Q_1 (\sqrt\mu\Lambda {P}) - \tilde \bR\cdot \nabla_{\bx, \varrho}^\mu  P  \, \dd \varrho \, \dd \bx.
	\]
	For $(Q_0,Q_1,\bR)\in L^2(\Omega)^{2+d+1}$, the existence and uniqueness of a (variational) solution to~\eqref{eq.Poisson-simple} in the functional space
	\[ H^1_0(\Omega):=\{ P\in L^2(\Omega) \ : \ \nabla_{\bx,\varrho} P \in L^2(\Omega) ,  \  P\big\vert_{\varrho=\rho_0} = 0\}\]
	classically follows from the Lax-Milgram Lemma thanks to the boundedness and the coercivity of the matrix $A$
	(recall that $\bar h+h\geq h_\star>0$ and the embedding of Lemma~\ref{L.embedding}), and the Poincar\'e inequality
	\begin{equation}\label{eq.Poincare} \forall P\in H^1_0(\Omega), \qquad \Norm{P}_{L^2(\Omega)}^2 = \int_{\RR^d} \int_{\rho_0}^{\rho_1} \left|\int_{\rho_0}^{\varrho} \partial_{\varrho'} P\dd \varrho'\right|^2 \dd\varrho \dd\bx
		\leq (\rho_1-\rho_0)^2 \Norm{\de_\varrho P }_{L^2(\Omega)}^2,
	\end{equation}
	and we have 
	\begin{equation}
		\Norm{\nabla^\mu_{\bx,\varrho}P}_{L^2(\Omega)}  \lesssim  \Norm{Q_0}_{L^2(\Omega)} +\Norm{Q_1}_{L^2(\Omega)} +{\Norm{\bR}_{L^2(\Omega)}} .\label{eq.basic-elliptic}
	\end{equation}
	The desired regularity for  $(Q_0,Q_1,\bR)\in   H^{s,k-1}(\Omega)^2\times H^{s,k}(\Omega)^{d+1}$ is then deduced following the standard approach for elliptic equations (notice the domain is flat) from the estimates which we obtain below. For more details, we refer  for instance to~\cite{Lannesbook}*{Chapter~2} where a very similar elliptic problem is thoroughly studied. We now focus on the estimates, assuming {\em a priori} the needed regularity to justify the following computations.
	
	First, we provide an estimate for $\Norm{ \nabla_{\bx,\varrho} P}_{H^{r,0}(\Omega)}$ for $1\leq r\leq s$.
	One readily checks that  $P_r:=\Lambda^r P$ with $\Lambda^r:=(\Id-\Delta_\bx)^{r/2}$ satisfies~\eqref{eq.Poisson-simple} 
	with $Q_0\leftarrow \Lambda^r Q_0$, $Q_1\leftarrow \Lambda^r Q_1$ and $\bR\leftarrow \Lambda^r \bR-[\Lambda^r,A^\mu ]\nabla_{\bx,\varrho} P $. We focus now on the contribution of $\bP_r:=[\Lambda^r,A^\mu ]\nabla^\mu_{\bx,\varrho} P $.
	By continuous embedding (Lemma~\ref{L.embedding}) and commutator estimates (Lemma~\ref{L.commutator-Hs}), we have
	\[  \Norm{\bP_r}_{L^2(\Omega)}   \lesssim \Norm{\nabla_\bx A^\mu}_{H^{s_0+\frac12,1}}   \Norm{\nabla^\mu_{\bx,\varrho} P}_{H^{r-1,0}} + \left\langle  \Norm{\nabla_\bx A^\mu}_{H^{r-1,0}}   \Norm{\nabla^\mu_{\bx,\varrho} P}_{H^{s_0+\frac12,1}}  \right\rangle_{r>s_0+1} . \]
	Hence, using product (Lemma~\ref{L.product-Hs},\ref{L.product-Hsk}) and composition (Lemma~\ref{L.composition-Hs},\ref{L.composition-Hsk-ex}) estimates, we deduce
	\begin{equation} \label{est-P} \Norm{\bP_r}_{L^2(\Omega)}   \leq  C\,\Big( \Norm{ \nabla^\mu_{\bx,\varrho} P}_{H^{r-1,0}}
		+\left\langle\big(\Norm{h}_{H^{r,0}} + \sqrt\mu\Norm{\nabla_\bx \cH}_{H^{r,0}}\big)   \Norm{\nabla^\mu_{\bx,\varrho} P}_{H^{s_0+\frac12,1}}  \right\rangle_{r>s_0+1}  \Big)  
	\end{equation}
	with
	$ C= C(h_\star,\norm{\bar h}_{W^{1,\infty}_\varrho},\Norm{(h,\sqrt\mu\nabla_\bx\cH)}_{ H^{s_0+\frac32,1}} )  $.

	Plugging~\eqref{est-P} in~\eqref{eq.basic-elliptic} and using continuous embedding (Lemma~\ref{L.embedding}) and $s_0+\frac32\leq s-1$ yields 
	\begin{multline}\label{eq.P-in-Hs1} 
		\Norm{\nabla^\mu_{\bx,\varrho} P}_{H^{r,0}} \lesssim \Norm{Q_0}_{H^{r,0}} +\Norm{Q_1}_{H^{r,0}} +\Norm{\bR}_{H^{r,0}} + \Norm{ \nabla^\mu_{\bx,\varrho} P}_{H^{r-1,0}}\ \\
		+ \ C\, \big( \Norm{  h}_{H^{r,0}}+\Norm{\nabla_\bx\cH}_{H^{r,0}} \big) \left\langle   \Norm{\nabla_{\bx,\varrho } P}_{H^{s_0+\frac12,1}} \right\rangle_{r>s_0+1},
	\end{multline}
	where we denote, here and thereafter, $a\lesssim b$ for $a\leq Cb$ with
	\[ C=C(h_\star,\norm{\bar h}_{W^{1,\infty}_\varrho},\Norm{h}_{H^{s-1,1}},\sqrt\mu\Norm{\nabla_\bx \cH}_{H^{s-1,1}}  ) =C(h_\star,\bar M,M) . \]
	
	Next we provide an estimate for $\Norm{ \nabla_{\bx,\varrho} P}_{H^{r,1}(\Omega)}$ appearing in the above right-hand side. This term involves $\de_\varrho^2P$, which we control by rewriting~\eqref{eq.Poisson-simple} as
	\begin{equation}\label{eq.Poisson-varrho}
		\tfrac{1+ \mu|\nabla_\bx \cH|^2}{\varrho (\bar h+h)} \partial_\varrho^2 P = -\partial_\varrho\big( \tfrac{1+\mu |\nabla_\bx \cH|^2}{\varrho (\bar h+h)}\big) (\partial_\varrho P ) - \nabla^\mu_{\bx,\varrho}\cdot A_0^\mu \nabla^\mu_{\bx,\varrho} P  +Q_0+\sqrt\mu\Lambda  Q_1 +\nabla^\mu_{\bx,\varrho}\cdot \bR   =:\widetilde\bR
	\end{equation}
	where we denote
	\begin{align*}  \nabla^\mu_{\bx,\varrho}\cdot A_0^\mu \nabla^\mu_{\bx,\varrho} P  
		&:= \nabla^\mu_{\bx,\varrho}\cdot  \left( \begin{pmatrix}
			\frac{\bar h+h}{\varrho}\Id& \frac{\sqrt\mu\nabla_\bx \cH}{\varrho} \\
			\frac{\sqrt\mu\nabla_\bx^\top \cH}{\varrho}& 0
		\end{pmatrix}\nabla^\mu_{\bx,\varrho} P\right) \\
		&=  \frac\mu\varrho \nabla_\bx\cdot \big(h  \nabla_{\bx} P +(\nabla_\bx \cH)( \partial_\varrho P)\big)+\partial_\varrho \big(\frac\mu\varrho (\nabla_\bx \cH)\cdot(\nabla_\bx P)\big) .
	\end{align*}
	When estimating the above, we use product estimates (Lemma~\ref{L.product-Hs}) and then continuous embedding (Lemma~\ref{L.embedding}), treating differently terms involving $ \Delta_\bx P$ or $\nabla_\bx\de_\varrho P$: for instance 
	\begin{align*}
		\Norm{\Lambda^{r-1} (h \Delta_\bx P)}_{L^2(\Omega)} \lesssim \Norm{h}_{H^{s_0+\frac 12, 1}} \Norm{\Delta_\bx P}_{H^{r-1,0}} + \big\langle \Norm{h}_{H^{r-\frac 12,1}} \Norm{\Delta_\bx P}_{H^{s_0,0}} \big\rangle_{r-1> s_0};
	\end{align*}
	and terms involving only $\nabla_{\bx} P$ or $\de_\varrho P$: for instance
	\begin{align*}
		\Norm{\Lambda^{r-1}((\Delta_\bx \cH) (\de_\varrho P))}_{L^2(\Omega)} & \lesssim \Norm{\Delta_\bx \cH}_{H^{s_0+\frac12,1}} \Norm{\de_\varrho P}_{H^{r-1,0}}  + \big\langle \Norm{\Delta_\bx \cH}_{H^{r-1,0}} \Norm{\de_\varrho P}_{H^{s_0+\frac 12, 1}} \big\rangle_{r-1>s_0} .
	\end{align*}
	We infer, using Lemma~\ref{L.embedding}, $\mu\in(0,1]$ and $s_0+\frac32\leq s-1$, that for any $1\leq r\leq s$,
	
	\begin{multline}\label{eq.P-in-Hs2} 
		\Norm{\partial_\varrho^2 P}_{H^{r-1,0}}\lesssim \Norm{Q_0}_{H^{r-1,0}} +\Norm{Q_1}_{H^{r,0}} +\Norm{\bR}_{H^{r,1}} + \Norm{ \nabla^\mu_{\bx,\varrho} P}_{H^{r,0}}   \\
		+\left( \Norm{  h}_{H^{r,1}}+\sqrt\mu\Norm{\nabla_\bx\cH}_{H^{r,1}}\right)\left\langle   \Norm{\nabla^\mu_{\bx,\varrho } P}_{H^{s_0+1,1}} \right\rangle_{r>s_0+1}.
	\end{multline}
	By combining~\eqref{eq.P-in-Hs1} and~\eqref{eq.P-in-Hs2} we obtain
	\begin{multline*}
		\Norm{\nabla^\mu_{\bx,\varrho} P}_{H^{r,1}} \lesssim \Norm{Q_0}_{H^{r,0}} +\Norm{Q_1}_{H^{r,0}} +\Norm{\bR}_{H^{r,1}} + \Norm{ \nabla^\mu_{\bx,\varrho} P}_{H^{r-1,0}}\ \\
		+ \ C\,  {\left\langle   \left(  \Norm{  h}_{H^{r,1}}+\sqrt\mu \Norm{\nabla_\bx\cH}_{H^{r,1}}\right) \Norm{\nabla^{\mu}_{\bx,\varrho } P}_{H^{s_0+1,1}} \right\rangle_{r>s_0+1} }
	\end{multline*}
	which, after finite induction on $1\leq r\leq s$ and using~\eqref{eq.basic-elliptic} for the initialization, yields
	\begin{multline}\label{eq.k=1}
		\Norm{\nabla^\mu_{\bx,\varrho} P}_{H^{r,1}} \lesssim \Norm{Q_0}_{H^{r,0}} +\Norm{Q_1}_{H^{r,0}} +\Norm{\bR}_{H^{r,1}}\\
		+ \left(  \Norm{  h}_{H^{r,1}}+\sqrt\mu\Norm{\nabla_\bx\cH}_{H^{r,1}}\right)\,  \times \left\langle\Norm{Q_0}_{H^{r-1,0}} +\Norm{Q_1}_{H^{r-1,0}} +\Norm{\bR}_{H^{r-1,1}}
		\right\rangle_{r>s_0+1}   .
	\end{multline}
	This, together with~\eqref{eq.Poincare}, proves~\eqref{ineq:est-P} when $k=1$.
	\medskip
	
	We now proceed to estimate higher $\varrho$-derivatives. In what follows, we denote
	\[ C=C(h_\star,\norm{\bar h}_{W^{k-1,\infty}_\varrho},\Norm{h}_{H^{s-1,k-1}},\sqrt\mu\Norm{\nabla_\bx \cH}_{H^{s-1,k-1}}  ) =C(h_\star,\bar M,M) . \]
	Let $2\leq j\leq k$. By definition, and using $\mu\in(0,1]$, we have
	\[ \Norm{\nabla^\mu_{\bx,\varrho }P}_{H^{s,j}} \leq \Norm{\nabla^\mu_{\bx,\varrho }P}_{H^{s,j-1}} +\Norm{\de_\varrho\nabla^\mu_{\bx,\varrho }P}_{H^{s-1,j-1}} \lesssim \Norm{\nabla^\mu_{\bx,\varrho }P}_{H^{s,j-1}}  +\Norm{\de_\varrho^2 P}_{H^{s-1,j-1}}  . \]
	We shall also use, when $j\leq k-1$, the corresponding estimate
	\[ \Norm{\nabla^\mu_{\bx,\varrho }P}_{H^{s-1,j}}  \lesssim \Norm{\nabla^\mu_{\bx,\varrho }P}_{H^{s-1,j-1}} +\Norm{\de_\varrho^2 P}_{H^{s-2,j-1}}. \]
	By using~\eqref{eq.Poisson-varrho} (according to which $\de_\varrho^2 P=\tfrac{\varrho (\bar h+h)}{1+ \mu|\nabla_\bx \cH|^2} \tilde \bR$), and since $1\leq j-1 \leq s-1$, using Lemma~\ref{L.product-Hsk} and Lemma~\ref{L.composition-Hsk} yields
	\[ \Norm{\partial_\varrho^{2} P}_{H^{s-1,j-1}}  \lesssim \Norm{\tfrac{\varrho (\bar h+h)}{1+ \mu|\nabla_\bx \cH|^2} }_{H^{s-1,j-1}}\Norm{ \widetilde\bR  }_{H^{s-1,j-1}} \leq C \Norm{ \widetilde\bR  }_{H^{s-1,j-1}}.\]
	If moreover $j\leq k-1\leq s-1$, then
	\[ \Norm{\partial_\varrho^{2} P}_{H^{s-2,j-1}}   \lesssim \Norm{\tfrac{\varrho (\bar h+h)}{1+\mu |\nabla_\bx \cH|^2} }_{H^{s-2,j-1}}\Norm{ \widetilde\bR  }_{H^{s-2,j-1}} \leq C \Norm{ \widetilde\bR  }_{H^{s-2,j-1}}.\]
	Applying Lemma~\ref{L.product-Hsk} and Lemma~\ref{L.composition-Hsk} to $\widetilde\bR$ defined in~\eqref{eq.Poisson-varrho}, we obtain
	\begin{multline*}\Norm{ \widetilde\bR  }_{H^{s-1,j-1}} \leq \Norm{ Q_0  }_{H^{s-1,j-1}}   +  \Norm{ Q_1  }_{H^{s,j-1}}   + \Norm{ \bR  }_{H^{s,j}} \\
		+ C \Norm{\nabla^\mu_{\bx,\varrho} P}_{H^{s,j-1}} + C\, \times\left(  \Norm{  h}_{H^{s,j}}+\sqrt\mu\Norm{\nabla_\bx\cH}_{H^{s,j}}\right) \Norm{\nabla^\mu_{\bx,\varrho} P}_{H^{s-1,j-1}} 
	\end{multline*}
	and, if moreover $j\leq k-1\leq  s-1$,
	\[ \Norm{ \widetilde\bR  }_{H^{s-2,j-1}} \leq \Norm{ Q_0  }_{H^{s-2,j-1}}  + \Norm{ Q_1  }_{H^{s-1,j-1}}  + \Norm{ \bR  }_{H^{s-1,j}} + C \Norm{\nabla^\mu_{\bx,\varrho} P}_{H^{s-1,j-1}} .
	\]
	From the second set of inequalities,~\eqref{eq.k=1} with $r=s-1$ and finite induction on $2\leq j\leq k-1$ we infer
	\[
	\Norm{\nabla^\mu_{\bx,\varrho }P}_{H^{s-1,j}}  \leq C \, \big( \Norm{ Q_0  }_{H^{s-1,j-1}}  +\Norm{ Q_1  }_{H^{s-1,j-1}}  +  \Norm{ \bR  }_{H^{s-1,j}}  \big).\]
	Then, from the first set of inequalities,~\eqref{eq.k=1} with $r=s$ and the previous result, we infer by finite induction on $2\leq j\leq k$
	\begin{multline*}
		\Norm{\nabla^\mu_{\bx,\varrho }P}_{H^{s,j}}  \leq C \, \big(
		\Norm{Q_0}_{H^{s,j-1}} + \Norm{Q_1}_{H^{s,j-1}} +\Norm{\bR}_{H^{s,j}}\big) \\
		+ C\,  \big(\Norm{  h}_{H^{s,j}}+\sqrt\mu\Norm{\nabla_\bx\cH}_{H^{s,j}}\big)\,   \big(
		\Norm{Q_0}_{H^{s-1,j-2}} + \Norm{Q_1}_{H^{s-1,j-2}} +\Norm{\bR}_{H^{s-1,j-1}}\big)   .
	\end{multline*}
	The result is proved.
\end{proof}

We now apply Lemma~\ref{L.Poisson} to obtain several estimates on the solution to~\eqref{eq.RHS}-\eqref{eq.Poisson}.

\begin{corollary}\label{C.Poisson}
	Let $s_0>d/2$, $s,k\in\NN$  such that $s\geq s_0+\frac52$ and $2\leq k\leq s$. Let $\bar M,M,h_\star>0$. There exists $C>0$ such that for any $\mu\in(0,1]$ and $\kappa\in\RR$, for any $(\bar h,\bar\bu)\in W^{k,\infty}((\rho_0,\rho_1))^{1+d}$, and for any $(h,\bu,w)\in H^{s+1,k}(\Omega)\times H^{s,k}(\Omega)^d\times H^{s,k-1}(\Omega) $ satisfying (denoting $\cH(\cdot,\varrho):=\int_\varrho^{\rho_1}h(\cdot,\varrho')\dd\varrho'$)
	\begin{itemize}
		\item the following bounds \[ \norm{\bar h}_{W^{k,\infty}_\varrho}+\norm{\bar \bu'}_{W^{k-1,\infty}_\varrho}\leq \bar M,\]  \[\Norm{h}_{H^{s-1,k-1}}+\Norm{\nabla_\bx\cH}_{H^{s-1,k-1}} +\Norm{\bu}_{H^{s,k}} +\sqrt\mu\Norm{w}_{H^{s,k-1}}\leq M; \]
		\item the stable stratification assumption
		\[ \inf_{(\bx,\varrho)\in\Omega} \bar h(\varrho)+h(\bx,\varrho) \geq {h_\star}; \]
		\item the incompressibility condition
		\[-(\bar h+h)\nabla_\bx\cdot \bu-(\nabla_\bx \cH)\cdot({\bar \bu}'+\partial_\varrho\bu)+\partial_\varrho  w=0,\] 
	\end{itemize}
	there exists a unique solution $P\in H^{s+1,k+1}(\Omega)$ to~\eqref{eq.RHS}-\eqref{eq.Poisson} and one has
	\begin{multline}\label{ineq:est-P-hydro}
		\Norm{P}_{L^2(\Omega)}  +\Norm{\nabla^\mu_{\bx,\varrho}P}_{H^{s,k}}  \leq C \, (1+\Norm{  h}_{H^{s,k}}+\sqrt\mu\Norm{\nabla_\bx\cH}_{H^{s,k}}) \\
		\times \left(
		\Norm{h}_{H^{s,k}} + \sqrt\mu\Norm{\nabla_\bx \cH}_{H^{s,k}}+ \Norm{(\bu,\bu_\star)}_{H^{s,k}} +\sqrt\mu\Norm{(w,w_\star)}_{H^{s,k-1}} \right)
	\end{multline}
	where we recall the notations $\bu_\star:=-\kappa\frac{\nabla_\bx  h}{\bar h+ h} $ and  $w_\star:=\kappa\Delta_\bx \cH-\kappa\frac{\nabla_\bx h\cdot\nabla_\bx \cH}{\bar h+h}$.
	
	Moreover, decomposing
	\[ P=\Ph+\Pnh, \qquad \Ph:=\int_{\rho_0}^{\varrho}  \varrho'  h(\cdot,\varrho') \, \dd\varrho',\]
	we have 
	\begin{multline}\label{ineq:est-Psmall-nonhydro}
		\Norm{\Pnh}_{L^2(\Omega)}  + \Norm{\nabla^\mu_{\bx,\varrho} \Pnh}_{H^{s-1,k-1}}  \leq C\sqrt\mu \Big(\Norm{ \nabla_\bx \cH}_{H^{s-1,k-1}} + \norm{\cH\big\vert_{\varrho=\rho_0}}_{H^{s}_\bx}\\
		+ \Norm{(\bu,\bu_\star)}_{H^{s,k}} +\sqrt\mu\Norm{(w,w_\star)}_{H^{s,k-1}} \Big),
	\end{multline}
	and, setting $\Lambda^\mu=1+\sqrt\mu|D|$,
	\begin{multline}\label{ineq:est-Ptame-nonhydro}
		\Norm{\Pnh}_{L^2(\Omega)}  +	\Norm{\nabla^\mu_{\bx,\varrho} \Pnh}_{H^{s-1,k-1}}  \leq C\,\mu\, \Big( \Norm{ (\Lambda^\mu)^{-1} \nabla_\bx \cH}_{H^{s,k-1}} + \norm{(\Lambda^\mu)^{-1} \cH\big\vert_{\varrho=\rho_0}}_{H^{s+1}_\bx} \\
		+ \big( \norm{\bar \bu'}_{W^{k-1,\infty}_\varrho}+\Norm{\bu}_{H^{s,k}}\big)\big( \Norm{(\bu,\bu_\star)}_{H^{s,k}}
		+\Norm{( w,w_\star)}_{H^{s,k-1}}\big) + \Norm{\bu_\star}_{H^{s,k}}  \Norm{w}_{H^{s,k-1}}\Big).	
	\end{multline}
\end{corollary}
\begin{proof}
	In view of Lemma~\ref{L.Poisson}, we shall first estimate {\rm (RHS)}, defined in~\eqref{eq.RHS}. We decompose
	\[ {\rm (RHS)} = R_1+R_2\]
	where $R_1$ is constituted by terms involving maximum one derivative on $h, \cH, \bu, \bu_\star, w, w_\star$, while
	\begin{multline*}
		R_2:= -( \bar h+h)   \left(\big({\bar \bu}+\bu+\bu_\star\big)\cdot\nabla_\bx (\nabla_\bx\cdot\bu) \right)
		-(\nabla_\bx \cH)\cdot\left(\big({\bar \bu}+\bu+\bu_\star\big)\cdot\nabla_\bx \partial_\varrho  \bu  \right)
		\\
		+\big({\bar \bu}+ \bu+\bu_\star\big)\cdot\nabla_\bx  \partial_\varrho w-\left(({\bar \bu}+\bu+\bu_\star)\cdot\nabla_\bx (\nabla_\bx  \cH)  \right)\cdot({\bar \bu}'+\partial_\varrho\bu).
	\end{multline*}
	Appealing to the incompressibility condition
	\[-(\bar h+h)\nabla_\bx\cdot \bu-(\nabla_\bx \cH)\cdot({\bar \bu}'+\partial_\varrho\bu)+\partial_\varrho  w=0,\]
	we have simply
	\[ R_2=  \left(\big({\bar \bu}+\bu+\bu_\star\big)\cdot(\nabla_\bx h)   \right) (\nabla_\bx\cdot\bu). \]
	As a matter of fact, this term compensates with 
	the second addend of $R_1$, so that contributions from $h$ are not differentiated. By inspecting the remaining terms and
	using Lemma~\ref{L.product-Hsk}, we infer that for any $r\geq s_0+1/2$ and $1\leq j\leq r \le s-1$,
	\begin{multline} \label{est-RHS}\Norm{{\rm (RHS)}}_{H^{r,j}}  \lesssim (1+\norm{\bar h}_{W^{j,\infty}_\varrho}+\Norm{h}_{H^{r,j}}+\Norm{\nabla_\bx \cH}_{H^{r,j}})\\
		\times\left( \big( \norm{\bar \bu'}_{W^{j,\infty}_\varrho}+\Norm{\bu}_{H^{r+1,j+1}}\big) \big(  \Norm{(\bu,\bu_\star)}_{H^{r+1,j+1}} +\Norm{(w,w_\star)}_{H^{r+1,j}} \big)+\Norm{w}_{H^{r+1,j}}\Norm{\bu_\star}_{H^{r+1,j+1}}\right).
	\end{multline}
	
	Owing to the fact that contributions from $(h,\cH,w,w^\star)$ to ${\rm (RHS)}$ are affine, and $\mu\in(0,1]$, we have
	for any $r\geq s_0+3/2 $ and $2\leq j\leq r\leq s$
	\begin{multline} \label{est-Q} \sqrt\mu\Norm{{\rm (RHS)}}_{H^{r-1,j-1}}  \lesssim (1+\norm{\bar h}_{W^{j-1,\infty}_\varrho}+\sqrt\mu\Norm{h}_{H^{r-1,j-1}}+\sqrt\mu\Norm{\nabla_\bx \cH}_{H^{r-1,j-1}})\\
		\times \left(\big(\norm{\bar \bu'}_{W^{j-1,\infty}_\varrho}+\Norm{\bu}_{H^{r,j}}\big) \big(  \Norm{(\bu,\bu_\star)}_{H^{r,j}} +\sqrt\mu\Norm{(w,w_\star)}_{H^{r,j-1}} \big)+\sqrt\mu\Norm{w}_{H^{r,j-1}} \Norm{\bu_\star}_{H^{r,j}}\right).
	\end{multline}

	For the first estimate, we write~\eqref{eq.Poisson} as~\eqref{eq.Poisson-simple} with $Q_0=0$, $Q_1=\sqrt \mu\Lambda^{-1}{\rm (RHS)}$ and 
	\[ \bR = \begin{pmatrix} -\sqrt\mu \bar h \nabla_\bx \cH\\ \frac{h}{\bar h+h} (1+\mu |\nabla_\bx \cH|^2)-\mu |\nabla_\bx \cH|^2\end{pmatrix}\]
	where we used the identities
	\[
	\frac{1+\mu|\nabla_{\bx} \cH|^2}{\varrho (\bar h + h)} \de_\varrho \Peq = \frac{\bar h (1+\mu|\nabla_{\bx} \cH|^2)}{ \bar h + h}=1+\mu|\nabla_\bx \cH|^2-\frac{h}{\bar h+h} (1+\mu|\nabla_\bx \cH|^2).
	\]
	Product (Lemma~\ref{L.product-Hsk}) and composition (Lemma~\ref{L.composition-Hsk-ex})  estimates yield if $r\geq s_0+1/2 $ and $1\leq j\leq r\leq s-1$
	\begin{equation}\label{est-R0}
		\Norm{\bR}_{H^{r,j}}   \leq  C(h_\star,\norm{\bar h}_{W^{j,\infty}_\varrho},\Norm{h}_{ H^{r,j}},\sqrt\mu\Norm{\nabla_\bx \cH }_{H^{r,j}} ) \times \big( \Norm{h}_{H^{r,j}} + \sqrt\mu\Norm{\nabla_\bx \cH}_{H^{r,j}} \big);
	\end{equation}
	and, if $r\geq s_0+3/2 $ and $2\leq j\leq r\leq s$, using the tame estimates, we obtain
	\begin{equation}\label{est-R1}
		\Norm{\bR}_{H^{r,j}}   \leq  C(h_\star,\norm{\bar h}_{W^{j,\infty}_\varrho},\Norm{h}_{ H^{r-1,j-1}},\sqrt\mu\Norm{\nabla_\bx \cH }_{H^{r-1,j-1}} ) \times \big( \Norm{h}_{H^{r,j}} + \sqrt\mu\Norm{\nabla_\bx \cH}_{H^{r,j}} \big). 
	\end{equation}
	Plugging in~\eqref{ineq:est-P} the estimates~\eqref{est-Q}  and~\eqref{est-R1} with $(r,j)=(s,k)$, and~\eqref{est-R0} with $(r,j)=(s-1,k-1)$, yields~\eqref{ineq:est-P-hydro}.
	
	For the next set of estimates, we notice that, by~\eqref{eq.Poisson}, 
	$P-P_{\h}$
	satisfies~\eqref{eq.Poisson-simple} with $\be_{d+1}\cdot \bR\big\vert_{\varrho=\rho_1} = 0$ and $Q_0+\sqrt\mu\Lambda Q_1+\nabla^\mu_{\bx,\varrho}\cdot\bR=\mu{\rm (RHS)}+\nabla^\mu_{\bx,\varrho}\cdot\widetilde\bR$ where
	\[  \widetilde \bR := -\begin{pmatrix}
		\sqrt{\mu}  \varrho^{-1} (\bar h+h) \nabla_\bx \psi \\
		{\mu}  \varrho^{-1} (\nabla_\bx \cH) \cdot (\nabla_\bx \psi) \end{pmatrix}, \quad \psi:=\int_{\rho_0}^{\varrho} \cH(\cdot,\varrho')\dd \varrho'+\rho_0 \cH\big\vert_{\varrho=\rho_0} .\]
	Indeed, we have immediately $\be_{d+1}\cdot \widetilde\bR\big\vert_{\varrho=\rho_1} = 0  = \partial_\varrho \big( P-\Ph)\big\vert_{\varrho=\rho_1}$ and we infer
	\[ - \nabla_{\bx, \varrho}^\mu \cdot \widetilde \bR = \nabla_{\bx, \varrho}^\mu \cdot \big(A^\mu  \nabla^\mu_{\bx,\varrho} (\Peq+\Ph) 
	\big)\]
	from (integrating by parts as in~\eqref{eq:nablaphi-intro})
	\begin{equation}\label{eq.id-Phydro}
		\Ph:=\int_{\rho_0}^{\varrho}  \varrho'  h(\cdot,\varrho') \, \dd\varrho' =  -\varrho\cH+\int_{\rho_0}^{\varrho} \cH(\cdot,\varrho')\dd \varrho'+\rho_0 \cH\big\vert_{\varrho=\rho_0}.
	\end{equation}
	By product estimates (Lemma~\ref{L.product-Hsk}), we infer immediately  for any $r\geq s_0+1/2 $ and $1\leq j\leq r$
	\begin{equation}\label{est-tR0}\Norm{\widetilde\bR}_{H^{r,j}}  \lesssim \Big(\norm{\bar h}_{W^{j,\infty}_\varrho}+\Norm{  h}_{H^{r,j}}+\sqrt\mu\Norm{\nabla_\bx\cH}_{H^{r,j}}\Big) \, \Big(  \sqrt\mu\Norm{\nabla_\bx \cH}_{H^{r,j}} +\sqrt\mu \norm{\cH\big\vert_{\varrho=\rho_0}}_{H^{r+1}_\bx}\Big).
	\end{equation}
	Moreover, using 
	the identity
	\[ \nabla^\mu_{\bx,\varrho}\cdot\widetilde\bR = \mu \nabla_\bx\cH\cdot \de_\varrho(\varrho^{-1} \nabla_\bx\psi)+\mu\varrho^{-1}(\bar h+h)\nabla_\bx\cdot\nabla_\bx\psi\]
	we infer for any $r\geq s_0+1/2 $ and $1\leq j\leq r$
	\begin{equation}\label{est-tR2}\Norm{\nabla^\mu_{\bx,\varrho}\cdot\widetilde\bR}_{H^{r,j}}  \lesssim \mu \Big(\norm{\bar h}_{W^{j,\infty}_\varrho}+\Norm{  h}_{H^{r,j}}+\Norm{\nabla_\bx\cH}_{H^{r,j}}\Big) \\
		\, \Big(  \Norm{\nabla_\bx \cH}_{H^{r+1,j}} + \norm{\cH\big\vert_{\varrho=\rho_0}}_{H^{r+2}_\bx}\Big).
	\end{equation}
	Finally, recalling $\Lambda^\mu=1+\sqrt\mu|D|$ and introducing
	\[\widetilde\bR_0 := -\begin{pmatrix}
		\sqrt{\mu}  \varrho^{-1} (\bar h+  (\Lambda^\mu)^{-1} h) ( (\Lambda^\mu)^{-1} \nabla_\bx \psi) \\
		{\mu}  \varrho^{-1} ((\Lambda^\mu)^{-1}\nabla_\bx \cH) \cdot ((\Lambda^\mu)^{-1}\nabla_\bx \psi), \end{pmatrix}  \]
	proceeding as for~\eqref{est-tR0} and~\eqref{est-tR2} and using  $\Id- (\Lambda^\mu)^{-1}=\sqrt\mu|D| (\Lambda^\mu)^{-1}$ and $\Norm{(\Lambda^\mu)^{-1}}_{L^2_\bx\to L^2_\bx}=1$, for any $r\geq s_0+1/2 $ and $1\leq j\leq r$ one has
	\begin{multline}\label{est-tRtame}
		\Norm{\widetilde\bR-\widetilde\bR_0}_{H^{r,j}}+\Norm{\nabla^\mu_{\bx,\varrho}\cdot\widetilde\bR_0}_{H^{r,j}}  \lesssim  \mu\, \Big(\norm{\bar h}_{W^{j,\infty}_\varrho}+\Norm{  h}_{H^{r,j}}+\Norm{\nabla_\bx\cH}_{H^{r,j}}
		\Big)\\
		\times \Big(
		\Norm{(\Lambda^\mu)^{-1}\nabla_\bx \cH}_{H^{r+1,j}} + \norm{(\Lambda^\mu)^{-1}\cH\big\vert_{\varrho=\rho_0}}_{H^{r+2}_\bx}\Big).
	\end{multline}
	We obtain~\eqref{ineq:est-Psmall-nonhydro}, by setting $Q_0=\mu{\rm (RHS)}$ and $Q_1=0$, and plug in~\eqref{ineq:est-Plow} the estimates~\eqref{est-Q}  with $(r,j)=(s,k)$, and~\eqref{est-tR0} with $(r,j)=(s-1,k-1)$. 
	For~\eqref{ineq:est-Ptame-nonhydro}, we set instead $Q_0=\mu{\rm (RHS)}+\nabla^\mu_{\bx,\varrho}\cdot\widetilde\bR_0$ and $Q_1=0$, $\bR=\widetilde\bR-\widetilde\bR_0$, and plug in~\eqref{ineq:est-Plow} the estimates~\eqref{est-RHS} and~\eqref{est-tRtame} with $(r,j)=(s-1,k-1)$.
\end{proof}

\subsection{Small-time well-posedness}\label{S.Nonhydro:small-time}
We infer small-time existence and uniqueness of regular solutions to the Cauchy problem associated with the non-hydrostatic problem,~\eqref{eq:nonhydro-iso-recall}, proceeding as for the hydrostatic system in Section~\ref{S.Hydro}, that is considering
the system as the combination of a transport-diffusion equation and transport equations, coupled through order-zero source terms (by the estimate \eqref{ineq:est-P-hydro} in Corollary~\ref{C.Poisson}). Specifically, we rewrite~\eqref{eq:nonhydro-iso-recall} as
\begin{equation}\label{eq:nonhydro-iso-redef}
	\begin{aligned}
		\partial_t  h+\nabla_{ \bx} \cdot\big((\bar h +h)(\bar\bu +  \bu)\big)&= \kappa \Delta_{ \bx}  h,\\
		\partial_t  \bu+\big((\bar \bu +  \bu - \kappa\tfrac{\nabla_{ \bx}  h}{ \bar h+  h})\cdot\nabla_\bx\big)  \bu+ \frac1\varrho \nabla_{ \bx}  P+ \big(1+\frac{\de_\varrho  P-\varrho h}{\varrho(\bar h+ h)} \big)\nabla_{ \bx}  \cH &=0,\\
		\partial_t  w+\big(\bar \bu +  \bu - \kappa\tfrac{\nabla_{ \bx}   h}{ \bar h+ h}\big)\cdot\nabla_{ \bx}   w-\frac1{\mu} \frac{\de_\varrho  P-\varrho h}{\varrho(\bar h+ h)}&=0,
	\end{aligned}
\end{equation}
where $\cH(\cdot, \varrho)=\int_{\varrho}^{\rho_1}  h(\cdot, \varrho')\dd\varrho'$ and $P$ is defined by Corollary~\ref{C.Poisson}.   Systems~\eqref{eq:nonhydro-iso-recall}  and~\eqref{eq:nonhydro-iso-redef} are equivalent (for sufficiently regular data) by the computations of Section~\ref{S.Nonhydro:Poisson}, and in particular regular solutions to~\eqref{eq:nonhydro-iso-redef} satisfy 
the boundary condition
$  w\big|_{\varrho=\rho_1}=0$ and 
the incompressibility constraint
\begin{equation}\label{eq:incompressibility-redef}
	(\bar h +  h) \nabla_{\bx} \cdot  \bu + \nabla_{\bx} \cH \cdot ({\bar \bu}'+\partial_\varrho\bu) -\partial_\varrho  w=0
\end{equation}
provided these identities hold initially.

\begin{proposition}\label{P.NONHydro-small-time} 
	Let $s_0>d/2$, $s,k\in\NN$ such that $s\geq s_0+\frac {5} 2$ and $2\leq k\leq s$. Let $h_\star,\bar M,M,\mu,\kappa>0$ and $C_0>1$. There exists $T>0$ such that for any  $(\bar h,\bar\bu)\in W^{k,\infty}((\rho_0,\rho_1))^{1+d}$, and for any initial data $(h_0,\bu_0,w_0)\in H^{s+1,k}(\Omega)\times H^{s,k}(\Omega)^d\times H^{s,k}(\Omega) $ satisfying
	\begin{itemize}
		\item the following bounds (where $\cH_0(\cdot, \varrho):=\int_{\varrho}^{\rho_1}  h_0(\cdot, \varrho')\dd\varrho'$) 
		\[ \norm{\bar h}_{W^{k,\infty}_\varrho}+\norm{\bar \bu'}_{W^{k-1,\infty}_\varrho}\leq \bar M\] 
		\[ M_0:=\Norm{h_0}_{H^{s,k}}+ \Norm{\nabla_\bx \cH_0}_{H^{s,k}}+\Norm{\bu_0}_{H^{s,k}} +\Norm{w_0}_{H^{s,k}}\leq M, \]
		\item the stable stratification assumption
		\[ \inf_{(\bx,\varrho)\in\Omega} \bar h(\varrho)+h_0(\bx,\varrho) \geq h_\star, \]
		\item the boundary condition $w_0|_{\varrho=\rho_1}=0$ and the incompressibility condition
		\[-(\bar h+h_0)\nabla_\bx\cdot \bu_0-(\nabla_\bx \cH_0)\cdot({\bar \bu}'+\partial_\varrho\bu_0)+\partial_\varrho  w_0=0,\] 
	\end{itemize}
	there exists a unique $(h,\bu,w)\in \cC^0(0,T;H^{s,k}(\Omega)^{2+d})$ and $P\in L^2(0,T;H^{s+1,k+1}(\Omega))$ solution to~\eqref{eq:nonhydro-iso-redef}
	satisfying the initial data $(h,\bu,w)\big\vert_{t=0}=(h_0,\bu_0,w_0)$. Moreover, the conditions $w\vert_{\varrho=\rho_1}=P\vert_{\varrho=\rho_0}=0$ and the incompressibility condition \eqref{eq:incompressibility-redef} hold on $[0,T]$ (and hence the solution satisfies \eqref{eq:nonhydro-iso-recall}) and one has $\cH\in L^\infty(0,T;H^{s+1,k}(\Omega))$ and $(h,\nabla_\bx\cH)\in   L^2(0,T;H^{s+1,k}(\Omega))$ and
	\begin{multline*}\cF_{s,k,T}:= \Norm{h}_{L^\infty(0,T;H^{s,k})}+\Norm{\nabla_\bx \cH}_{L^\infty(0,T;H^{s,k})}+\Norm{\bu}_{L^\infty(0,T;H^{s,k})}+\Norm{w}_{L^\infty(0,T;H^{s,k})}\\ + c_0 \kappa^{1/2} \Norm{ h}_{L^2(0,T;H^{s,k})} + c_0\kappa^{1/2} \Norm{\nabla_\bx^2\cH}_{L^2(0,T;H^{s,k})} \leq C_0M_0
	\end{multline*}
	with $c_0$ a universal constant.
\end{proposition}
\begin{proof}
	Since a very similar proof has been detailed in the hydrostatic framework in Section~\ref{S.Hydro}, we will only briefly sketch the main arguments.
	As aforementioned, thanks to Corollary~\ref{C.Poisson}, we may consider the contributions of the pressure as zero-order source terms in the energy space displayed in the statement, and~\eqref{eq:nonhydro-iso-redef} is then interpreted as a standard set of evolution equations. 
	We now explain how to infer the necessary bounds on all contributions to $\cF_{s,k,T}$, assuming enough regularity.  
	
	The desired control of $ \Norm{h}_{L^\infty(0,T;H^{s,k})}+ c_0 \kappa^{1/2} \Norm{\nabla_\bx h}_{L^2(0,T;H^{s,k})} $ is a direct consequence of the first equation of~\eqref{eq:nonhydro-iso-redef}, and the regularization properties of the heat semigroup already summoned in Proposition~\ref{P.WP-nu}. The corresponding control of  $ \Norm{\nabla_x \cH}_{L^\infty(0,T;H^{s,k})}+ c_0 \kappa^{1/2} \Norm{\nabla_\bx^2 \cH}_{L^2(0,T;H^{s,k})} $ demands an additional structure. We recall  (see~\eqref{eq:h} or~\eqref{cancellation}) that
	by the identity~\eqref{eq:incompressibility-redef} and integrating the first equation of~\eqref{eq:nonhydro-iso-redef}, one has 
	\begin{equation} \label{eq.id-eta}
		\partial_t \cH +(\bar \bu+\bu)\cdot\nabla_\bx\cH -w=\kappa\Delta_\bx\cH.
	\end{equation}
	By the regularization properties of the heat semigroup, we infer (with $c_0$ a universal constant)
	\[\Norm{\nabla_\bx\cH }_{L^\infty(0,T;H^{s,k})}+c_0\kappa^{1/2} \Norm{\nabla_\bx^2 \cH}_{L^2(0,T;H^{s,k})} \leq  \Norm{\nabla_\bx\cH_0}_{H^{s,k}} +\frac1{c_0\kappa^{1/2}} \Norm{(\bar \bu+\bu)\cdot\nabla_\bx\cH -w}_{L^2(0,T;H^{s,k})},\]
	and the right-hand side is estimated by product estimates (Lemma~\ref{L.product-Hsk}).
	Finally, the desired {\em a priori} estimates on $\Norm{\bu}_{L^\infty(0,T;H^{s,k})}$ and $\Norm{w}_{L^\infty(0,T;H^{s,k})}$ for sufficiently regular solutions follow by the energy method (that is integrating by parts in the variable $\bx$) on the second and third equations of~\eqref{eq:nonhydro-iso-redef}, which can be seen as transport equations with source terms.  More precisely, by \eqref{ineq:est-P-hydro} in Corollary~\ref{C.Poisson}, we have the existence and uniqueness of $P\in L^2(0,T;H^{s+1,k+1}(\Omega))$, satisfying the bound
	\[ \Norm{P}_{L^2(0,T;H^{s+1,k+1})} \leq C(h_\star,\mu,\kappa,\bar M,\cF_{s,k,T})\cF_{s,k,T}.\]
	Moreover, the advection velocity is controlled (using Lemma~\ref{L.embedding}, $s-2\geq s_0+\frac12$, $k\geq 1$) by
	\[ \Norm{\nabla\cdot\big(\bar \bu +  \bu - \kappa\tfrac{\nabla_{ \bx}   h}{ \bar h+ h}\big)}_{L^\infty(0,T;L^\infty(\Omega))}  \leq C(h_\star,\kappa,\cF_{s,k,T}),\]
	and using commutator (Lemma~\ref{L.commutator-Hsk}) and composition (Lemma~\ref{L.composition-Hsk-ex}) estimates, one has for any $f\in H^{s,k}(\Omega)$, and any $\balpha \in\NN^d$, $j\in\NN$ with $0\leq j\leq k$ and $|\balpha|+j\leq s$, 
	\[ \Norm{[\partial_\bx^\balpha \de_\varrho^j, \bar \bu +  \bu - \kappa\tfrac{\nabla_{ \bx}   h}{ \bar h+ h}] \nabla_\bx f}_{L^2(0,T;L^2(\Omega))}  \leq C(h_\star,\kappa,\bar M,\cF_{s,k,T}) \Norm{f}_{H^{s,k}}.\]
	It follows
	\[ \Norm{\bu}_{H^{s,k}}+\Norm{w}_{H^{s,k}}\leq \Big( \Norm{\bu_0}_{H^{s,k}}+\Norm{w_0}_{H^{s,k}}+C\sqrt T \Big)\exp(C T),\]
	with $C= C(h_\star,\mu,\kappa,\bar M,\cF_{s,k,T})$.
	
	Altogether, and using standard continuity arguments, we find that for any $C_0>1$ we can restrict the time $T=T(h_\star,\mu,\kappa,\bar M,C_0 M_0)>0$ so that all sufficiently regular solutions to~\eqref{eq:nonhydro-iso-redef} satisfy the bound $\cF_{s,k,T}\leq C_0 M_0$.
	We may infer the existence of solutions using for instance the parabolic regularization approach (see the closing paragraph of Section~\ref{S.Hydro}), and uniqueness is straightforward. This concludes the proof.
\end{proof}
\begin{remark}
	Proposition~\ref{P.NONHydro-small-time} does not provide any lower bound on the time of existence (and control) of solutions with respect to either $\mu\ll 1$ or $\kappa\ll 1$, hence the ``small-time'' terminology.
\end{remark}

\subsection{Quasi-linearization of the non-hydrostatic system}\label{S.Nonhydro:quasi}

In this section we extract the leading-order terms of the equations satisfied by the spatial derivatives of the solutions to system~\eqref{eq:nonhydro-iso-redef}. This will allow us to obtain improved energy estimates in the subsequent section. Notice that starting from here, our study is restricted to the situation $k=s$.
\begin{lemma}\label{lem.quasilin-nonhydro}
	Let  $s, k \in \NN$ such that $k=s> \frac 52+\frac d 2$ and $\bar M,M,h_\star>0$. Then there exists $C>0$ such that 
	for any $\mu,\kappa\in(0,1]$, and
	for any $(\bar h, \bar \bu) \in W^{k,\infty}((\rho_0,\rho_1)) \times W^{k+1,\infty}((\rho_0,\rho_1))$ satisfying
	\[  \norm{\bar h}_{W^{k,\infty}_\varrho } + \norm{\bar \bu'}_{W^{k,\infty}_\varrho }\leq \bar M \,;\]
	and any $(h, \bu,w)  \in L^\infty(0,T;H^{s,k}(\Omega)^{d+2})$ solution to~\eqref{eq:nonhydro-iso-redef} (with $P$ defined by Corollary~\ref{C.Poisson})
	with any  $T>0$ and satisfying  for almost every $t \in [0,T]$, 
	\[
	\Norm{h(t, \cdot)}_{H^{s-1, k-1}} + \Norm{\cH(t, \cdot)}_{H^{s,k}}+\norm{\cH\big\vert_{\varrho=\rho_0}(t, \cdot)}_{H^s_\bx}+\Norm{\bu(t, \cdot)}_{H^{s,k}} +\sqrt\mu\Norm{w(t, \cdot)}_{H^{s,k}}+\kappa^{1/2}\Norm{h(t, \cdot)}_{H^{s,k}}
	\le M\]
	(where $\cH(t,\bx,\varrho):=\int_{\varrho}^{\rho_1} h(t,\bx,\varrho')\dd\varrho'$) and 
	\[ \inf_{(t,\bx,\varrho)\in  (0,T)\times\Omega } \bar h(\varrho)+h(t, \bx,\varrho) \geq h_\star,\]
	the following results hold.
	
	Denote, for any multi-index $\balpha \in \NN^d$ and any $j\in\NN$ such that $|\balpha|+j\leq s$, $ h^{(\balpha,j)}=\de_\varrho^j\de_\bx^\balpha  h,\, \cH^{(\balpha,j)}=\de_\varrho^j\de_\bx^\balpha  \cH, \, \bu^{(\balpha,j)}=\de_\varrho^j\de_\bx^\balpha \bu,\, w^{(\balpha,j)}=\de_\varrho^j\de_\bx^\balpha w$, and $\Pnh^{(\balpha)}=\de_\varrho^j\de_\bx^\balpha  \Pnh$, with
	\[ \Pnh :=P-\Ph, \qquad \Ph:=\int_{\rho_0}^{\varrho}  \varrho'  h(\cdot,\varrho') \, \dd\varrho'.\]
	We have  
	\begin{subequations}
		\begin{equation}\label{eq.quasilin-j-nonhydro}
			\begin{aligned}
				\partial_t  \cH^{(\balpha,j)}+ (\bar \bu + \bu) \cdot\nabla_\bx \cH^{(\balpha,j)} -w^{(\balpha,j)}-\kappa\Delta_\bx   \cH^{(\balpha,j)} &=\widetilde R_{\balpha,j},\\
				\partial_t  \cH^{(\balpha,j)}+ (\bar \bu + \bu) \cdot\nabla_\bx \cH^{(\balpha,j)} +\Big\langle \int_\varrho^{\rho_1} (\bar \bu' + \de_{\varrho}\bu) \cdot\nabla_\bx \cH^{(\balpha,j)} \, \dd\varrho'\phantom{-\kappa\Delta_\bx   \cH^{(\balpha,j)}}\qquad  &\\
				+\int_{\varrho}^{\rho_1} (\bar h+h)\nabla_\bx \cdot\bu^{(\balpha,j)} \dd\varrho'\Big\rangle_{j=0}-\kappa\Delta_\bx   \cH^{(\balpha,j)}&=R_{\balpha,j},\\
				\partial_t\bu^{(\balpha,j)}+\big(({\bar \bu}+\bu-\kappa\tfrac{\nabla_\bx h}{\bar h+ h})\cdot\nabla_\bx\big) \bu^{(\balpha,j)}
				+ \Big\langle  \frac{\rho_0}{ \varrho }\nabla_\bx \cH^{(\balpha,j)}\big\vert_{\varrho=\rho_0} +\frac 1 \varrho\int_{\rho_0}^\varrho  \nabla_\bx  \cH^{(\balpha,j)} \dd\varrho'\Big\rangle_{j=0} &\\
				+\frac 1 \varrho{\nabla_\bx \Pnh^{(\balpha,j)}}  + \frac{\nabla_\bx \cH}{\varrho(\bar h + h)} \de_\varrho \Pnh^{(\balpha,j)} &=\bR^\nh_{\balpha,j},\\
				\sqrt\mu  \left( \partial_t w^{(\balpha,j)} + (\bar \bu + \bu - \kappa \tfrac{\nabla_\bx h}{\bar h + h} ) \cdot \nabla_\bx w^{(\balpha,j)}\right) - \frac1{\sqrt\mu}\frac{\de_\varrho \Pnh^{(\balpha,j)}}{ \varrho( \bar h + h)}  &= R^\nh_{\balpha,  j}, \\
				- \de_\varrho w^{(\balpha)} + (\bar h + h) \nabla_\bx \cdot \bu^{(\balpha,j)} + (\nabla_\bx \cH )\cdot ( \de_\varrho  \bu^{(\balpha,j)})&\\
				+   (\nabla_\bx \cdot \bu)h^{(\balpha,j)}  + (\bar \bu'+\de_\varrho \bu) \cdot \nabla_\bx  \cH^{(\balpha,j)} &=R^{\rm div}_{\balpha,j},
			\end{aligned}
		\end{equation}
		where for almost every $t\in[0,T]$, one has $(R_{\balpha,j}(t,\cdot),\bR^\nh_{\balpha,j}(t,\cdot), R^\nh_{\balpha, j}(t,\cdot), R^{\rm div}_{\balpha, j}(t,\cdot) )\in L^2(\Omega)^{d+3}$ and  $R_{\balpha,0}(t,\cdot)\in\cC((\rho_0,\rho_1);L^2(\RR^d))$, and
		\begin{equation}
			\begin{aligned}
				&\Norm{R_{\balpha,j}}_{L^2(\Omega) } +  \norm{R_{\balpha, 0}|_{\varrho=\rho_0}}_{L^2_\bx} +\Norm{\widetilde R_{\balpha,j}}_{L^2(\Omega) } + \Norm{R^{\rm div}_{\balpha, j}}_{L^2(\Omega)}    \leq  C\, M \,, \\
				&\Norm{ \bR^\nh_{\balpha,j}}_{L^2(\Omega)}+\Norm{ R^\nh_{\balpha, j}}_{L^2(\Omega)}
				\leq  C\, M \, \big(1  +\kappa\Norm{ \nabla_\bx  h}_{H^{s,k}}  \big) \\
				&\phantom{\Norm{ \bR^\nh_{\balpha,j}}_{L^2(\Omega)}+\Norm{ R^\nh_{\balpha, j}}_{L^2(\Omega)}
					\leq  \qquad}  + C\,  \big( \Norm{ h}_{H^{s,k}}+\sqrt\mu\Norm{\nabla_\bx \cH}_{H^{s,k}}  \big) \big(M+ \Norm{\bu_\star}_{H^{s,k}}+\sqrt\mu \Norm{w_\star}_{H^{s,k-1} }  \big) \,,
			\end{aligned}
			\label{eq.est-quasilin-j-nonhydro}
		\end{equation}
	\end{subequations}
	and
	\begin{subequations}
		\begin{equation}\label{eq.quasilin-j-h-nonhydro}
			\begin{aligned}
				\partial_t   h^{(\balpha,j)}+(\bar\bu+\bu)\cdot \nabla_\bx h^{(\balpha,j)}
				&=\kappa\Delta_\bx h^{(\balpha,j)}+r_{\balpha,j}+\nabla_{\bx} \cdot \br_{\balpha,j},
			\end{aligned}
		\end{equation}
		where  for almost every $t\in[0,T]$, one has $(r_{\balpha,j}(t,\cdot),\br_{\balpha,j}(t,\cdot))\in L^2(\Omega)^{1+d}$ and
		\begin{align}\label{eq.est-quasilin-j-h-nonhydro}
			\kappa^{1/2} \Norm{r_{\balpha,j}}_{L^2(\Omega) }  + \Norm{\br_{\balpha,j}}_{L^2(\Omega) } & \leq C\,M.
		\end{align}
	\end{subequations}
\end{lemma}
\begin{proof}
	Let us first point out that the estimates for $\norm{R_{\balpha, 0}|_{\varrho=\rho_0}}_{L^2_\bx}$, $ \Norm{R_{\balpha, j}}_{L^2(\Omega)}$, $\Norm{r_{\balpha,j}}_{L^2(\Omega) }$ and $ \Norm{\br_{\balpha,j}}_{L^2(\Omega) }$ have been stated and proved in  Lemma~\ref{lem:quasilinearization}. Thus we only need to focus on the other terms. 
	In the following, we denote $s_0=s-\frac52>\frac d2$.
	
	Using the identity  (already pointed out in~\eqref{eq.id-eta})
	\[
	\partial_t \cH +(\bar \bu+\bu)\cdot\nabla_\bx\cH -w=\kappa\Delta_\bx\cH,
	\]
	and  the commutator estimate in Lemma~\ref{L.commutator-Hsk}, we find immediately
	\[ \widetilde R_{\balpha,j} = [\de_\varrho^j \de_\bx^\balpha ,\bar \bu+\bu]\cdot\nabla_\bx\cH, \qquad \Norm{\widetilde R_{\balpha,j}}_{L^2(\Omega)} \lesssim \big(\norm{\bar\bu'}_{W^{k-1,\infty}_\varrho}+\Norm{\bu}_{H^{s,k}} \big) \Norm{\cH}_{H^{s,k}}.\]
	Using the decomposition $P=P_\h+P_\nh$ as in Corollary~\ref{C.Poisson} and the identity~\eqref{eq.id-Phydro} we have
	\begin{align*}
		\nabla_\bx \Ph:= \int_{\rho_0}^\varrho \varrho' \nabla_\bx  h (\cdot, \varrho') \, \dd \varrho' = - \varrho \nabla_\bx  \cH + \rho_0 \nabla_\bx  \cH |_{\varrho=\rho_0} + \int_{\rho_0}^\varrho \nabla_\bx  \cH(\cdot, \varrho') \, \dd \varrho',
	\end{align*}
	and hence the evolution equation for $ \bu$ reads 
	\begin{multline*}
		\de_t  \bu+ \big((\bar \bu + \bu - \kappa \tfrac{\nabla_\bx h}{\bar h+h} ) \cdot \nabla_\bx \big)\bu+ \frac{\rho_0 }{\varrho}\nabla_\bx  \cH |_{\varrho=\rho_0}+ \frac1{\varrho}\int_{\rho_0}^\varrho \nabla_\bx  \cH(\cdot, \varrho') \, \dd \varrho' \\
		+\frac1{\varrho}\nabla_\bx {{\Pnh}} +  \frac{\nabla_\bx \cH}{\varrho(\bar h +h)} \de_\varrho {{\Pnh}}= 0.
	\end{multline*}
	Differentiating $\balpha$ times with respect to $\bx$ and $j$ times with respect to $\varrho$ yields the corresponding equations in~\eqref{eq.quasilin-j-nonhydro}, with remainder terms
	\[
	\bR_{\balpha, j}^\nh:=\bR_{\balpha, j} - \big[ \de_\varrho^j \de_\bx^\balpha, \tfrac{\nabla_\bx \cH}{\varrho(\bar h + h)} \big] \de_\varrho  \Pnh,
	\]
	using the notation $ \bR_{\balpha, j}$ for the hydrostatic contributions introduced in Lemma~\ref{lem:quasilinearization}. The first
	addends have been estimated in Lemma~\ref{lem:quasilinearization},~\eqref{eq.est-quasilin} (when $j=0$) and~\eqref{eq.est-quasilin-j} (when $j\ge1$). We now estimate the second addend as follows.
	By the commutator estimate in Lemma~\ref{L.commutator-Hsk} with $k=s\geq s_0+3/2$, we have
	\[\Norm{ \big[ \de_\varrho^j \de_\bx^\balpha, \tfrac{\nabla_\bx \cH}{\varrho(\bar h + h)} \big] \de_\varrho  \Pnh}_{L^2(\Omega)}  \lesssim \Norm{\tfrac{\nabla_\bx \cH}{\varrho(\bar h + h)}}_{H^{s,k}} \Norm{\de_\varrho  \Pnh}_{H^{s-1, k-1}}.\]
	Then by tame product estimate Lemma~\ref{L.product-Hsk} and composition estimates in Lemma~\ref{L.composition-Hsk-ex},
	we have
	\[ \Norm{\tfrac{\nabla_\bx \cH}{\varrho(\bar h + h)}}_{H^{s,k}} \leq C(h_\star,\norm{\bar h}_{W^{k,\infty}},\Norm{h}_{H^{s-1,k-1}}) (\Norm{h}_{H^{s,k}} \Norm{\nabla_\bx \cH}_{H^{s-1,k-1}} +\Norm{\nabla_\bx \cH}_{H^{s,k}}) \]
	and there remains to use estimate~\eqref{ineq:est-Psmall-nonhydro} in Corollary~\ref{C.Poisson} to infer 
	\begin{multline*}\Norm{ \bR_{\balpha, j}^\nh}_{L^2(\Omega)} \leq C(h_\star,\bar M,M) \, M \, \big(1 +\kappa\Norm{ \nabla_\bx h}_{H^{s,k}} \big)\\
		+ C(h_\star,\bar M,M)\, \sqrt\mu\,  (\Norm{h}_{H^{s,k}} +\Norm{\nabla_\bx \cH}_{H^{s,k}}) \big(M+ \Norm{\bu_\star}_{H^{s,k}}+\sqrt\mu \Norm{w_\star}_{H^{s,k-1} }  \big).
	\end{multline*}
	Now consider
	\[
	R^\nh_{\balpha, j} := - \sqrt\mu [\de_\varrho^j \de_\bx^\balpha, \bar \bu + \bu] \cdot \nabla_\bx w + \kappa \sqrt\mu  \left[\de_\varrho^j  \de_\bx^\balpha, \tfrac{\nabla_\bx h}{\bar h+h} \right] \cdot \nabla_\bx w +  \tfrac1{\sqrt\mu}  [\de_\varrho^j  \de_\bx^\balpha, \tfrac{1}{\varrho(\bar h + h)}] \de_\varrho P.
	\]
	We have, by  Lemma~\ref{L.commutator-Hsk}  with $k=s\geq s_0+3/2$,
	\begin{align*}
		\sqrt\mu\Norm{[\de_\varrho^j\de_\bx^\balpha, \bar \bu+\bu] \cdot \nabla_\bx w}_{L^2(\Omega)}  \lesssim \sqrt\mu \big(\norm{\bar\bu'}_{W^{k-1,\infty}_\varrho} +\Norm{ \bu}_{H^{s,k}}\big) \Norm{\nabla_\bx w}_{H^{s-1,k-1}},
	\end{align*}
	and similarly, using tame product estimate Lemma~\ref{L.product-Hsk} and composition estimates in Lemma~\ref{L.composition-Hsk-ex} as above,
	\[
	\kappa \sqrt\mu \Norm{  \big[ \de_\varrho^j\de_\bx^\balpha, \tfrac{\nabla_\bx h}{\bar h+h} \big] \cdot \nabla_\bx w}_{L^2(\Omega)}  \leq \kappa\sqrt\mu\,  C(h_\star,\bar M,M)  (\Norm{h}_{H^{s,k}}^2 +\Norm{\nabla_\bx h}_{H^{s,k}}) \Norm{w}_{H^{s,k-1}} .
	\]
	and
	\[
	\tfrac1{\sqrt\mu} \Norm{  [\de_\varrho^j\de_\bx^\balpha, \tfrac1{\varrho(\bar h + h)}] \de_\varrho \Pnh}_{L^2(\Omega)} 
	\le \, \tfrac1{\sqrt\mu}  C(h_\star,\bar M,M) \Norm{h}_{H^{s,k}} \Norm{\de_\varrho \Pnh}_{H^{s-1,k-1}}.
	\]
	Collecting the above and using estimate~\eqref{ineq:est-Psmall-nonhydro} in Corollary~\ref{C.Poisson} yields 
	\begin{align*}\Norm{ R_{\balpha, j}^\nh}_{L^2(\Omega)} &\leq C(h_\star,\bar M,M) \, M \, \big(1 +\kappa\Norm{ \nabla_\bx h}_{H^{s,k}} \big)\\
		&\quad + C(h_\star,\bar M,M)\,  \Norm{h}_{H^{s,k}} \big(M+ \Norm{\bu_\star}_{H^{s,k}}+\sqrt\mu \Norm{w_\star}_{H^{s,k-1} }  \big).
	\end{align*}
	Finally, we consider the remainder (stemming from differentiating the incompressibility condition~\eqref{eq:incompressibility-redef})
	\begin{align*}
		R^{\rm{div}}_{\balpha,j}&=  \big(\de_\varrho^j \de_\bx^\balpha (h  \nabla_\bx \cdot \bu)- ( h^{(\balpha,j)} )\nabla_\bx \cdot \bu - h \nabla_\bx \cdot  \bu^{(\balpha,j)}\big) \\
		& \qquad + \big(\de_\varrho^j \de_\bx^\balpha((\de_\varrho \bu) \cdot (\nabla_\bx \cH)) -  (\de_\varrho\bu^{(\balpha,j)} ) \cdot (\nabla_\bx \cH)-(\de_\varrho \bu) \cdot( \nabla_\bx  \cH^{(\balpha,j)} )\big)\\
		& \qquad +  \big(\de_\varrho^j \de_\bx^\balpha (\bar h  \nabla_\bx \cdot \bu)-  \bar h \nabla_\bx \cdot \bu^{(\balpha,j)}\big) + \big(\de_\varrho^j \de_\bx^\balpha(\bar \bu' \cdot \nabla_\bx \cH) -\bar \bu'\cdot \nabla_\bx \cH^{(\balpha,j)} \big).
	\end{align*}
	Using Lemma~\ref{L.commutator-Hsk-sym} for the first two terms and direct estimates for the last to terms, and $k=s\geq s_0+5/2$,
	\begin{align*}
		\Norm{R_{\balpha,j}^{\rm{div}}}&\lesssim \Norm{h}_{H^{s-1, k-1}} \Norm{\nabla_\bx \cdot \bu}_{H^{s-1, k-1}} + \Norm{\de_\varrho \bu}_{H^{s-1, k-1}} \Norm{\nabla_\bx \cH}_{H^{s-1, k-1}}\\
		&\qquad + \Norm{\bar h}_{W^{k, \infty}_\varrho} \Norm{\nabla_\bx \cdot \bu}_{H^{s-1,k-1}} + \Norm{\bar \bu'}_{W^{k, \infty}_\varrho} \Norm{\nabla_\bx \cH}_{H^{s-1, k-1}}\\
		&\lesssim \big(  \Norm{\bar h}_{W^{k, \infty}_\varrho}+ \Norm{\bar \bu'}_{W^{k, \infty}_\varrho}+ \Norm{h}_{H^{s-1, k-1}}+\Norm{ \bu}_{H^{s, k}}\big) \big( \Norm{ \cH}_{H^{s, k-1}}+\Norm{ \bu}_{H^{s, k-1}} \big).
	\end{align*}
	This concludes the proof.
\end{proof}

\subsection{A priori energy estimates} In this section we provide {\em a priori} energy estimates associated with the equations featured in Lemma~\ref{lem.quasilin-nonhydro}.  We point out that
such estimates concerning $ h^{(\balpha,j)}$ solving the transport-diffusion equation~\eqref{eq.quasilin-j-h-nonhydro} have been provided in Lemma~\ref{lem:estimate-transport-diffusion}. Corresponding estimates for $\nabla_\bx\cH$ stemming from the first equation of~\eqref{eq.quasilin-j-nonhydro} are easily obtained. Hence we consider the remaining equations in~\eqref{eq.quasilin-j-nonhydro}. Specifically, recalling the notation $\dot \cH= \cH^{(\balpha,j)}, \dot h=h^{(\balpha, j)}, \dot \bu=\bu^{(\balpha, j)}, \dot w=w^{(\balpha, j)}, \dot P_{\text{nh}}= P_{\text{nh}}^{(\balpha, j)}$,
we consider the following linearized system:
\begin{equation}\label{eq.nonhydro-linearized}
	\begin{aligned}
		\partial_t  \dot\cH+(\bar\bu+\bu)\cdot \nabla_\bx \dot \cH
		+ \int_\varrho^{\rho_1} \big( (\bar \bu' + \de_{\varrho}\bu) \cdot\nabla_\bx \dot \cH + (\bar h+h)\nabla_\bx \cdot \dot\bu\big) \dd\varrho' -\kappa\Delta_\bx   \dot\cH &=R,\\
		\varrho \big( \partial_t\dot\bu+\big(({\bar \bu}+\bu-\kappa\tfrac{\nabla_\bx h}{\bar h+ h})\cdot\nabla_\bx \big)\dot\bu \big)
		+\rho_0 \nabla_\bx \dot \cH |_{\varrho=\rho_0} + \int_{\rho_0}^\varrho \nabla_\bx \dot \cH\dd\varrho'+\nabla_\bx \dot{{\Pnh}}  + \frac{\nabla_\bx \cH}{\bar h + h} \de_\varrho \dot{\Pnh} 
		&=\bR^\nh ,\\
		\sqrt\mu \varrho \big(  \partial_t \dot w + \mu \big({\bar \bu}+\bu-\kappa\tfrac{\nabla_\bx h}{\bar h+ h}\big)\cdot\nabla_\bx \dot w\big) - \frac1{\sqrt\mu}\frac{\de_\varrho \dot{\Pnh} }{\bar h + h} &=  R^\nh,\\
		-\de_\varrho \dot w + (\bar h + h) \nabla_\bx \cdot \dot \bu + \nabla_\bx \cH \cdot \de_\varrho \dot \bu - (\de_\varrho \dot \cH)\nabla_\bx \cdot \bu + \nabla_\bx \dot \cH \cdot (\bar \bu'+\de_\varrho \bu) &= R^{\rm div} ,\end{aligned}
\end{equation}
where we denote as always $\cH(\cdot,\varrho)=\int_{\varrho}^{\rho_1}h(\cdot,\varrho)\dd\varrho$.

We shall use the following definitions of the spaces $Y^0$ and $Y^1$
\begin{equation}\label{eq:Y01spaces}
	\begin{aligned}
		Y^0&:= \cC^0([\rho_0,\rho_1];L^2(\RR^d))\times L^2(\Omega)^d \times  L^2(\Omega) \times  L^2(\Omega)\ , \qquad \text{ and } \\
		Y^1&:=\left\{ (\cH,\bu,w,P)\in  H^{1,1}(\Omega)^{d+3} \ : \ \cH\big\vert_{\varrho=\rho_0}\in H^1(\RR^d),\ w\big\vert_{\varrho=\rho_1}=0,\ P\big\vert_{\varrho=\rho_0}=0\right\}.
	\end{aligned}
\end{equation}
\begin{lemma}\label{lem:estimates-nonhydro}
	Let $M, h_\star,h^\star>0$ be fixed. There exists $C(M,h_\star,h^\star)>0$ such that for any $\kappa>0$ and $\mu > 0$, and for any $(\bar h, \bar \bu) \in W^{1, \infty}((\rho_0, \rho_1))^{1+d}$ and any $T>0$ and
	$(h,\bu, w)\in L^\infty(0,T;W^{1,\infty}(\Omega))$ with $\Delta_\bx h\in L^1(0,T;L^\infty(\Omega))$
	satisfying~\eqref{eq:nonhydro-iso-redef} and, for almost any $t\in [0,T]$, the estimate
	\[
	\Norm{ h(t,\cdot)}_{L^\infty(\Omega)}
	+\Norm{ \nabla_\bx h(t,\cdot)}_{L^\infty_\bx L^2_\varrho}
	+\Norm{ \nabla_\bx\cdot \bu(t,\cdot) }_{L^\infty(\Omega)}  \le M
	\]
	and the upper and lower bounds 
	\[ \forall (\bx,\varrho)\in  \Omega , \qquad h_\star \leq  \bar h(\varrho)+h(t,\bx,\varrho) \leq h^\star ;  \]
	and
	for any $(\dot\cH, \dot\bu, \dot w,\dot \Pnh) \in \cC^0([0,T]; Y^0)\cap L^2(0,T; Y^1)$ and $(R,\bR^\nh,R^\nh,R^{\rm div})\in L^2(0,T; Y^0)$ satisfying system~\eqref{eq.nonhydro-linearized} in $L^2(0,T; Y^1)'$,
	the following inequality holds:
	\begin{align*} \frac{\dd}{\dd t} \cE(\dot\cH, \dot\bu,\dot w)+ \tfrac\kappa2 \Norm{\nabla_\bx \dot \cH}_{L^2(\Omega)}^2& +  \rho_0\kappa \norm{\nabla_\bx \dot \cH \big\vert_{\varrho=\rho_0}  }_{ L^2_\bx}^2\\
		\leq &\, C\, (1+\kappa^{-1}\Norm{ \bar \bu'+\de_\varrho \bu }_{L^\infty_\bx L^2_\varrho}^2 ) 
		\cE(\dot\cH, \dot\bu,\dot w) \\
		&+C\,\big( M+\Norm{ \bar \bu'+\de_\varrho \bu }_{ L^\infty_\bx L^\infty_\varrho} \big) \Norm{\dot{\Pnh}}_{L^2(\Omega)}  \big(  \Norm{\de_\varrho\dot \cH}_{L^2(\Omega)}+\Norm{\nabla_\bx \dot\cH}_{L^2(\Omega)} \big)  \\
		&+\Norm{\dot{\Pnh}}_{L^2(\Omega)}  \Norm{R^{\rm div}}_{L^2(\Omega)} + C\, \cE(\dot\cH, \dot\bu,\dot w)^{1/2}  \cE(R, \bR^\nh, R^\nh)^{1/2} ,
	\end{align*}
	where we denote
	\begin{equation}\label{eq:energy-nonhydro}
		\cE(\dot \cH, \dot \bu, \dot w) = \frac 12 \int_{\rho_0}^{\rho_1} \int_{\RR^d} \dot \cH^2 + \varrho (\bar h + h) |\dot\bu|^2+\mu \varrho (\bar h + h) \dot w^2 \, \dd\bx \dd \varrho + \frac 12 \int_{\RR^d} \dot \cH^2|_{\varrho=\rho_0} \, \dd \bx .
	\end{equation}
\end{lemma}
\begin{proof}
	We test the first equation against $\dot \cH \in L^2(0,T; H^{1,1}(\Omega))$ and its trace on $\{(\bx,\rho_0),\ \bx\in\RR^d\}$ against $\rho_0\dot \cH|_{\varrho=\rho_0}\in L^2(0,T; H^{1}_\bx(\RR^d))$, the second equation against $(\bar h + h) \dot \bu \in L^2(0,T; H^{1,1}(\Omega)^d)$ and the third equation against $\sqrt\mu(\bar h +h) \dot w \in L^2(0,T; H^{1,1}(\Omega))$. This yields:
	\begin{align*}
		&\frac{\dd}{\dd t}\cE(\dot \cH, \dot \bu, \dot w) + \kappa \Norm{ \nabla_\bx \dot \cH}_{L^2(\Omega)}^2+\kappa \norm{ \nabla_\bx \dot \cH|_{\varrho=\rho_0}}_{L^2_\bx }^2  \\
		& = -\big( (\bar \bu + \bu) \cdot \nabla_\bx \dot \cH, \dot \cH\big)_{L^2(\Omega)} - \left(\int_\varrho^{\rho_1} (\bar \bu'+\de_\varrho \bu) \cdot \nabla_\bx \dot \cH \, \dd\varrho', \dot \cH \right)_{L^2(\Omega)}  & {\rm (i)}\\
		& \quad - \left( \int_\varrho^{\rho_1} (\bar h+h) \nabla_\bx \cdot \dot \bu \, \dd \varrho', \dot \cH\right)_{L^2(\Omega)} + \big(R, \dot \cH\big)_{L^2(\Omega)} & {\rm (ii)} \\
		&\quad - \big(\varrho (\bar \bu + \bu) \cdot \nabla_\bx \dot \bu, (\bar h + h) \dot \bu\big)_{L^2(\Omega)} + \kappa \big(\varrho (\nabla_\bx h \cdot \nabla_\bx) \dot \bu, \dot \bu\big)_{L^2(\Omega)}  & {\rm (iii)}\\
		& \quad - \big(\rho_0 \nabla_\bx \dot \cH|_{\varrho=\rho_0},(\bar h +h) \dot \bu\big)_{L^2(\Omega)}  - \left(\int_{\rho_0}^\varrho \nabla_\bx \dot \cH \, \dd \varrho', (\bar h +h) \dot \bu \right)_{L^2(\Omega)}  & {\rm (iv)}\\
		&\quad - \big(\nabla_\bx \dot {\Pnh},  (\bar h+h)\dot \bu\big)_{L^2(\Omega)} - \big((\de_\varrho \dot{\Pnh}) \, \nabla_\bx \cH , \dot \bu\big)_{L^2(\Omega)} + \big(\bR^\nh, (\bar h + h) \dot \bu\big)_{L^2(\Omega)}  & {\rm (v)}\\
		& \quad - \mu \big(\varrho (\bar \bu + \bu) \cdot \nabla_\bx \dot w, (\bar h + h) \dot w\big)_{L^2(\Omega)} +\mu \kappa \big(\varrho (\nabla_\bx h \cdot \nabla_\bx ) \dot w, \dot w\big)_{L^2(\Omega)}  & {\rm (vi)}\\
		&\quad + \big(\de_\varrho  \dot{\Pnh}, \dot w\big)_{L^2(\Omega)}  + \sqrt\mu\big( R^\nh, (\bar h + h) \dot w\big)_{L^2(\Omega)}  & {\rm (vii)}\\
		& \quad - \rho_0 \left ((\bar \bu+\bu) (\nabla_\bx \dot \cH|_{\varrho=\rho_0}), \dot \cH|_{\varrho=\rho_0} \right)_{L^2_\bx} - \rho_0 \left( \int_{\rho_0}^{\rho_1} (\bar \bu'+\de_\varrho \bu) \cdot \nabla_\bx \dot \cH \dd \varrho', \dot \cH|_{\varrho=\rho_0} \right)_{L^2_\bx} & {\rm (viii)}\\
		& \quad - \rho_0 \left( \int_{\rho_0}^{\rho_1} (h+\bar h) (\nabla_\bx \cdot \dot \bu) \dd \varrho', \dot \cH|_{\varrho=\rho_0} \right)_{L^2_\bx} + \big(R|_{\varrho=\rho_0}, \dot \cH|_{\varrho=\rho_0}\big)_{L^2_\bx} & {\rm (ix)}\\
		&\quad + \frac 12 \big(\varrho (\de_t h) \dot \bu, \dot \bu\big)_{L^2(\Omega)} + \frac \mu 2 \big(\varrho (\de_t h) \dot w, \dot w\big)_{L^2(\Omega)}. & {\rm (x)}
	\end{align*}
	Some terms have already been treated in the course of the proof of Lemma~\ref{lem:estimate-system}: the second term in \rm{(i)} and the second term in \rm{(viii)} require $\kappa>0$; the first terms in \rm{(i)}, \rm{(viii)} are advection terms; the first addend of \rm{(ii)} together with the second term in \rm{(iv)} after integration by parts; the first addend of \rm{(iv)} with the first term in \rm{(ix)}.
	The contributions in \rm{(iii)} compensate with the first addend of \rm{(x)}, using the first equation of~\eqref{eq:nonhydro-iso-redef} and, in the same way, the contributions in \rm{(vi)} compensate with the second addend of \rm{(x)}.
	It remains only to deal with the contribution frm the non-hydrostatic pressure terms in \rm{(v)} and \rm{(vii)}, and remainder terms.
	
	Consider the sum of the first two terms in \rm{(v)} and the first term in \rm{(vii)}. We integrate by parts in $\bx$ the first term and in $\varrho$ in the last two terms. Thus we have
	\begin{multline*} 
		-\big(\nabla_\bx \dot{\Pnh}, (\bar h + h) \dot \bu\big)_{L^2(\Omega)} - \big((\de_\varrho \dot{\Pnh}) \nabla_\bx  \cH, \dot \bu\big)_{L^2(\Omega)} + \big(\de_\varrho \dot{\Pnh}, \dot w\big)_{L^2(\Omega)} \\
		= \big(\dot{\Pnh}, (\bar h + h) \nabla_\bx \cdot \dot \bu\big)_{L^2(\Omega)}+\big(\dot{\Pnh} \nabla_\bx  \cH, \de_\varrho \dot \bu\big)_{L^2(\Omega)} -\big(\dot{\Pnh}, \de_\varrho \dot w\big)_{L^2(\Omega)},
	\end{multline*}
	where we used the identity $ h=-\de_\varrho \cH$ and the boundary conditions $\dot{\Pnh}|_{\varrho=\rho_0}= \dot \cH|_{\varrho=\rho_1}=\dot w|_{\varrho=\rho_1}=0$ when integrating by parts with respect to $\varrho$. 
	Using the last equation in~\eqref{eq.nonhydro-linearized} (stemming from the incompressibility condition), the above term reads
	\[
	\big(\dot{\Pnh}, (\nabla_\bx \cdot \bu)(\de_\varrho\dot \cH)-  (\bar \bu'+\de_\varrho \bu)\cdot \nabla_\bx \dot \cH +R^{\rm div}\big)_{L^2(\Omega)}.
	\]
	These terms, alike remainder terms
	\[
	\big\vert \big(R, \dot \cH\big)_{L^2(\Omega)}\big\vert+\big\vert\big(R|_{\varrho=\rho_0}, \dot \cH|_{\varrho=\rho_0}\big)_{L^2_\bx} \big\vert+\big\vert\big(\bR^\nh, (\bar h+h) \dot \bu\big)_{L^2(\Omega)}\big\vert+ \sqrt\mu\big\vert (R^\nh, (\bar h + h) \dot w\big)_{L^2(\Omega)}\big\vert\ ,
	\]
	are bounded by Cauchy-Schwarz inequality and using $\rho_0 h_\star\leq \varrho(\bar h+h)\leq \rho_1 h^\star$.

	Altogether, we obtain the differential inequality
	\begin{align*} \frac{\dd}{\dd t} \cE(\dot\cH, \dot\bu,\dot w)+ \kappa \Norm{\nabla_\bx \cH}_{L^2(\Omega)}^2 &+  \rho_0\kappa \norm{\nabla_\bx \cH \big\vert_{\varrho=\rho_0}  }_{ L^2_\bx}^2\\
		\leq &\,C\,
		\cE(\dot\cH, \dot\bu,0) +C \Norm{ \bar \bu'+\de_\varrho \bu }_{L^2_\varrho L^\infty_\bx}   \cE(\dot\cH, \dot\bu,0)^{1/2} \Norm{\nabla_\bx \dot \cH}_{L^2(\Omega)}\\
		&+C\, \big( M+\Norm{ \bar \bu'+\de_\varrho \bu }_{L^\infty_\varrho L^\infty_\bx} \big) \Norm{\dot{\Pnh}}_{L^2(\Omega)}  \big( \Norm{\de_\varrho\dot \cH}_{L^2(\Omega)}+\Norm{\nabla_\bx \dot\cH}_{L^2(\Omega)} \big)  \\
		&+\Norm{\dot{\Pnh}}_{L^2(\Omega)}  \Norm{R^{\rm div}}_{L^2(\Omega)} + C \cE(\dot\cH, \dot\bu,\dot w)^{1/2}  \cE(R, \bR^\nh, R^\nh)^{1/2} 
	\end{align*}
	with $C=C(h_\star,h^\star,M)$, and the desired estimate follows straightforwardly.
\end{proof}
\begin{remark}\label{R.estimates-nonhydro}
	Lemma~\ref{lem:estimates-nonhydro} will be applied to the system~\eqref{eq.quasilin-j-nonhydro}-\eqref{eq.est-quasilin-j-nonhydro} appearing in Lemma~\ref{lem.quasilin-nonhydro}, {\em when $j=0$}. A similar result holds for the simplified system when $j\neq 0$. The main difference is that the result does not require nor provide the control of the trace $\de_\varrho^j\cH\big\vert_{\varrho=\rho_0}$.
\end{remark}

\subsection{Large-time well-posedness}
We prove the large-time existence of strong solutions to system~\eqref{eq:nonhydro-iso-recall}. As for the hydrostatic system, \emph{large time} underlines the fact that the existence time that is provided by the following result is uniformly bounded (from below) with respect to the vanishing parameter $\mu\in(0,1]$. Besides, the result below keeps track of the dependency of this large time-scale on the diffusivity parameter $\kappa\in[\mu,1]$.

\begin{proposition}\label{P.NONHydro-large-time}
	Let  $s, k \in \NN$ be such that $k=s> \frac 52 +\frac d 2$ and $\bar M,M,h_\star,h^\star>0$. Then, there exists $C>0$ such that, for any $ 0<\mu\leq \kappa\leq 1$, and any $(\bar h, \bar \bu) \in W^{k,\infty}((\rho_0,\rho_1)) \times W^{k+1,\infty}((\rho_0,\rho_1))^d $ such that 
	\[  \norm{\bar h}_{W^{k,\infty}_\varrho } + \norm{\bar \bu'}_{W^{k,\infty}_\varrho }\leq \bar M\,;\]
	for any initial data $(h_0, \bu_0, w_0)\in H^{s,k}(\Omega)^{d+2}$ with
	\[
	M_0:=
	\Norm{\cH_0}_{H^{s,k}}+\Norm{\bu_0}_{H^{s,k}}+\sqrt \mu \Norm{w_0}_{H^{s,k}}+\norm{\cH_0\big\vert_{\varrho=\rho_0}}_{H^s_\bx}+\kappa^{1/2}\Norm{h_0}_{H^{s,k}} + \mu^{1/2}\kappa^{1/2} \Norm{\nabla_\bx \cH_0}_{H^{s,k}}  \leq M\,,
	\]
	and satisfying  the boundary condition $w_0|_{\varrho=\rho_1}=0$ and the incompressibility condition
	\[-(\bar h+h_0)\nabla_\bx\cdot \bu_0-(\nabla_\bx \cH_0)\cdot({\bar \bu}'+\partial_\varrho\bu_0)+\partial_\varrho  w_0=0,\] 
	the lower and upper bounds
	\[\forall (\bx,\varrho)\in  \Omega , \qquad h_\star \leq  \bar h(\varrho)+h_0(\bx,\varrho) \leq h^\star \,, \]
	\item and the smallness assumption
	\[  C \kappa^{-1}\, \big(\norm{ \bar \bu'}_{L^\infty_\varrho }^2+ M_0^2) \leq 1\,,\]
	the following holds. Denoting by
	\[
	T^{-1}= C\, \big(1+ \kappa^{-1} \big(\norm{\bar \bu'}_{L^2_\varrho}^2+M_0^2\big)  \big),  
	\]
	there exists a unique 
	$(h,\bu,w)\in \cC^0([0,T];H^{s,k}(\Omega)^{2+d})$ and $P\in L^2(0,T;H^{s+1,k+1}(\Omega))$ strong solution to~\eqref{eq:nonhydro-iso-redef}
	with initial data $(h,\bu,w)\big\vert_{t=0}=(h_0,\bu_0,w_0)$.
	Moreover, one has $\cH\in L^\infty(0,T;H^{s+1,k}(\Omega))$ and $(h,\nabla_\bx\cH)\in   L^2(0,T;H^{s+1,k}(\Omega))$ and
	, for any $t\in[0,T]$, the lower and the upper bounds hold
	\[\forall (\bx,\varrho)\in  \Omega , \qquad h_\star/2 \leq  \bar h(\varrho)+h(t,\bx,\varrho) \leq 2 h^\star \,, \]
	and the estimate below holds true
	\begin{align}\cF(t):&=
		\Norm{\cH(t,\cdot)}_{H^{s,k}}+\Norm{\bu(t,\cdot)}_{H^{s,k}} +\mu^{1/2}\Norm{w(t,\cdot)}_{H^{s,k}}+\norm{\cH\big\vert_{\varrho=\rho_0}(t,\cdot)}_{H^s_\bx} \notag\\ 
		&\quad+\kappa^{1/2}\Norm{h(t,\cdot)}_{H^{s,k}}  + \mu^{1/2}\kappa^{1/2} \Norm{\nabla_\bx \cH (t, \cdot)}_{H^{s,k}} \notag\\ 
		&\quad+ \kappa^{1/2} \Norm{\nabla_\bx \cH}_{L^2(0,t;H^{s,k})} +   \kappa^{1/2} \norm{\nabla_\bx \cH \big\vert_{\varrho=\rho_0}  }_{ L^2(0,t;H^s_\bx)} \notag\\ 
		&\quad+\kappa \Norm{\nabla_\bx h}_{L^2(0,t;H^{s,k})} +  \mu^{1/2}\kappa \Norm{\nabla_\bx^2 \cH }_{L^2(0,t; H^{s,k})} \leq C\, M_0. \label{eq:functional-nonhydro}
	\end{align}
\end{proposition}
\begin{proof}
	As for the large-time existence for the hydrostatic system (see Proposition~\ref{P.regularized-large-time-WP}), the proof is based on a bootstrap argument on the functional $\cF$. Recalling that the (short-time) existence and uniqueness of the solution has been provided in Proposition~\ref{P.NONHydro-small-time},
	we denote by $T^\star$ the maximal existence time, and set 
	\begin{equation}\label{control-F}
		T_\star= \sup \{0 < T < T^\star \, : \; \forall \, t \in (0,T), \; h_\star/2 \le \bar h(\varrho) + h(t, \bx, \varrho) \le 2h^\star \quad \text{and} \quad \cF(t) \le C_0 M_0\},
	\end{equation}
	with $C_0=C(h_\star,h^\star,\bar M, M)$ sufficiently large (to be determined). Henceforth, we restrain to $0<T<T_\star$, and and denote by $C$ any positive constant depending uniquely on $\bar M, h_\star, h^\star, C_0M_0$ and $s,k$.
	
	By means of~\eqref{eq.quasilin-j-h-nonhydro}-\eqref{eq.est-quasilin-j-h-nonhydro} in Lemma~\ref{lem.quasilin-nonhydro} and Lemma~\ref{lem:estimate-transport-diffusion}, we infer as in the proof of Proposition~\ref{P.regularized-large-time-WP} the control
	\begin{equation}\label{control-h}
		\kappa^{1/2}\Norm{h}_{L^\infty(0,T;H^{s,k})}+\kappa\Norm{\nabla_\bx h}_{L^2(0,T;H^{s,k})}
		\leq  \left(c_0 M_0+C\, C_0 M_0\, \big( T+\sqrt{T} \big)\right)
		\times \exp\Big( C C_0 M_0\, T\Big)
	\end{equation}
	with the same notations as above and $c_0$ a universal constant.
	In the non-hydrostatic situation, additional controls can be inferred on $\cH$. Indeed, from the first equation in
	in Lemma~\ref{lem.quasilin-nonhydro},~\eqref{eq.quasilin-j-nonhydro}-\eqref{eq.est-quasilin-j-nonhydro}, we find that 
	\[	\partial_t  \cH^{(\balpha,j)}+ (\bar \bu + \bu) \cdot\nabla_\bx \cH^{(\balpha,j)}  = \kappa \Delta_\bx \cH^{(\balpha)}+\widetilde R_{\balpha,j}+ w^{(\balpha,j)} \]
	with
	\[
	\sqrt\mu \Norm{\widetilde R_{\balpha,j}+ w^{(\balpha,j)}}_{L^2(\Omega) }
	\leq C\, C_0 M_0\, . 
	\]
	Differentiating once with respect to the space variables and proceeding as in Lemma~\ref{lem:estimate-transport-diffusion}, we infer
	\begin{multline}\label{control-cH}
		\mu^{1/2}\kappa^{1/2}\Norm{\nabla_\bx \cH}_{L^\infty(0,T;H^{s,k})}+\mu^{1/2}\kappa\Norm{\nabla_\bx^2 \cH}_{L^2(0,T;H^{s,k})}
		\\
		\leq \big(c_0M_0 + C C_0 M_0(T +\sqrt T)\big)\times\exp\Big(C C_0 M_0 T\Big).
	\end{multline}
	Next we use again Lemma~\ref{lem.quasilin-nonhydro},~\eqref{eq.quasilin-j-nonhydro}-\eqref{eq.est-quasilin-j-nonhydro}, together with Lemma~\ref{lem:estimates-nonhydro} (see also Remark~\ref{R.estimates-nonhydro}) to obtain that the functional
	\[\cE^{s,k} :=\frac12\sum_{j=0}^k \sum_{|\balpha|=0}^{s-j} \iint_{\Omega} ( \de_\varrho^j \de_\bx^\balpha \cH)^2 + \varrho (\bar h + h) |\de_\varrho^j\de_\bx^\balpha\bu|^2+\mu \varrho (\bar h + h) ( \de_\varrho^j\de_\bx^\balpha w)^2 \, \dd\bx \dd \varrho + \frac 12  \sum_{|\balpha|=0}^{s}\int_{\RR^d} (\de_\bx^\balpha \cH|_{\varrho=\rho_0})^2 \, \dd \bx ,\]
	satisfies the differential inequality
	\begin{equation} \label{ineq}\frac{\dd}{\dd t} \cE^{s,k}+ \tfrac\kappa2 \Norm{\nabla_\bx \cH}_{H^{s,k}}^2 +  \rho_0\kappa \norm{\nabla_\bx \cH \big\vert_{\varrho=\rho_0}  }_{ H^s_\bx}^2\leq C\, \big(  R_1+R_2+R_3 \big);
	\end{equation}
	with
	\begin{align*}
		R_1 &:= (1+\kappa^{-1}\Norm{ \bar \bu'+\de_\varrho \bu }_{L^\infty_\bx L^2_\varrho }^2 )\cE^{s,k}  ,\\
		R_2 &:= \big( C_0 M_0+\Norm{ \bar \bu'+\de_\varrho \bu }_{L^\infty_\varrho L^\infty_\bx} \big) \Norm{{\Pnh}}_{H^{s,k}}  \big(  \Norm{ h}_{H^{s,k}}+ \Norm{\nabla_\bx \cH}_{H^{s,k}}\big),\\
		R_3 &:=\Norm{ P_\nh}_{H^{s,k}} \Norm{R^{\rm{div}}_{s,k}}_{L^2(\Omega)} + (\cE^{s,k})^{1/2} \Norm{\mathcal{R}_{s,k}}_{L^2(\Omega)},
	\end{align*}
	and
	\begin{align} 
		\Norm{R^{\rm{div}}_{s,k}}_{L^2(\Omega)} & \leq C\  C_0 M_0\,, \label{ineq-Rdiv} \\
		\Norm{\mathcal R_{s,k}}_{L^2(\Omega)}  & \leq C\, C_0M_0\, \big(1+\kappa \Norm{\nabla_\bx h}_{H^{s,k}}  \big)   \notag\\
		& \quad + C\,   \big( \Norm{ h}_{H^{s,k}}+\mu^{1/2}\Norm{\nabla_\bx \cH}_{H^{s,k}}  \big) \big(C_0M_0+ \Norm{\bu_\star}_{H^{s,k}}+\mu^{1/2} \Norm{w_\star}_{H^{s,k} }  \big)  \,. \label{ineq-R}
	\end{align}
	
	By~\eqref{control-F}, we have obviously for any  $0<t<T_\star$,
	\[  \frac1{2\rho_1 h^\star}\cE^{s,k}(t) \leq \Norm{\cH(t,\cdot)}_{H^{s,k}}^2+\Norm{\bu(t,\cdot)}_{H^{s,k}}^2 +\mu\Norm{w(t,\cdot)}_{H^{s,k}}^2+\norm{\cH\big\vert_{\varrho=\rho_0}(t,\cdot)}_{H^s_\bx}^2 \leq  \frac2{\rho_0 h_\star}\cE^{s,k}(t).\]
	Moreover, we have the following control on $\bu_\star:=-\kappa\frac{\nabla_\bx  h}{\bar h+ h} $ and  $w_\star:=\kappa\Delta_\bx \cH-\kappa\frac{\nabla_\bx h\cdot\nabla_\bx \cH}{\bar h+h}$ stemming from (tame) product and composition estimates (Lemma~\ref{L.product-Hsk} and~\ref{L.composition-Hsk-ex}), and using that $\mu\leq \kappa\leq 1$:
	\begin{equation}\label{control-u-w-star}
		\Norm{\bu_\star}_{L^2(0,T;H^{s,k})}+\mu^{1/2}\Norm{w_\star}_{L^2(0,T;H^{s,k})} \leq C \,C_0 M_0 (1+\sqrt T)\,.
	\end{equation}
	Finally, using estimate~\eqref{ineq:est-Psmall-nonhydro} in Corollary~\ref{C.Poisson} yields
	\begin{align*}
		\Norm{\Pnh}_{H^{s,k}} & \leq \Norm{\Pnh}_{L^2} +\Norm{\nabla_{\bx,\varrho} \Pnh}_{H^{s-1,k-1}} \leq \Norm{\Pnh}_{L^2} + \mu^{-1/2}  {\Norm{\nabla^\mu_{\bx,\varrho} \Pnh}_{H^{s-1,k-1}} }\notag \\
		& \le C\, \left(\Norm{ \nabla_\bx \cH}_{H^{s-1,k-1}} + \norm{\cH\big\vert_{\varrho=\rho_0}}_{H^{s}_\bx}+ \Norm{(\bu,\bu_\star)}_{H^{s,k}} +\mu^{1/2}\Norm{(w,w_\star)}_{H^{s,k-1}} \right),
	\end{align*}
	from which we infer, using the controls~\eqref{control-F} and~\eqref{control-u-w-star}, that
	\begin{equation}\label{ineq-P}
		\Norm{\Pnh}_{L^2(0,T;H^{s,k})}\leq C\, C_0 M_0(1+\sqrt T).
	\end{equation}
	
	From~\eqref{control-F} and~\eqref{ineq-Rdiv}-\eqref{ineq-R}-\eqref{control-u-w-star}-\eqref{ineq-P} we infer
	\begin{align*}
		\int_0^T R_1(t)\dd t &\leq  C\,  (C_0 M_0)^2 (1+\kappa^{-1}\norm{ \bar \bu'}_{L^2_\varrho}^2+\kappa^{-1}(C_0M_0)^2  )\, T,\\
		\int_0^T R_2(t)\dd t  &\leq  C\, \kappa^{-1/2} \, \big( C_0 M_0+\norm{ \bar \bu'}_{L^\infty_\varrho } \big) \, (C_0 M_0)^2(1+\sqrt T)^2 \, ,\\
		\int_0^T R_3(t)\dd t  &\leq    C\,  (C_0 M_0)^2(T+\sqrt T)\ +\   C\,  (C_0 M_0)^2\, \big( T+C_0 M_0 \sqrt T+ \kappa^{-1/2}(C_0 M_0)(T+\sqrt T)\big) .
	\end{align*}
	Hence there exists $C>0$, depending on $\bar M, h_\star, h^\star, C_0, M_0$ (and $s,k$), such that if
	\[C\, T\,  \big( 1+\kappa^{-1}(\norm{\bu'}_{L^2_\varrho}^2 +M_0^2))\leq 1,\]
	and imposing additionally
	\footnote{We point out that the only term requiring the above smallness condition~\eqref{eq:smallness} on the initial data is (the time integral of) $R_2$, and more precisely the product $ \Norm{P_{\text{nh}}}_{H^{s,k}}\Norm{\nabla_\bx \cH}_{H^{s,k}}$, where both terms are only square-integrable in time.} 
	that
	\begin{equation}\label{eq:smallness}
		C\, \kappa^{-1/2}\,  \big( C_0 M_0+\norm{ \bar \bu'}_{L^\infty_\varrho } \big)  \leq \tfrac1{16} \rho_0 h_\star
	\end{equation}
	we have, when integrating the differential inequality~\eqref{ineq} and combine with~\eqref{control-h} and~\eqref{control-cH},
	\[ \cE^{s,k}(t) \leq \cE^{s,k}(0)+ \tfrac18 (\rho_0h_\star) (C_0 M_0)^2\,.\]
	
	Now, setting $C_0=\max(\{ 4(\frac{\rho_1 h^\star}{\rho_0h_\star})^{1/2}, 8 c_0\}$, and $C$ accordingly, one has 
	$\cF(t) \leq  C_0 M_0/2$ for all ${0<t<T}$. We obtain as in the proof of Proposition~\ref{P.regularized-large-time-WP} the lower and upper bounds $2h_\star/3 \le \bar h(\varrho) + h(t, \bx, \varrho) \le 3h^\star/2$, augmenting $C$ if necessary, and the standard continuity argument allows to conclude the proof.
\end{proof}

\section{Convergence}\label{S.Convergence}

This section is devoted to the proof of the convergence of regular solutions to the non-hydrostatic equations~\eqref{eq:nonhydro-iso-intro} towards the corresponding solutions to the limit hydrostatic equations~\eqref{eq:hydro-iso-intro}, namely Theorem~\ref{thm-convergence}. 
Our convergence result holds in the strong sense and ``with loss of derivatives'': we prove that the solutions to the approximating (non-hydrostatic) equations converge towards the solutions to the limit (hydrostatic) equations in a suitable strong topology that is strictly weaker than the one measuring the size of the initial data.

For a given set of initial data, we use the apex $\h$ to refer to the solution to the hydrostatic equations (provided by the analysis of Section~\ref{S.Hydro} culminating with Theorem~\ref{thm-well-posedness}), and the apex $\nh$ for the corresponding solution to the non-hydrostatic equations (provided by the analysis of Section~\ref{S.NONHydro}, specifically Proposition~\ref{P.NONHydro-small-time}). The apex $\d$ denotes the difference between the non-hydrostatic solution and the hydrostatic one, whose size will be controlled in the limit $\mu \searrow 0$.

While we can appeal to Theorem~\ref{thm-well-posedness} to obtain the existence, uniqueness and control of solutions to the hydrostatic equations over a large time interval, Proposition~\ref{P.NONHydro-small-time} provides only a time interval which {\em a priori} vanishes as $\mu \searrow 0$, and Proposition~\ref{P.NONHydro-large-time} only applies to sufficiently small initial data. The standard strategy (used for instance in~\cite{KlainermanMajda82} in the context of weakly compressible flows) that we apply here relies on a bootstrap argument to control the difference between the non-hydrostatic solution and the hydrostatic one in the time-interval provided by the hydrostatic solution, from which the existence and control of the non-hydrostatic solution (again, with loss of derivatives) can be inferred. We perform this analysis in Sections~\ref{S.CV-consistency} to~\ref{S.CV-control}, where we first provide a consistency result (Lemma~\ref{L.CV-consistency}), then exhibit the (non-hydrostatic) quasilinear structure of the equations satisfied by the difference (Lemma~\ref{L.CV-quasilinear}), and finally infer the uniform control of the non-hydrostatic solution and the strong convergence towards the corresponding hydrostatic solution (Proposition~\ref{P.CV-control}). In a last step, in Section~\ref{S.CV-convergence}, we use this uniform control to offer an improved convergence rate based this time on the structure of the hydrostatic equations (Proposition~\ref{P.CV-convergence}). Propositions~\ref{P.CV-control} and~\ref{P.CV-convergence} immediately yield Theorem~\ref{thm-convergence}.

\subsection{Consistency}\label{S.CV-consistency}

In the following result we prove that solutions to the hydrostatic equations~\eqref{eq:hydro-iso-intro} emerging from smooth initial data satisfy (suitably defining the horizontal velocity and pressure variables) the non-hydrostatic equations~\eqref{eq:nonhydro-iso-intro}, up to small remainder terms.

\begin{lemma}\label{L.CV-consistency}
	There exists $p\in\NN$ such that for any $s,k\in \NN$ with $0\leq k\leq s$, the following holds. 
	Let $\bar M,M,h_\star,h^\star>0$ be fixed. Then there exists $C_0>0$ and $C_1>0$ such that for any $\kappa\in(0,1]$, any $(\bar h,\bar \bu)\in W^{k+p,\infty}((\rho_0,\rho_1))^{1+d} $ satisfying
	\[  \norm{\bar h}_{W^{k+p,\infty}_\varrho } + \norm{\bar \bu'}_{W^{k+p-1,\infty}_\varrho }\leq \bar M,\]
	and any initial data $(h_0, \bu_0) \in H^{s+p,k+p}(\Omega)$ satisfying the following estimate
	\[
	M_0:=
	\Norm{\cH_0}_{H^{s+p,k+p}}+\Norm{\bu_0}_{H^{s+p,k+p}}+\norm{\cH_0\big\vert_{\varrho=\rho_0}}_{H^{s+p}_\bx}+\kappa^{1/2}\Norm{h_0}_{H^{s+p,k+p}} 
	\le M
	\]
	(where we denote  $\cH_0(\cdot,\varrho):=\int_\varrho^{\rho_1} h_0(\cdot,\varrho')\dd\varrho'$) and the stable stratification assumption
	\[ \inf_{(\bx,\varrho)\in  \Omega } h_\star\leq \bar h(\varrho)+h_0(\bx,\varrho) \leq h^\star,\] 
	there exists a unique $(h^\h,\bu^\h)\in  \cC^0([0,T];H^{s+p,k+p}(\Omega)^{1+d})$ strong solution to~\eqref{eq:hydro-iso-intro} with initial data $(h^\h,\bu^\h)\big\vert_{t=0}=(h_0,\bu_0)$, where
	\[ T^{-1}= C_0\, \big(1+ \kappa^{-1} \big(\norm{\bar \bu'}_{L^2_\varrho}^2+M_0^2\big)  \big) \,  .\]
	Moreover, one has for all $t\in[0,T]$,
	\[ \forall (\bx,\varrho)\in  \Omega , \qquad h_\star/2 \leq  \bar h(\varrho)+h^\h(t,\bx,\varrho) \leq 2\,h^\star , \]
	and, denoting $\cH^\h(\cdot,\varrho):=\int_\varrho^{\rho_1} h^\h(\cdot,\varrho')\dd\varrho'$ and
	\begin{align}\label{def-wh}
		w^\h(\cdot,\varrho) &:=-\int_\varrho^{\rho_1} (\bar h(\varrho') +  h^\h(\cdot,\varrho')) \nabla_\bx \cdot  \bu^\h(\cdot,\varrho')+\nabla_\bx \cH^\h(\cdot,\varrho') \cdot ({\bar \bu}'(\varrho')+\partial_\varrho\bu^\h(\cdot,\varrho'))\dd\varrho'\,,\\
		\label{def-Ph} 
		P^\h(\cdot,\varrho)&:=\int_{\rho_0}^\varrho \varrho' h^\h(\cdot,\varrho')\dd\varrho'\,,
	\end{align}
	one has for any $ t\in[0,T]$,
	\begin{equation}\label{eq:hydro-estimate}
		\Norm{\big(h^\h(t,\cdot),\cH^\h(t,\cdot),\bu^\h(t,\cdot),w^\h(t,\cdot),P^\h(t,\cdot)\big)}_{H^{s+1,k+1}}\leq C_1\, M_0\,,
	\end{equation}
	and \begin{subequations}
		\begin{equation}\label{eq:hydro-consistency}
			\begin{aligned}
				\partial_t  h^\h+\nabla_\bx \cdot\big((\bar h +h^\h)(\bar\bu +\bu^\h)\big)&= \kappa \Delta_\bx  h^\h,\\
				\varrho\Big( \partial_t  \bu^\h+\big((\bar \bu +  \bu^\h - \kappa\tfrac{\nabla_\bx  h^\h}{ \bar h+  h^\h})\cdot\nabla_\bx\big)  \bu^\h\Big)+  \nabla_\bx  P^\h+ \frac{\nabla_\bx  \cH^\h}{\bar h+  h^\h}( \de_\varrho P^\h + \varrho \bar h^\h) &=0,\\
				\mu \varrho\Big( \partial_t  w^\h+\big(\bar \bu +  \bu^\h - \kappa\tfrac{\nabla_\bx   h^\h}{ \bar h+ h^\h}\big)\cdot\nabla_\bx   w^\h\Big)- \frac{\de_\varrho  P^\h}{\bar h+ h^\h} + \frac{\varrho  h^\h}{\bar h + h^\h}&=\mu\, R^\h,\\
				-(\bar h +  h^\h) \nabla_\bx \cdot  \bu^\h-\nabla_\bx \cH^\h \cdot ({\bar \bu}'+\partial_\varrho\bu^\h) +\partial_\varrho  w^\h&=0,
			\end{aligned}
		\end{equation}
		with $R^\h(t,\cdot) \in \cC^0([0,T];H^{s,k}(\Omega))$ and satisfying for any $ t\in[0,T]$,
		\begin{equation}\label{est:hydro-consistency}
			\Norm{R^\h(t,\cdot)}_{H^{s,k}}  \leq C_1\, M_0.
		\end{equation}
	\end{subequations}
\end{lemma}
\begin{proof}
	From Theorem~\ref{thm-well-posedness}  we infer immediately (for $p>2+d/2$) the existence, uniqueness and control of the hydrostatic solution $(h^\h,\bu^\h)\in \cC^0([0,T];H^{s+p,k+p}(\Omega)^{1+d})$, and $C_0>0$. From the formula~\eqref{def-wh},~\eqref{def-Ph} and product  estimates (Lemma~\ref{L.product-Hsk}) in the space $H^{s+p',k+p'}(\Omega)$ (for $1\leq p'\leq p$ sufficiently large) we infer the estimate~\eqref{eq:hydro-estimate}.
	
	We obtain similarly the desired consistency estimate,~\eqref{eq:hydro-consistency}-\eqref{est:hydro-consistency}, using the identity (recall~\eqref{eq.id-Phydro})
	\[
	P^\h +\varrho\cH^\h = \int_{\rho_0}^{\varrho} \cH^\h(\cdot,\varrho')\dd \varrho'+\rho_0 \cH^\h\big\vert_{\varrho=\rho_0}\,,
	\]
	and denoting
	\[R^\h:= \varrho\Big( \partial_t  w^\h+\big(\bar \bu +  \bu^\h - \kappa\tfrac{\nabla_\bx   h^\h}{ \bar h+ h^\h}\big)\cdot\nabla_\bx   w^\h\Big),\]
	differentiating with respect to time the identity~\eqref{def-wh}, and using~\eqref{eq:hydro-iso-intro} to infer the control of $ \de_t\bu^\h$ and $\de_t w^\h$.
\end{proof}

As a corollary to the above, we can write the equations satisfied by the difference between  $(h^\h,\bu^\h,w^\h)$, i.e. the maximal solution to the hydrostatic equations emerging from given regular, well-prepared initial data, and  $(h^\nh,\bu^\nh,w^\nh)$, i.e. the maximal solution to the non-hydrostatic with the same data (see Proposition~\ref{P.NONHydro-small-time}). Specifically, under the assumptions and using the notations of Lemma~\ref{L.CV-consistency}, we have that
\[
h^\d:=h^\nh-h^\h;  \quad 
\bu^\d:=\bu^\nh-\bu^\h;  \quad
w^\d:=w^\nh-w^\h; \quad 
\]
satisfies $(h^d,\bu^\d,w^\d)\big\vert_{t=0}=(0,0,0)$ and
\begin{equation}\label{eq:difference}
	\begin{aligned}
		\de_t h^\d 
		+\nabla_\bx\cdot \big(  (\bar \bu+\bu^\nh)   h^\d
		+ (\bar h + h^\h)  \bu^\d \big)  & =\kappa \Delta_\bx h^\d,\\
		\de_t \cH^\d 
		+ (\bar \bu+\bu^\nh) \cdot \nabla_\bx \cH^\d
		+\int_\varrho^{\rho_1} (\bar \bu'+\de_\varrho \bu^\nh) \cdot \nabla_\bx \cH^\d \, \dd \varrho' 
		+\int_\varrho^{\rho_1} (\bar h + h^\nh) \nabla_\bx \cdot \bu^\d \, \dd\varrho' & \\
		+\int_\varrho^{\rho_1} \bu^\d \cdot \nabla_\bx h^\h + h^\d \nabla_\bx \cdot \bu^\h \, \dd \varrho' & = \kappa\Delta_\bx\cH^\d,\\
		\de_t \bu^\d 
		+ \big((\bar\bu + \bu^\nh - \kappa \tfrac{\nabla_\bx h^\nh}{\bar h + h^\nh} ) \cdot \nabla_\bx\big) \bu^\d 
		+ {\frac{\rho_0}{\varrho} \nabla_\bx \cH^\d|_{\varrho=\rho_0}} 
		+ {\frac 1 \varrho \int_{\rho_0}^\varrho \nabla_\bx \cH^\d \, \dd\varrho'} & \\
		+ \big(( \bu^\d - \kappa ( \tfrac{\nabla_\bx h^\nh}{\bar h + h^\nh} - \tfrac{\nabla_\bx h^\h}{\bar h+h^\h})) \cdot \nabla_\bx\big) \bu^\h
		+ \frac{\nabla_\bx{ {\Pnh} }}{\varrho} 
		+ \frac{\nabla_\bx \cH^\nh}{\varrho (\bar h+h^\nh)} \de_\varrho{ {\Pnh} } &= 0,\\
		\mu \big( \partial_t  w^\d+\big(\bar \bu 
		+  \bu^\nh - \kappa\tfrac{\nabla_\bx   h^\nh}{ \bar h+ h^\nh}\big)\cdot\nabla_\bx   w^\d 
		+ ( \bu^\d - \kappa ( \tfrac{\nabla_\bx h^\nh}{\bar h + h^\nh} - \tfrac{\nabla_\bx h^\h}{\bar h+h^\h})) \cdot \nabla_\bx  w^\h\big)
		- \frac{\de_\varrho  P_\nh}{\varrho(\bar h+ h^\nh)}
		&=-\mu\, R^\h,\\
		-(\bar h +  h^\nh) \nabla_\bx \cdot  \bu^\d
		-  h^\d \nabla_\bx \cdot  \bu^\h
		-\nabla_\bx \cH^\d \cdot ({\bar \bu}'+\partial_\varrho\bu^\nh) 
		-\nabla_\bx \cH^\h \cdot (\partial_\varrho\bu^\d)
		+\partial_\varrho  w^\d &=0, 
	\end{aligned}
\end{equation}
where we denote  as usual $\cH^\h(\cdot,\varrho)=\int_\varrho^{\rho_1} h^\h(\cdot,\varrho')\dd\varrho'$ (and analogously $\cH^\nh$, $\cH^\d$), and define the non-hydrostatic pressure ${P_\nh(\cdot,\varrho):=P^\nh(\cdot,\varrho)-\int_{\rho_0}^\varrho \varrho' h^\nh(\cdot,\varrho')\dd\varrho'}$ where $P^\nh$ is defined by Corollary~\ref{C.Poisson}.

\subsection{Quasi-linearization}\label{S.CV-quasilinear}
In this section we extract the leading order terms of the system~\eqref{eq:difference}, in the spirit of Lemma~\ref{lem.quasilin-nonhydro}.

\begin{lemma}\label{L.CV-quasilinear}
	There exists $p\in\NN$ such that for any  $s, k \in \NN$ such that $k=s> \frac 52+\frac d 2$ and $\bar M,M,h_\star>0$,  there exists $C>0$ and $C_1>0$ such that the following holds.
	For any $0<\mu\leq \kappa \leq 1$, and
	for any $(\bar h, \bar \bu) \in W^{k+p,\infty}((\rho_0,\rho_1))^{1+d} $ satisfying
	\[  \norm{\bar h}_{W^{k+p,\infty}_\varrho } + \norm{\bar \bu'}_{W^{k+p-1,\infty}_\varrho }\leq \bar M \,;\]
	and any $(h^\nh,\bu^\nh,w^\nh) \in \cC^0([0,T^\nh];H^{s,k}(\Omega)^{d+2})$ and  $P^\nh\in L^2(0,T^\nh;H^{s+1,k+1}(\Omega))$ solution to~\eqref{eq:nonhydro-iso-redef}
	with some  $T^\nh>0$ and satisfying for any $t\in [0,T^\nh]$
	\begin{multline*}
		\Norm{h^\nh(t, \cdot)}_{H^{s-1, k-1}} + \Norm{\cH^\nh(t, \cdot)}_{H^{s,k}}+\norm{\cH^\nh(t, \cdot)\big\vert_{\varrho=\rho_0}}_{H^s_\bx}+\Norm{\bu^\nh(t, \cdot)}_{H^{s,k}}+\mu^{1/2}\Norm{w^\nh(t, \cdot)}_{H^{s,k}}\\
		+\kappa^{1/2}\Norm{h^\nh(t, \cdot)}_{H^{s,k}}+ \mu^{1/2} \kappa^{1/2} \Norm{\nabla_\bx \cH^\nh(t, \cdot)}_{H^{s,k}}
		\le M 
	\end{multline*}
	(where $\cH^\nh(t,\bx,\varrho):=\int_\varrho^{\rho_1} h^\nh(t,\bx,\varrho')\dd\varrho'$), the stable stratification assumption
	\[ \inf_{(\bx,\varrho)\in  \Omega } \bar h(\varrho)+h^\nh(t, \bx,\varrho) \geq h_\star,\]
	and the initial bound
	\[
	M_0:=
	\Norm{\cH^\nh\big\vert_{t=0}}_{H^{s+p,k+p}}+\Norm{\bu^\nh\big\vert_{t=0}}_{H^{s+p,k+p}}+\norm{(\cH^\nh\big\vert_{\varrho=\rho_0})\big\vert_{t=0}}_{H^{s+p}_\bx}+\kappa^{1/2}\Norm{h^\nh\big\vert_{t=0}}_{H^{s+p,k+p}} 
	\le M,
	\]
	we have the following.
	
	Denote $(h^\h,\bu^\h,w^\h) \in  \cC^0([0,T^\h];H^{s+1,k+1}(\Omega)^{2+d})$ the corresponding strong solution to the hydrostatic equations~\eqref{eq:hydro-iso-nu} (see Lemma~\ref{L.CV-consistency}) satisfying
	\[
	\Norm{h^\h(t,\cdot)}_{H^{s+1,k+1}} +\Norm{\bu^\h(t,\cdot)}_{H^{s+1,k+1}}+\Norm{\cH^\h(t,\cdot)}_{H^{s+1,k+1}} +\Norm{w^\h(t,\cdot)}_{H^{s+1,k+1}} \leq C_1M_0
	\]
	and, for any multi-index $\balpha \in \NN^d$ and $j\in\NN$ such that $0 \le |\balpha| +j\le s$,
	\[
	\cH^{(\balpha,j)}:=\de_\bx^\balpha\de_\varrho^j\cH^\nh-\de_\bx^\balpha\de_\varrho^j\cH^\h;  \quad 
	\bu^{(\balpha,j)}:=\de_\bx^\balpha\de_\varrho^j\bu^\nh-\de_\bx^\balpha\de_\varrho^j\bu^\h;  \quad
	w^{(\balpha,j)}:=\de_\bx^\balpha \de_\varrho^j w^\nh-\de_\bx^\balpha \de_\varrho^j w^\h; \quad 
	\]
	and
	${P^{(\balpha, j)}_\nh(\cdot,\varrho)=\de_\bx^\balpha\de_\varrho^j \big(P^\nh(\cdot,\varrho)-\int_{\rho_0}^\varrho \varrho' h^\nh(\cdot,\varrho')\dd\varrho'\big)}$.
	
	Then restricting to $t\in[0,\min(T^\h,T^\nh)]$ and such that
	\begin{align*}
		\cF_{s,k}:=\Norm{h^\d}_{H^{s-1, k-1}} + \Norm{\cH^\d}_{H^{s,k}}+\norm{\cH^\d\big\vert_{\varrho=\rho_0}}_{H^s_\bx}+\Norm{\bu^\d}_{H^{s,k}} +\mu^{1/2}\Norm{w^\d}_{H^{s,k}} &\\
		+\kappa^{1/2}\Norm{h^\d}_{H^{s,k}}+\mu^{1/2}\kappa^{1/2}\Norm{\nabla_\bx\cH^\d}_{H^{s,k}} & \leq \kappa^{1/2}M
	\end{align*}
	we have 
	\begin{subequations}
		\begin{equation}\label{eq.quasilin-diff}
			\begin{aligned}
				\partial_t  \cH^{(\balpha,j)}+ (\bar \bu + \bu^\nh) \cdot\nabla_\bx \cH^{(\balpha,j)} + w^{(\balpha,j)} - \kappa \Delta_\bx \cH^{(\balpha)}&= \widetilde R_{\balpha,j}, \\
				\partial_t  \cH^{(\balpha,j)}+ (\bar \bu + \bu^\nh) \cdot\nabla_\bx \cH^{(\balpha,j)} +\Big\langle \int_\varrho^{\rho_1} (\bar \bu' + \de_\varrho\bu^\nh) \cdot\nabla_\bx \cH^{(\balpha,j)} \, \dd\varrho'\phantom{\qquad -\kappa\Delta_\bx   \cH^{(\balpha,j)}} &\\
				+\int_\varrho^{\rho_1} (\bar h+h^\nh)\nabla_\bx \cdot\bu^{(\balpha,j)} \dd\varrho'\Big\rangle_{j=0}-\kappa\Delta_\bx   \cH^{(\balpha,j)}&=R_{\balpha,j},\\
				\partial_t\bu^{(\balpha,j)}+\big(({\bar \bu}+\bu^\nh-\kappa\tfrac{\nabla_\bx h^\nh}{\bar h+ h^\nh})\cdot\nabla_\bx\big) \bu^{(\balpha,j)}
				+\Big\langle \frac{\rho_0}{ \varrho }\nabla_\bx \cH^{(\balpha,j)}\big\vert_{\varrho=\rho_0} +\frac 1 \varrho\int_{\rho_0}^\varrho  \nabla_\bx  \cH^{(\balpha,j)} \dd\varrho' \Big\rangle_{j=0}&\\
				+\frac 1 \varrho\nabla_\bx \Pnh^{(\balpha,j)}  + \frac{\nabla_\bx \cH^\nh}{\varrho(\bar h + h^\nh)} \de_\varrho \Pnh^{(\balpha,j)} &=\bR^\nh_{\balpha,j},\\
				\mu^{1/2}  \left( \partial_t w^{(\balpha,j)} + (\bar \bu + \bu^\nh - \kappa \tfrac{\nabla_\bx h^\nh}{\bar h + h^\nh} ) \cdot \nabla_\bx w^{(\balpha,j)}\right) - \frac1{\mu^{1/2}}\frac{\de_\varrho \Pnh^{(\balpha,j)}}{ \varrho( \bar h + h^\nh)}  &= R^\nh_{\balpha,  j}, \\
				- \de_\varrho w^{(\balpha,j)} + (\bar h + h^\nh) \nabla_\bx \cdot \bu^{(\balpha,j)} + (\bar \bu'+\de_\varrho \bu^\nh) \cdot \nabla_\bx  \cH^{(\balpha,j)}&\\
				+   (\nabla_\bx \cdot \bu^\nh)h^{(\balpha,j)}+ (\nabla_\bx \cH^\nh )\cdot ( \de_\varrho  \bu^{(\balpha,j)})  &=R^{\rm div}_{\balpha,j},
			\end{aligned}
		\end{equation}
		where $(R_{\balpha,j}(t,\cdot),\bR^\nh_{\balpha,j}(t,\cdot),R_{\balpha,  j}^\nh(t,\cdot),R_{\balpha,  j}^{\rm div})\in L^2(\Omega)^{d+3}$,  $R_{\balpha,0}(t,\cdot)\in\cC((\rho_0,\rho_1);L^2(\RR^d))$ and 
		\begin{align} \label{eq.est-quasilin-diff0}
			&
			\Norm{  R_{\balpha,j}}_{L^2(\Omega) } +\norm{R_{\balpha,0}\big\vert_{\varrho=\rho_0}}_{L^2_\bx }  +  \Norm{ \widetilde R_{\balpha,j}}_{L^2(\Omega) } +\Norm{R^{\rm div}_{\balpha, j}}_{L^2(\Omega)}   \leq C\,\cF_{s,k},  \\
			\label{eq.est-quasilin-diff1}
			& \Norm{ \bR^\nh_{\balpha,j}}_{L^2(\Omega)}+\Norm{ R^\nh_{\balpha, j}}_{L^2(\Omega)} 
			\leq  
			C  \, \big(  \cF_{s,k} + \kappa  \Norm{\nabla_\bx h^\d}_{H^{s,k}} +\mu^{1/2}\kappa \Norm{\Delta_\bx \cH^\d}_{H^{s,k}} \big)+C\, \mu^{1/2}M ,
		\end{align}
	\end{subequations}
	and
	\begin{subequations}
		\begin{equation}\label{eq.quasilin-h-diff}
			\partial_t  h^{(\balpha,j)}+(\bar\bu+\bu^\nh)\cdot \nabla_\bx  h^{(\balpha,j)}-\kappa\Delta_\bx h^{(\balpha,j)}
			=r_{\balpha,j}+\nabla_\bx \cdot \br_{\balpha,j},
		\end{equation}
		where $(r_{\balpha,j}(t,\cdot),\br_{\balpha,j}(t,\cdot))\in L^2(\Omega)^{1+d}$ and
		\begin{equation}\label{eq.est-quasilin-h-diff}
			\kappa^{1/2} \Norm{r_{\balpha,j}}_{L^2(\Omega) }  + \Norm{\br_{\balpha,j}}_{L^2(\Omega) }  \leq C\,\cF_{s,k}.
		\end{equation}
	\end{subequations}
\end{lemma}
\begin{proof}
	Explicit expressions for the remainder terms follow from~\eqref{eq:difference}. Specifically, 
	the following equation is obtained by combining the second and last equation (recall~\eqref{eq.id-eta})
	\[\partial_t  \cH^\d+ (\bar \bu + \bu^\nh) \cdot\nabla_\bx \cH^\d + \bu^\d\cdot\nabla_\bx\cH^\h - w^\d = \kappa \Delta_\bx \cH^\d \]
	and hence
	\[ \widetilde R_{\balpha,j}:=-[\de_\bx^\balpha\de_\varrho^j, \bar \bu+\bu^\nh]\cdot\nabla_\bx\cH^\d- \de_\bx^\balpha\de_\varrho^j(\bu^\d\cdot\nabla_\bx\cH^\h),\]
	and it follows from product (Lemma~\ref{L.product-Hsk}) and commutator (Lemma~\ref{L.commutator-Hsk}) estimates
	\[ \Norm{\widetilde R_{\balpha,j}}_{L^2(\Omega)} \lesssim \big(\norm{\bar\bu'}_{W^{k-1,\infty}_\varrho}+\Norm{\bu^\nh}_{H^{s,k}} +\Norm{\cH^\h}_{H^{s+1,k}} \big)\,\big(\Norm{\cH^\d}_{H^{s,k-1}}+\Norm{ \bu^\d}_{H^{s,k}}\big) .\]
	
	Then, from the second equation we have
	\[ R_{\balpha,j} := R^{(i)}_{\balpha,j}+R^{(ii)}_{\balpha,j}\]
	with $R^{(i)}_{\balpha,j}:=- [\de_\bx^\balpha\de_\varrho^j,\bar\bu+\bu^\nh] \cdot \nabla_\bx \cH^\d$ and
	\[R^{(ii)}_{\balpha,j}:=\begin{cases}
		-\int_\varrho^{\rho_1} [\partial_\bx^\balpha , \de_\varrho \bu^\nh ]\cdot \nabla_\bx \cH^\d + [\partial_\bx^\balpha , h^\nh] \nabla_\bx \cdot \bu^\d 
		+ \partial_\bx^\balpha \big( \bu^\d \cdot \nabla_\bx h^\h + h^\d \nabla_\bx \cdot \bu^\h\big) \, \dd \varrho' & \text{if $j=0$,} \\
		\de_\varrho^{j-1}\partial_\bx^\balpha \big( (\bar \bu'+\de_\varrho \bu^\nh) \cdot \nabla_\bx \cH^\d+ (\bar h + h^\nh) \nabla_\bx \cdot \bu^\d + \bu^\d \cdot \nabla_\bx h^\h + h^\d \nabla_\bx \cdot \bu^\h\big)  & \text{if $j\geq1$.}
	\end{cases}
	\]
	Using Lemma~\ref{L.product-Hsk}, Lemma~\ref{L.commutator-Hsk} and the continuous embedding $L^\infty((\rho_0,\rho_1))\subset L^2((\rho_0,\rho_1)) \subset L^1((\rho_0,\rho_1))$ we find
	\begin{multline*}\Norm{ R_{\balpha,j}}_{L^2(\Omega)} \lesssim \big( \norm{\bar\bu'}_{W^{k-1,\infty}_\varrho}+\Norm{\bu^\nh}_{H^{s,k}}+\Norm{h^\nh}_{H^{s-1,k-1}}+\Norm{\cH^\nh}_{H^{s,k-1}} +\Norm{h^\h}_{H^{s+1,k-1}} +\Norm{\bu^\h}_{H^{s+1,k-1}}\big)\\
		\times\big(\Norm{\cH^\d}_{H^{s,k-1}}+\Norm{ \bu^\d}_{H^{s,k-1}}+\Norm{ h^\d}_{H^{s-1,k-1}}\big) 
	\end{multline*}
	where for $j=0$ we used the identities (and Lemma~\ref{L.embedding} and Lemma~\ref{L.commutator-Hs},(\ref{L.commutator-Hs-3}) and (\ref{L.commutator-Hs-3}))
	\begin{multline*} \int_\varrho^{\rho_1} [\partial_\bx^\balpha , \de_\varrho \bu^\nh ]\cdot \nabla_\bx \cH^\d + [\partial_\bx^\balpha , h^\nh] \nabla_\bx \cdot \bu^\d \dd\varrho' = \int_\varrho^{\rho_1} [\partial_\bx^\balpha ; \de_\varrho \bu^\nh ,\nabla_\bx \cH^\d]+ [\partial_\bx^\balpha ; h^\nh, \nabla_\bx \cdot \bu^\d ] \dd\varrho'\\
		+\int_\varrho^{\rho_1} \partial_\bx^\balpha \bu^\nh\cdot \nabla_\bx h^\d +(\partial_\bx^\balpha  \cH^\nh) (\nabla_\bx \cdot \de_\varrho\bu^\d) \dd\varrho'+ \partial_\bx^\balpha \bu^\nh\cdot\nabla_\bx \cH^\d  +(\partial_\bx^\balpha  \cH^\nh) (\nabla_\bx \cdot \bu^\d)
	\end{multline*}
	and
	\[ \int_\varrho^{\rho_1} \partial_\bx^\balpha ( h^\d \nabla_\bx \cdot \bu^\h) \dd\varrho' = \int_\varrho^{\rho_1} [\partial_\bx^\balpha ,\nabla_\bx \cdot \bu^\h] h^\d  +(\partial_\bx^\balpha \cH^\d)(\nabla_\bx \cdot \de_\varrho\bu^\h)\dd\varrho' +  (\partial_\bx^\balpha \cH^\d)(\nabla_\bx \cdot \bu^\h).\]
	This yields the desired estimate for $\Norm{ R_{\balpha,j}}_{L^2(\Omega)}$ and the corresponding estimate for $\norm{R_{\balpha,0}\vert_{\varrho=\rho_0}}_{L^2_\bx }$ relies on the additional estimate (stemming from Lemma \ref{L.commutator-Hs}(\ref{L.commutator-Hs-2}) and Lemma \ref{L.embedding})
	\[ \norm{ (R^{(i)}_{\balpha,0}+ \partial_\bx^\balpha \bu^\nh\cdot\nabla_\bx \cH^\d )\vert_{\varrho=\rho_0}}_{L^2_\bx }= \norm{ [\de_\bx^\balpha;\bu^\nh\vert_{\varrho=\rho_0}, \nabla_\bx \cH^\d\vert_{\varrho=\rho_0}]}_{L^2_\bx }\lesssim \Norm{\bu^\nh}_{H^{s,1}}\norm{\cH^\d\vert_{\varrho=\rho_0}}_{H^s_\bx}.\]
	
	Then, we have
	\begin{align*} R^{\rm div}_{\balpha, j}&:= [\de_\bx^\balpha\de_\varrho^j, \bar h ]\nabla_\bx \cdot \bu^\d +  [\de_\bx^\balpha\de_\varrho^j, \bar \bu'  ]\cdot \nabla_\bx  \cH^\d \\
		&\quad + [\de_\bx^\balpha\de_\varrho^j,  h^\nh ]\nabla_\bx \cdot \bu^\d +[\de_\bx^\balpha\de_\varrho^j, \nabla_\bx \cdot \bu^\h ] h^\d + \nabla_\bx \cdot (\bu^\h-\bu^\nh)\de_\bx^\balpha\de_\varrho^j h^\d\\
		&\quad + [\de_\bx^\balpha\de_\varrho^j, \de_\varrho \bu^\nh ]\cdot \nabla_\bx  \cH^\d+[\de_\bx^\balpha\de_\varrho^j, \nabla_\bx \cH^\h ] \cdot  \de_\varrho  \bu^\d +\nabla_\bx (\cH^\h-\cH^\nh)\cdot \de_\bx^\balpha\de_\varrho^j  \de_\varrho  \bu^\d  .
	\end{align*}
	Decomposing $ h^\nh= h^\h+ h^\d$, $ \de_\varrho \bu^\nh= \de_\varrho \bu^\h+ \de_\varrho \bu^\d$, some manipulations of the terms to exhibit symmetric commutators and the use of Lemma~\ref{L.commutator-Hsk} and~\ref{L.commutator-Hsk-sym} lead to
	\begin{multline*} \Norm{ R^{\rm div}_{\balpha, j} }_{L^2(\Omega)}  \lesssim \big(\norm{\bar h}_{W^{k,\infty}_\varrho}+\norm{\bar \bu'}_{W^{k,\infty}_\varrho}+\Norm{ h^\d}_{H^{s-1,k-1}}+\Norm{ \de_\varrho \bu^\d}_{H^{s-1,k-1}}\\
		+\Norm{ h^\h}_{H^{s,k}} +\Norm{ \nabla_\bx \cH^\h}_{H^{s,k}}+\Norm{ \bu^\h}_{H^{s+1,k+1}} \big)
		\times\big( \Norm{\bu^\d}_{H^{s,k}}+\Norm{\cH^\d}_{H^{s,k-1}}+\Norm{h^\d}_{H^{s-1,k-1}}\big) 
	\end{multline*}
	which concludes the estimate~\eqref{eq.est-quasilin-diff0}.
	
	We focus now on $\Norm{ \bR^\nh_{\balpha,j}}_{L^2(\Omega)}$ and $\Norm{ R^\nh_{\balpha, j}}_{L^2(\Omega)} $.
	We have
	\begin{align*}
		\bR^\nh_{\balpha,j}&\textstyle:=[\de_\bx^\balpha\de_\varrho^j,(\bar\bu + \bu^\nh - \kappa \tfrac{\nabla_\bx h^\nh}{\bar h + h^\nh})\cdot\nabla_\bx ]  \bu^\d 
		+ \big\langle\de_\bx^\balpha\de_\varrho^j \big(\frac{\rho_0}{\varrho} \nabla_\bx \cH^\d|_{\varrho=\rho_0} 
		+ \frac 1 \varrho \int_{\rho_0}^\varrho \nabla_\bx \cH^\d \, \dd\varrho' \big)\big\rangle_{j\geq 1} \\
		&\quad + \de_\bx^\balpha\de_\varrho^j\big(\big(( \bu^\d - \kappa ( \tfrac{\nabla_\bx h^\nh}{\bar h + h^\nh} - \tfrac{\nabla_\bx h^\h}{\bar h+h^\h})) \cdot \nabla_\bx\big) \bu^\h\big)
		+  [\de_\bx^\balpha\de_\varrho^j,\tfrac{1}{\varrho}]\nabla_\bx {\Pnh}  
		+ [\de_\bx^\balpha\de_\varrho^j,\tfrac{\nabla_\bx \cH^\nh}{\varrho (\bar h+h^\nh)}] \de_\varrho {\Pnh}  ,\\
		R^\nh_{\balpha, j}&:=\mu^{1/2} [\de_\bx^\balpha\de_\varrho^j,\big(\bar \bu +  \bu^\nh - \kappa\tfrac{\nabla_\bx   h^\nh}{ \bar h+ h^\nh}\big)\cdot\nabla_\bx]   w^\d 
		+ \mu^{1/2} \de_\bx^\balpha\de_\varrho^j\big( ( \bu^\d - \kappa ( \tfrac{\nabla_\bx h^\nh}{\bar h + h^\nh} - \tfrac{\nabla_\bx h^\h}{\bar h+h^\h})) \cdot \nabla_\bx  w^\h\big)\\
		&\quad -\tfrac1{\mu^{1/2}} [\de_\bx^\balpha\de_\varrho^j,\tfrac{1}{\varrho(\bar h+ h^\nh)}] \de_\varrho  P_\nh -\mu^{1/2}\de_\bx^\balpha\de_\varrho^j R^\h,
	\end{align*}
	where $R^\h$ is the consistency remainder introduced in Lemma~\ref{L.CV-consistency},~\eqref{eq:hydro-consistency} and estimated in~\eqref{est:hydro-consistency}, namely \[\Norm{\de_\bx^\balpha\de_\varrho^j R^\h}_{L^2(\Omega)} \lesssim M_0\leq M.\]
	Let us estimate each contribution. In the following, we shall use repeatedly that $\cF^{s,k}\leq \kappa^{1/2}M$ and hence  $\Norm{ h^\d}_{H^{s,k}}\leq M$. As a consequence, by Lemma~\ref{L.composition-Hsk-ex} and triangular inequality,
	\[\Norm{\tfrac{h^\nh }{\bar h+ h^\nh}}_{H^{s,k}} \leq C(h_\star,\norm{\bar h}_{W^{k,\infty}_\varrho}\Norm{h^\nh}_{H^{s-1,k-1}})\Norm{h^\nh }_{H^{s,k}} \leq  C(h_\star,\bar M,M)M. \]
	By Lemma~\ref{L.commutator-Hsk}, we have
	\begin{multline*}\Norm{ [\de_\bx^\balpha \de_\varrho^j,(\bar\bu+\bu^\nh )\cdot \nabla_\bx] \bu^\d}_{L^2(\Omega)}+\mu^{1/2}\Norm{[\de_\bx^\balpha \de_\varrho^j,(\bar\bu+\bu^\nh )\cdot \nabla_\bx] w^\d }_{L^2(\Omega)}\\ \lesssim \big(\norm{\bar\bu'}_{W^{k-1,\infty}_\varrho}+\Norm{\bu^\nh}_{H^{s,k}}\big)\big(\Norm{\nabla_\bx \bu^\d}_{H^{s-1,k-1}}+\mu^{1/2}\Norm{\nabla_\bx w^\d}_{H^{s-1,k-1}}\big) \leq C(\bar M, M) \cF_{s,k}.
	\end{multline*}
	By Lemma~\ref{L.commutator-Hsk-sym}, Lemma~\ref{L.product-Hsk} and Lemma~\ref{L.composition-Hsk-ex} 
	\begin{align*}\Norm{  [\de_\bx^\balpha\de_\varrho^j,\big(\tfrac{\nabla_\bx h^\nh}{\bar h+ h^\nh} \cdot \nabla_\bx\big)] \bu^\d}_{L^2(\Omega)}
		&\lesssim \Norm{\tfrac{\nabla_\bx h^\nh}{\bar h+ h^\nh}}_{H^{s,k}}\Norm{\nabla_\bx \bu^\d}_{H^{s-1,k-1}}\\
		&\lesssim \Norm{\tfrac{\nabla_\bx h^\nh}{\bar h}}_{H^{s,k}}\big(1+\Norm{\tfrac{h^\nh }{\bar h+ h^\nh}}_{H^{s,k}}\big) \cF_{s,k}\\
		&\leq C(h_\star,\bar M,M) \big(M+\Norm{\nabla_\bx h^\d}_{H^{s,k}} \big) \cF_{s,k}.
	\end{align*}
	In the same way, we have
	\[ \mu^{1/2}\Norm{  [\de_\bx^\balpha\de_\varrho^j,\big(\tfrac{\nabla_\bx h^\nh}{\bar h+ h^\nh} \cdot \nabla_\bx\big)] w^\d}_{L^2(\Omega)}\leq C(h_\star,\bar M,M) \big(M+\Norm{\nabla_\bx h^\d}_{H^{s,k}} \big) \cF_{s,k}.\]
	When $j\geq 1$,
	\begin{multline*}\Norm{ \de_\bx^\balpha\de_\varrho^j \big(\frac{\rho_0}{\varrho} \nabla_\bx \cH^\d|_{\varrho=\rho_0} \big)}_{L^2(\Omega)}
		+ \Norm{ \de_\bx^\balpha\de_\varrho^j \big(\frac 1 \varrho \int_{\rho_0}^\varrho \nabla_\bx \cH^\d \, \dd\varrho' \big)}_{L^2(\Omega)}\\
		\lesssim \norm{ \nabla_\bx \cH^\d|_{\varrho=\rho_0}}_{H^{s-1}_\bx}+ \Norm{ \nabla_\bx \cH^\d}_{H^{s-1,k-1}}\leq \cF_{s,k}.
	\end{multline*}
	By Lemma~\ref{L.product-Hsk}, we have
	\begin{multline*}\Norm{ \de_\bx^\balpha \de_\varrho^j \big(( \bu^\d \cdot \nabla_\bx) \bu^\h\big)}_{L^2(\Omega)}+\mu^{1/2}\Norm{\de_\bx^\balpha \de_\varrho^j \big(( \bu^\d \cdot \nabla_\bx) w^\h\big) }_{L^2(\Omega)}\\ \lesssim \Norm{\bu^\d}_{H^{s,k}}\big(\Norm{\nabla_\bx \bu^\h}_{H^{s,k}}+\mu^{1/2}\Norm{\nabla_\bx w^\d}_{H^{s,k}}\big) \leq C(\bar M, M) \cF_{s,k}.
	\end{multline*}
	By repeated use of tame estimates in Lemma~\ref{L.product-Hsk} and Lemma~\ref{L.composition-Hsk}, we find
	\begin{align*} \Norm{\de_\bx^\balpha\de_\varrho^j\big( (\tfrac{\nabla_\bx h^\nh}{\bar h + h^\nh} - \tfrac{\nabla_\bx h^\h}{\bar h+h^\h}) \cdot \nabla_\bx\big) \bu^\h\big)}_{L^2(\Omega)}
		&\lesssim \Norm{\tfrac{\nabla_\bx h^\d}{\bar h +h^\h}+ \tfrac{h^\d\nabla_\bx h^\nh}{(\bar h+h^\nh)(\bar h+h^\h)}}_{H^{s,k}}
		\Norm{\nabla_\bx \bu^\h}_{H^{s,k}}\\
		&\leq  C(h_\star,\bar M,M) M ( \Norm{\nabla_\bx h^\d}_{H^{s,k}} + M\Norm{ h^\d}_{H^{s,k}}),
	\end{align*}
	and similarly
	\[\mu^{1/2} \Norm{\de_\bx^\balpha\de_\varrho^j\big( (\tfrac{\nabla_\bx h^\nh}{\bar h + h^\nh} - \tfrac{\nabla_\bx h^\h}{\bar h+h^\h}) \cdot \nabla_\bx\big) w^\h\big)}_{L^2(\Omega)}\leq  C(h_\star,\bar M,M) M ( \Norm{\nabla_\bx h^\d}_{H^{s,k}} + M\Norm{ h^\d}_{H^{s,k}}).\]
	Contributions from the pressure remain.
	By direct inspection, and since $|\balpha|+j-1\leq s-1$,
	\[  \Norm{  [\de_\bx^\balpha\de_\varrho^j,\tfrac{1}{\varrho} ]\nabla_\bx \Pnh }_{L^2(\Omega)}\lesssim \Norm{\nabla_\bx \Pnh}_{H^{s-1,k-1}}.\]
	By Lemma~\ref{L.commutator-Hsk} and since $s=k>\frac52+\frac d2$, using the above and Lemma~\ref{L.embedding}
	\begin{align*}  \Norm{  [\de_\bx^\balpha\de_\varrho^j,\tfrac{\nabla_\bx \cH^\nh}{\varrho(\bar h + h^\nh)} ]\de_\varrho \Pnh}_{L^2(\Omega)} 
		&\lesssim \Norm{\tfrac{\nabla_\bx \cH^\nh}{\varrho(\bar h + h^\nh)}}_{H^{s,k}}\Norm{\de_\varrho \Pnh}_{H^{s-1,k-1}}\\
		&\leq  C(h_\star,\bar M,M)\,\big( M+ \Norm{\nabla_\bx \cH^\d}_{H^{s,k}} \big) \Norm{\de_\varrho \Pnh}_{H^{s-1,k-1}}. 
	\end{align*}
	Similarly,
	\begin{align*}  \Norm{  [\de_\bx^\balpha\de_\varrho^j,\tfrac{1}{\varrho(\bar h + h^\nh)}]\de_\varrho \Pnh }_{L^2(\Omega)} 
		&\lesssim \big(\norm{\tfrac1{\varrho\bar h}}_{W^{k,\infty}_\varrho}+\Norm{\tfrac{h^\nh}{\varrho(\bar h + h^\nh)}}_{H^{s,k}}\big)\Norm{\de_\varrho \Pnh}_{H^{s-1,k-1}}\\
		&\leq  C(h_\star,\bar M,M)\,  \Norm{\de_\varrho \Pnh }_{H^{s-1,k-1}}. 
	\end{align*}
	Altogether, and using $\cF_{s,k} \leq \kappa^{1/2}M$ and $\mu\leq \kappa$,  we find 
	\begin{equation}\label{est-Rj} \Norm{ \bR^\nh_{\balpha,j}}_{L^2(\Omega)}+\Norm{ R^\nh_{\balpha, j}}_{L^2(\Omega)}  \leq C(h_\star,\bar M,M)  \, \big(  \cF_{s,k}  + \kappa  \Norm{\nabla_\bx h^\d}_{H^{s,k}} +  \mu^{-1/2}  \Norm{\nabla^\mu_{\bx,\varrho} \Pnh }_{H^{s-1,k-1}}\big).
	\end{equation}
	Now, we use Corollary~\ref{C.Poisson}, specifically~\eqref{ineq:est-Ptame-nonhydro}:
	\begin{align*}
		\Norm{\nabla^\mu_{\bx,\varrho} \Pnh}_{H^{s-1,k-1}} & \leq C(h_\star,\bar M,M)\,\mu\, \Big( \Norm{ (\Lambda^\mu)^{-1} \nabla_\bx \cH^\nh}_{H^{s,k}} + \norm{(\Lambda^\mu)^{-1} \cH^\nh\big\vert_{\varrho=\rho_0}}_{H^{s+1}_\bx} \\
		&\qquad +\Norm{(\bu^\nh,\bu_\star^\nh)}_{H^{s,k}} +\Norm{( w^\nh,w_\star^\nh)}_{H^{s,k-1}} + \Norm{\bu_\star^\nh}_{H^{s,k}} \Norm{ w^\nh}_{H^{s,k-1}}\Big).
	\end{align*}
	where we  recall the notations $\Lambda^\mu:=1+\sqrt \mu |D|$, $\bu_\star^\nh:=-\kappa\frac{\nabla_\bx  h^\nh}{\bar h+ h^\nh} $ and  $w_\star^\nh:=\kappa\Delta_\bx \cH^\nh-\kappa\frac{\nabla_\bx h^\nh\cdot\nabla_\bx \cH^\nh}{\bar h+h^\nh}$.
	Then we use on one hand that
	\[ \Norm{ (\Lambda^\mu)^{-1} h^\nh}_{H^{s,k-1}} \leq \Norm{ h^\h}_{H^{s,k-1}}+\mu^{-1/2}\Norm{  h^\d}_{H^{s-1,k-1}} \lesssim M+\mu^{-1/2} \cF_{s,k},\]
	and, similarly,
	\begin{align*}\Norm{ (\Lambda^\mu)^{-1} \nabla_\bx \cH^\nh}_{H^{s,k}} &
		\leq \Norm{  \nabla_\bx \cH^\h}_{H^{s,k}} +\mu^{-1/2}\Norm{ \nabla_\bx \cH^\d}_{H^{s-1,k}}\lesssim M+\mu^{-1/2} \cF_{s,k},\\
		\norm{(\Lambda^\mu)^{-1} \cH^\nh\big\vert_{\varrho=\rho_0}}_{H^{s+1}_\bx}&\leq  \norm{\cH^\h\big\vert_{\varrho=\rho_0}}_{H^{s+1}_\bx} +\mu^{-1/2} \norm{\cH^\d\big\vert_{\varrho=\rho_0}}_{H^{s}_\bx} \lesssim M+\mu^{-1/2} \cF_{s,k}.
	\end{align*}
	On the other hand,
	\[\Norm{ w^\nh}_{H^{s,k-1}} \leq  \Norm{ w^\h}_{H^{s,k-1}}+\Norm{  w^\d}_{H^{s,k-1}} \lesssim C(\bar M,M) M+\mu^{-1/2} \cF_{s,k}\]
	where, for the first contribution, we applied the product estimates to the expression in~\eqref{def-wh}.
	Then, we have
	\begin{align*}\Norm{\bu_\star^\nh}_{H^{s,k}} &\leq \kappa \Norm{ \tfrac{\nabla_\bx h^\nh}{\bar h+h^\nh} }_{H^{s,k}} \leq  \kappa \big( \Norm{\nabla_\bx h^\nh }_{H^{s,k}} + \Norm{h^\nh }_{H^{s,k}}^2\big) \\
		&\leq C(h_\star,\bar M,M)\,\big( M+\kappa \Norm{\nabla_\bx h^\d }_{H^{s,k}}\big),\\  
		\Norm{ w_\star^\nh}_{H^{s,k-1}} &\leq \kappa\Norm{\Delta_\bx \cH^\nh}_{H^{s,k-1}}+C(h_\star,\bar M,M)\kappa\Norm{\nabla_\bx h^\nh}_{H^{s,k-1}} \Norm{\nabla_\bx \cH^\nh}_{H^{s,k-1}}\\
		&\leq C(h_\star,\bar M,M)\,\big(M+ \kappa\Norm{\Delta_\bx \cH^\d}_{H^{s,k-1}} + \kappa^{1/2}M\Norm{\nabla_\bx h^\d}_{H^{s,k-1}} \big).
	\end{align*}
	
	Altogether, this yields
	\begin{multline*}
		\mu^{-1/2}\Norm{\nabla^\mu_{\bx,\varrho} \Pnh}_{H^{s-1,k-1}}  \leq C_0\, \Big( \mu^{1/2}M+ \cF_{s,k} + \mu^{1/2}\kappa^{1/2} \Norm{\nabla_\bx h^\d }_{H^{s,k}}+\mu^{1/2}\kappa \Norm{\Delta_\bx \cH^\d}_{H^{s,k}} \\
		+\big( M+\kappa \Norm{\nabla_\bx h^\d }_{H^{s,k}} \big)\,\big( \mu^{1/2}M+ \cF_{s,k} \big)\Big).
	\end{multline*}
	Plugging this estimate in~\eqref{est-Rj}, using $\cF_{s,k}\leq \kappa^{1/2} M$ and $\mu\leq \kappa$, we obtain~\eqref{eq.est-quasilin-diff1}.
	
	\medskip
	
	Finally, we set
	\[
	\br_{\balpha,j}:= - [\de_\bx^\balpha\de_\varrho^j, \bar \bu+\bu^\nh ]  h^\d - \de_\bx^{\balpha} \de_\varrho^{j} \big( (\bar h+ h^\h) \bu^\d\big), \qquad 
	r_{\balpha,j}:= -(\de_\bx^\balpha\de_\varrho^j  h^\d )\nabla_\bx \cdot \bu^\nh  .
	\]
	By Lemma~\ref{L.product-Hsk} and Lemma~\ref{L.commutator-Hsk} and since  $s\geq s_0+\frac32$ and $2\leq k= s$, we have
	\[\Norm{\br_{\balpha,j}}_{L^2(\Omega)} \lesssim  \big(\norm{\bar \bu'}_{W^{k-1,\infty}_\varrho}+\Norm{\bu^\nh}_{H^{s,k}}\big)\Norm{h^\d}_{H^{s-1,k-1}} +\big(\norm{\bar h}_{W^{k,\infty}_\varrho}+\Norm{h^\h}_{H^{s,k}}\big)\Norm{\bu^\d}_{H^{s,k}}\]
	and (by Lemma~\ref{L.embedding})
	\[\Norm{r_{\balpha,j}}_{L^2(\Omega)} \lesssim \Norm{h^\d}_{H^{s,k}}\Norm{\bu^\nh}_{H^{s,k}}.\]
	This yields immediately~\eqref{eq.est-quasilin-h-diff}. The proof is complete.
\end{proof}

\subsection{Strong convergence}\label{S.CV-control}
In this section, we prove that for $\mu$ sufficiently small and starting from regu\-lar and well-prepared initial data, the solution to the non-hydrostatic equations exists at least within the existence time of the solution to the hydrostatic equation. We also prove the strong convergence of the non-hydrostatic system to the hydrostatic one as $\mu\searrow0$.

\begin{proposition}\label{P.CV-control}
	There exists $p\in\NN$ such that for any  $s, k \in \NN$ such that $k=s> \frac 52+\frac d 2$ and any $\bar M,M,h_\star,h^\star>0$,  there exists $C=C(s,k,\bar M,M,h_\star,h^\star)>0$ such that the following holds. For any $ 0< M_0\leq M$, $0<\kappa\leq 1$, and $\mu>0$ such that 
	\[\mu \leq \kappa/(C M_0^2 ),\] 
	for any $(\bar h, \bar \bu) \in W^{k+p,\infty}((\rho_0,\rho_1))^{1+d} $ satisfying
	\[  \norm{\bar h}_{W^{k+p,\infty}_\varrho } + \norm{\bar \bu'}_{W^{k+p-1,\infty}_\varrho }\leq \bar M \,;\]
	for any initial data $(h_0, \bu_0, w_0)\in H^{s,k}(\Omega)^{2+d}$ satisfying the boundary condition $w_0|_{\varrho=\rho_1}=0$ and the incompressibility condition
	\[-(\bar h+h_0)\nabla_\bx\cdot \bu_0-(\nabla_\bx \cH_0)\cdot({\bar \bu}'+\partial_\varrho\bu_0)+\partial_\varrho  w_0=0,\] 
	(denoting $\cH_0(\cdot,\varrho)=\int_\varrho^{\rho_1} h_0(\cdot,\varrho')\dd\varrho'$), the bounds
	\[
	\Norm{\cH_0}_{H^{s+p,k+p}}+\Norm{\bu_0}_{H^{s+p,k+p}}+\norm{\cH_0\big\vert_{\varrho=\rho_0}}_{H^{s+p}_\bx}+\kappa^{1/2}\Norm{h_0}_{H^{s+p,k+p}} =M_0
	\le M
	\]
	and the stable stratification assumption
	\[ \inf_{(\bx,\varrho)\in  \Omega } h_\star\leq \bar h(\varrho)+h_0(\bx,\varrho) \leq h^\star,\]
	the following holds. Denoting
	\[
	(T^\h)^{-1}= C^\h\, \big(1+ \kappa^{-1} \big(\norm{\bar \bu'}_{L^2_\varrho}^2+M_0^2\big)  \big),  
	\]
	as in Lemma~\ref{L.CV-consistency} 
	there exists a unique strong solution $(h^\nh,\bu^\nh,w^\nh)\in \cC^0([0,T^\h];H^{s,k}(\Omega)^{1+d})$ to the non-hydrostatic equations~\eqref{eq:nonhydro-iso-intro} with initial data $(h^\nh,\bu^\nh, w^\nh)\big\vert_{t=0}=(h_0,\bu_0, w_0)$.
	Moreover, one has $h^\nh\in  L^2(0,T^\h;H^{s+1,k}(\Omega))$, $\cH^\nh\in  L^2(0,T^\h;H^{s+2,k}(\Omega))$ and, for any $t\in[0,T^\h]$, the lower and the upper bounds hold
	\[    \inf_{(\bx,\varrho)\in \Omega } \bar h(\varrho)+h^\nh(t,\bx,\varrho) \geq h_\star/3, \quad  \sup_{(\bx,\varrho)\in  \Omega } \bar h(\varrho)+h^\nh(t,\bx,\varrho) \le 3h^\star, \]
	and the estimate below holds true
	\begin{align}&
		\Norm{\cH^\nh(t,\cdot)}_{H^{s,k}}+\Norm{\bu^\nh(t,\cdot)}_{H^{s,k}} +\mu^{1/2}\Norm{w^\nh(t,\cdot)}_{H^{s,k}}+\norm{\cH^\nh\big\vert_{\varrho=\rho_0}(t,\cdot)}_{H^s_\bx} \notag\\ 
		&\quad+\kappa^{1/2}\Norm{h^\nh(t,\cdot)}_{H^{s,k}}  + \mu^{1/2}\kappa^{1/2} \Norm{\nabla_\bx \cH^\nh (t, \cdot)}_{H^{s,k}} \notag\\ 
		&\quad+ \kappa^{1/2} \Norm{\nabla_\bx \cH^\nh}_{L^2(0,t;H^{s,k})} +   \kappa^{1/2} \norm{\nabla_\bx \cH^\nh \big\vert_{\varrho=\rho_0}  }_{ L^2(0,t;H^s_\bx)} \notag\\ 
		&\quad+\kappa \Norm{\nabla_\bx h^\nh}_{L^2(0,t;H^{s,k})} +  \mu^{1/2}\kappa \Norm{\nabla_\bx^2 \cH^\nh }_{L^2(0,t; H^{s,k})} \leq  C\, M_0, \label{eq:control-F-nh}
	\end{align}
	and $(h^\nh,\bu^\nh)$ converges strongly in $L^\infty(0,T;H^{s,k}(\Omega)^{1+d})$ towards $(h^\h,\bu^\h)$ the corresponding solution to the hydrostatic equations~\eqref{eq:hydro-iso-intro}, as $\mu\searrow 0$.
\end{proposition}
\begin{proof}
	We closely follow the proof of Proposition~\ref{P.NONHydro-large-time} and exhibit a bootstrap argument on the functional 
	\begin{align*}\cF(t):&=
		\Norm{\cH^\d(t,\cdot)}_{H^{s,k}}+\Norm{\bu^\d(t,\cdot)}_{H^{s,k}} +\mu^{1/2}\Norm{w^\d(t,\cdot)}_{H^{s,k}}+\norm{\cH^\d\big\vert_{\varrho=\rho_0}(t,\cdot)}_{H^s_\bx} \notag\\ 
		&\quad+\kappa^{1/2}\Norm{h^\d(t,\cdot)}_{H^{s,k}}  + \mu^{1/2}\kappa^{1/2} \Norm{\nabla_\bx \cH^\d (t, \cdot)}_{H^{s,k}} \notag\\ 
		&\quad+ \kappa^{1/2} \Norm{\nabla_\bx \cH^\d}_{L^2(0,t;H^{s,k})} +   \kappa^{1/2} \norm{\nabla_\bx \cH^\d \big\vert_{\varrho=\rho_0}  }_{ L^2(0,t;H^s_\bx)} \notag\\ 
		&\quad+\kappa \Norm{\nabla_\bx h^\d}_{L^2(0,t;H^{s,k})} +  \mu^{1/2}\kappa \Norm{\nabla_\bx^2 \cH^\d }_{L^2(0,t; H^{s,k})} 
	\end{align*}
	where we denote
	\[ h^\d:= h^\nh-h^\h;\quad \cH^\d:=\cH^\nh-\cH^\h; \quad \bu^\d:=\bu^\nh-\bu^\h;\quad w^\d:=w^\nh-w^\h\]
	with the usual notation for $\cH^\nh$, $\cH^\h$, and $w^\h$ is defined by~\eqref{def-wh}. Denoting  by $T^\star$ the maximal existence time of the non-hydrostatic solution provided by Proposition~\ref{P.NONHydro-small-time}, we set
	\begin{multline}\label{eq:control-F-diff}
		T^\nh= \sup \Big\{0 < T < \min(T^\star,T^\h) \, : \; \forall \, t \in (0,T), \; h_\star/3 \le \bar h(\varrho) + h^\nh(t, \bx, \varrho) \le 3h^\star \quad \\
		\text{and} \quad \cF(t) \le \mu^{1/2} M_0 \exp(C_0t), \quad  \cF(t) \le \kappa^{1/2} M_0\Big\},
	\end{multline}
	with $C_0$ sufficiently large (to be determined later on). We will show by the standard continuity argument that $T^\nh=\min(T^\star,T^\h)$, which in turns yields $T^\star> T^\h$ and shows the result. Indeed, the converse inequality $T^\nh=T^\star\leq T^\h$ yields a contradiction by Proposition~\ref{P.NONHydro-small-time} and the desired estimates immediately follow from the control of $\cF$, the bound 
	\begin{equation} \label{eq:control-F-h}
		\Norm{h^\h(t,\cdot)}_{H^{s+1,k+1}} +\Norm{\bu^\h(t,\cdot)}_{H^{s+1,k+1}}+\Norm{\cH^\h(t,\cdot)}_{H^{s+1,k+1}} +\Norm{w^\h(t,\cdot)}_{H^{s+1,k+1}}  \leq  C^\h M_0
	\end{equation}
	provided by Lemma~\ref{L.CV-consistency}, and triangular inequality (when $C$ is chosen sufficiently large).
	
	Let us now derive from Lemma~\ref{L.CV-quasilinear} the necessary estimates for the bootstrap argument. In the following we repeatedly use the triangular inequality to infer from~\eqref{eq:control-F-diff} and~\eqref{eq:control-F-h} the corresponding control~\eqref{eq:control-F-nh} with $C$ depending only $C^\h$, $T^\h$ (and $\kappa\leq 1$). We shall denote by $C$ a constant depending uniquely on $s,k,\bar M,M,h_\star,h^\star$ and $C^\h$, $T^\h$, but not on $C_0$, and which may change from line to line.
	
	By means of~\eqref{eq.quasilin-h-diff}-\eqref{eq.est-quasilin-h-diff} and Lemma~\ref{lem:estimate-transport-diffusion}, we infer from~\eqref{eq:control-F-diff}-\eqref{eq:control-F-h}
	\[
	\kappa^{1/2}\Norm{h^\d}_{L^\infty(0,T;H^{s,k})}+\kappa\Norm{\nabla_\bx h^\d}_{L^2(0,T;H^{s,k})}
	\leq C\, \big( \norm{\cF}_{L^1_T}+\norm{\cF}_{L^2_T} \big).
	\]
	Next, by differentiating with respect to space the first equation of~\eqref{eq.quasilin-diff} using~\eqref{eq.est-quasilin-diff0} and  Lemma~\ref{lem:estimate-transport-diffusion}, we infer
	\[
	\mu^{1/2}\kappa^{1/2}\Norm{\nabla_\bx \cH^\d}_{L^\infty(0,T;H^{s,k})}+\mu^{1/2}\kappa\Norm{\nabla_\bx^2 \cH^\d}_{L^2(0,T;H^{s,k})}
	\leq C\, \big( \norm{\cF}_{L^1_T}+\norm{\cF}_{L^2_T} \big).
	\]
	Now, we use~\eqref{eq.quasilin-diff}-\eqref{eq.est-quasilin-diff0}-\eqref{eq.est-quasilin-diff1}
	and proceeding as in the proof of Proposition~\ref{P.NONHydro-large-time} (together with the above estimates)  we infer that for any $t\in(0,T)$,
	\[
	\cF(t) \leq  C_1 \norm{\cF}_{L^1_t}+C_2 \norm{\cF}_{L^2_t}+
	C_3\, \mu^{1/2}M_0  \,  t\]
	with $C_i$ ($i\in\{1,2,3\}$) depending uniquely on $s,k,\bar M,M,h_\star,h^\star$ and $C^\h$, $T^\h$.
	By using the inequality $ \cF(t) \le \mu^{1/2} M_0 \exp(C_0t)$ from~\eqref{eq:control-F-diff} and the inequality $\tau\leq \exp(\tau)$ (for $\tau\geq0$), we deduce
	\[
	\cF(t) \leq  C_1 \mu^{1/2} M_0 C_0^{-1} \exp(C_0t) +C_2  \mu^{1/2} M_0  (2C_0)^{-1/2} \exp(C_0t) + 	C_3\mu^{1/2}M_0 C_0^{-1} \exp(C_0t).\]
	There remains to choose $C_0$ sufficiently large so that $C_1 C_0^{-1} + C_2 (2C_0)^{-1/2} + C_3  C_0^{-1} <1$, and restrict to $\mu$ sufficiently small so that $\mu^{1/2} M_0 \exp(C_0 T^\h) \leq \mu^{1/2} M_0 C^{1/2}/2\leq \kappa^{1/2}/2$. The upper and lower bounds for $\bar h+h^\nh$ follow immediately from the corresponding ones for $\bar h+h^\h $ provided by Lemma~\ref{L.CV-consistency} and triangular inequality, augmenting $C$ if necessary.
	Then the usual continuity argument yields, as desired, $T^\nh=\min(T^\star,T^\h)$.
\end{proof}

\subsection{Improved convergence rate}\label{S.CV-convergence}

Proposition~\ref{P.CV-control} established the strong convergence for regular well-prepared initial data of the solution to the non-hydrostatic equations, \eqref{eq:nonhydro-iso-intro}, towards the corresponding solution to the hydrostatic equations, \eqref{eq:hydro-iso-intro}, as $\mu \searrow 0$. The convergence rate displayed in the proof is $\mathcal O(\mu^{1/2})$. The aim of this section is to provide an improved and \emph{optimal} convergence rate $\mathcal O (\mu)$. The strategy is based on the interpretation of the  non-hydrostatic solution as an approximate solution to the hydrostatic equations (in the sense of consistency) and the use of the uniform control obtained in Proposition~\ref{P.CV-control}.

\begin{proposition}\label{P.CV-convergence}
	There exists $p\in\NN$ such that for any  $s, k \in \NN$ with $k=s> \frac 52+\frac d 2$ and $\bar M,M,h_\star,h^\star>0$,  there exists $C=C(s,k,\bar M,M,h_\star,h^\star)>0$ such that under the assumptions of Proposition~\ref{P.CV-control} and using the notations therein, 
	\[	\Norm{ h^\nh-h^\h}_{L^\infty(0,T^\h;H^{s-1,k-1})}+\Norm{ \cH^\nh-\cH^\h}_{L^\infty(0,T^\h;H^{s,k})}+\Norm{\bu^\nh-\bu^\h}_{L^\infty(0,T^\h;H^{s,k})}\leq C\,\mu.\]
\end{proposition}
\begin{corollary}
	Incrementing $p\in\NN$, we find that  for any  $s, k \in \NN$ such that $k=s> \frac 32+\frac d 2$,
	\[	\Norm{ h^\nh-h^\h}_{L^\infty(0,T^\h;H^{s,k})}+\Norm{ \cH^\nh-\cH^\h}_{L^\infty(0,T^\h;H^{s+1,k+1})}+\Norm{\bu^\nh-\bu^\h}_{L^\infty(0,T^\h;H^{s+1,k+1})}\leq C\,\mu\]
	with $C=C(s+1,k+1,\bar M,M,h_\star,h^\star)>0$.
\end{corollary}
\begin{proof} Since all arguments of the proof have been already used in slightly different contexts, we only quickly sketch the argument.

	For any $p'\in\NN$, we may use Proposition~\ref{P.CV-control} with indices $s+p'$ and $k+p'$ to infer the existence of the non-hydrostatic solution $(h^\nh,\bu^\nh,w^\nh)\in  \cC([0,T^\h];H^{s+p',k+p'}(\Omega)^{2+d})$ and the control
	\[	\sup_{t\in [0,T^\h]}\Big( 	\Norm{\cH^\nh(t,\cdot)}_{H^{s+p',k+p'}}+\Norm{\bu^\nh(t,\cdot)}_{H^{s+p',k+p'}} +\norm{\cH^\nh\big\vert_{\varrho=\rho_0}(t,\cdot)}_{H^{s+p'}_\bx} \Big) \leq C\, M_0.
	\]
	By using $h^\nh=-\de_\varrho \cH^\nh$ and the divergence-free condition
	\[ w^\nh=  ({\bar \bu}+\bu^\nh)\cdot\nabla_\bx \cH^\nh-\int_\varrho^{\rho_1}   \nabla_\bx\cdot((\bar h+ h^\nh)({\bar \bu}+\bu^\nh)) \dd\varrho'\]
	we obtain (augmenting $C$ if necessary)
	\[	\sup_{t\in [0,T^\h]}\Big( 	\Norm{h^\nh(t,\cdot)}_{H^{s+p'-1,k+p'-1}}+\Norm{w^\nh(t,\cdot)}_{H^{s+p'-1,k+p'-1}} \Big)\leq C\, M_0
	\]
	and hence, by Corollary~\ref{C.Poisson} (specifically~\eqref{ineq:est-Ptame-nonhydro}), Poincar\'e inequality~\eqref{eq.Poincare} and choosing $p'$ sufficiently large, that $P_\nh(\cdot,\varrho):=P^\nh(\cdot,\varrho)-\int_{\rho_0}^\varrho \varrho' h^\nh(\cdot,\varrho')\dd\varrho'$ satisfies
	\[	\sup_{t\in [0,T^\h]} \Norm{P_\nh(t,\cdot)}_{H^{s+1,k+1}}\leq C\,\mu \, M_0.\]
	From this estimate we infer (by Lemma~\ref{L.CV-consistency}) that
	$
	h^\d:=h^\nh-h^\h$ and $ 
	\bu^\d:=\bu^\nh-\bu^\h$
	satisfies 
	\begin{equation}\label{eq:difference-CV}
		\begin{aligned}
			\de_t \cH^\d 
			+ (\bar \bu+\bu^\nh) \cdot \nabla_\bx \cH^\d
			+\int_\varrho^{\rho_1} (\bar \bu'+\de_\varrho \bu^\nh) \cdot \nabla_\bx \cH^\d \, \dd \varrho' 
			+\int_\varrho^{\rho_1} (\bar h + h^\nh) \nabla_\bx \cdot \bu^\d \, \dd\varrho' & \\
			+\int_\varrho^{\rho_1} \bu^\d \cdot \nabla_\bx h^\h + h^\d \nabla_\bx \cdot \bu^\h \, \dd \varrho' & = \kappa\Delta_\bx\cH^\d,\\
			\de_t \bu^\d 
			+ \big((\bar\bu + \bu^\nh - \kappa \tfrac{\nabla_\bx h^\nh}{\bar h + h^\nh} ) \cdot \nabla_\bx\big) \bu^\d 
			+ {\frac{\rho_0}{\varrho} \nabla_\bx \cH^\d|_{\varrho=\rho_0}} 
			+ {\frac 1 \varrho \int_{\rho_0}^\varrho \nabla_\bx \cH^\d \, \dd\varrho'} & \\
			+ \big(( \bu^\d - \kappa ( \tfrac{\nabla_\bx h^\nh}{\bar h + h^\nh} - \tfrac{\nabla_\bx h^\h}{\bar h+h^\h})) \cdot \nabla_\bx\big) \bu^\h  &=  \bR^\nh,
		\end{aligned}
	\end{equation}
	with $\bR^\nh:=- \frac{\nabla_\bx{ {\Pnh} }}{\varrho} 
	- \frac{\nabla_\bx \cH^\nh}{\varrho (\bar h+h^\nh)} \de_\varrho{ {\Pnh} } $ satisfying
	(by Lemma~\ref{L.product-Hsk} and Lemma~\ref{L.composition-Hsk-ex}) the bound
	\[	\sup_{t\in [0,T^\h]}\Norm{\bR^\nh(t,\cdot)}_{H^{s,k}}\leq C\,\mu\, M_0.\]
	From this, inspecting the proof of Lemma~\ref{L.CV-quasilinear}, we infer that as long as
	\[ \cF_{s,k}:=\Norm{h^\d}_{H^{s-1, k-1}} + \Norm{\cH^\d}_{H^{s,k}}+\norm{\cH^\d\big\vert_{\varrho=\rho_0}}_{H^s_\bx}+\Norm{\bu^\d}_{H^{s,k}} +\kappa^{1/2}\Norm{h^\d}_{H^{s,k}} \leq \kappa^{1/2}M_0,\]
	one has for any $\balpha\in\NN^d$ and $j\in\NN$ such that $|\balpha|+j\leq s$ that $\cH^{(\balpha,j)}:=\de_\bx^\balpha\de_\varrho^j\cH^\d$,
	$\bu^{(\balpha,j)}:=\de_\bx^\balpha\de_\varrho^j\bu^\d$ and $h^{(\balpha,j)}:=\de_\bx^\balpha\de_\varrho^jh^\d$ satisfy
	\[
	\begin{aligned}
		\partial_t  \cH^{(\balpha,j)}+ (\bar \bu + \bu^\nh) \cdot\nabla_\bx \cH^{(\balpha,j)} +\Big\langle \int_\varrho^{\rho_1} (\bar \bu' + \de_\varrho\bu^\nh) \cdot\nabla_\bx \cH^{(\balpha,j)} \, \dd\varrho' \phantom{-\kappa\Delta_\bx   \cH^{(\balpha,j)}}\qquad  &\\
		+\int_\varrho^{\rho_1} (\bar h+h^\nh)\nabla_\bx \cdot\bu^{(\balpha,j)} \dd\varrho'\Big\rangle_{j=0}-\kappa\Delta_\bx   \cH^{(\balpha,j)}&=R_{\balpha,j},\\
		\partial_t\bu^{(\balpha,j)}+\big(({\bar \bu}+\bu^\nh-\kappa\tfrac{\nabla_\bx h^\nh}{\bar h+ h^\nh})\cdot\nabla_\bx\big) \bu^{(\balpha,j)}
		+\Big\langle \frac{\rho_0}{ \varrho }\nabla_\bx \cH^{(\balpha,j)}\big\vert_{\varrho=\rho_0} \qquad &\\
		+\frac 1 \varrho\int_{\rho_0}^\varrho  \nabla_\bx  \cH^{(\balpha,j)} \dd\varrho' \Big\rangle_{j=0} &=\bR^\nh_{\balpha,j},
	\end{aligned}
	\]
	and
	\[
	\partial_t  h^{(\balpha,j)}+(\bar\bu+\bu)\cdot \nabla_\bx  h^{(\balpha,j)}
	=\kappa\Delta_\bx  \partial_\varrho^{j} h^{(\balpha,j)}+r_{\balpha,j}+\nabla_\bx \cdot \br_{\balpha,j},
	\]
	with 
	\[
	\Norm{R_{\balpha,j}}_{L^2(\Omega) }+\Norm{ \bR^\nh_{\balpha,j}}_{L^2(\Omega)}
	\leq  
	C  \, \big(  \cF_{s,k} + M_0\kappa  \Norm{\nabla_\bx h^\d}_{H^{s,k}} \big)	+C\, \mu\, M_0 
	\]
	and
	\[
	\kappa^{1/2} \Norm{r_{\balpha,j}}_{L^2(\Omega) }  + \Norm{\br_{\balpha,j}}_{L^2(\Omega) }  \leq C\,\cF_{s,k}.
	\]
	We may then proceed as in the proof of Proposition~\ref{P.regularized-large-time-WP}, and bootstrap the control
	\[\cF_{s,k}(t)+ \kappa^{1/2} \Norm{\nabla_\bx \cH^\d}_{L^2(0,t;H^{s,k})} +  \kappa^{1/2} \norm{\nabla_\bx \cH^\d \big\vert_{\varrho=\rho_0}  }_{ L^2(0,t;H^s_\bx)} +\kappa \Norm{\nabla_\bx h^\d}_{L^2(0,t;H^{s,k})} \leq C \,\mu\, M_0
	\]
	(choosing $C$ large enough) on the time interval $[0,T^\h]$. This concludes the proof.
\end{proof}

\appendix

\section{Product, composition and commutator estimates}\label{S.Appendix}

In this section we collect useful estimates in the spaces
$H^{s,k}(\Omega)$ introduced in~\eqref{def:Hsk}.
Our results will follow from standard estimates in Sobolev spaces $H^s(\RR^d)$ (see {\em e.g.} \cite{Lannesbook}*{Appendix~B} and references therein), and the following continuous embedding. Henceforth we denote  $\Omega=\RR^d\times (\rho_0,\rho_1)$. 
\begin{lemma}\label{L.embedding}
	For any $s\in \RR$ and $\rho_0<\rho_1$, $H^{s+1/2,1}(\Omega) \subset \cC^0([\rho_0,\rho_1]; H^s(\RR^d) )$ and there exists $C>0$ such that for any $F\in H^{s+1/2,1}(\Omega) $,
	\[ \max_{\varrho \in [\rho_0, \rho_1]}\norm{F(\cdot,\varrho)}_{H^s_\bx } \leq   C \Norm{F}_{H^{s+1/2,1}} .\]
	More generally, for any $k\geq 1$, $H^{s+1/2,1}(\Omega) \subset \bigcap_{j=0}^{k-1}\cC^j([\rho_0,\rho_1]; H^{s-j}(\RR^d) )$, and in particular, for any $s_0>d/2$ and $j\in\NN$, $H^{j+s_0+\frac12,j+1}(\Omega)\subset\big( \cC^j(\Omega)\cap W^{j,\infty}(\Omega)\big)$.
\end{lemma}
\begin{proof}
	By a density argument, we only need to prove the inequality for smooth functions $F$. Set $\phi:[\rho_0,\rho_1]\to\RR^+$ a smooth function such that $\phi(\rho_0)=0$ and $\phi(\varrho)=1$ if $\varrho\geq \tfrac{\rho_0+\rho_1}2$,  and deduce that for any $\varrho\geq \tfrac{\rho_0+\rho_1}2$, recalling the notation $\Lambda^s:=(\Id-\Delta_\bx)^{s/2}$,
	\begin{align*} 
		\int_{\RR^d} (\Lambda^s F)^2(\bx,\varrho) \dd\bx&= \int_{\RR^d}\int_{\rho_0}^\varrho \partial_\varrho \big(\phi(\varrho')(\Lambda^s F)^2(\bx,\varrho')\big)\dd \varrho'\dd\bx \\
		&\leq 2\norm{\phi}_{L^\infty_\varrho}\int_{\rho_0}^{\rho_1}\norm{\Lambda^s F(\cdot, \varrho)}_{H^{1/2}_\bx}\norm{\Lambda^s \partial_\varrho F(\cdot,\varrho)}_{H^{-1/2}_\bx}\dd \varrho\\
		&\quad +\norm{\phi'}_{L^\infty}\int_{\rho_0}^{\rho_1}\norm{\Lambda^s F(\cdot, \varrho)}_{L^2}\norm{\Lambda^s F(\cdot, \varrho)}_{L^2}\dd \varrho\\
		&\lesssim \Norm{F}_{H^{s+1/2,0}}^2+\Norm{\partial_\varrho F}_{H^{s-1/2,0}}^2
	\end{align*}
	Using symmetrical considerations when $\varrho< \tfrac{\rho_0+\rho_1}2$, we prove {the claimed inequality, which yields the first continuous embedding. Higher-order embeddings follow immediately.}
\end{proof}

Recall the notation
\[ A_s+\big\langle B_s\big\rangle_{s>s_\star} =\begin{cases}
	A_s &\text{ if } s\leq s_\star\,,\\
	A_s+B_s& \text{otherwise.}
\end{cases}\]

\subsection*{Product estimates}
Recall the standard product estimates in Sobolev spaces $H^s(\RR^d)$.
\begin{lemma}\label{L.product-Hs}
	Let $d\in\NN^\star$, $s_0>d/2$. 
	\begin{enumerate}
		\item \label{L.product-Hs-1} For any $s,s_1,s_2\in\RR$ such that $s_1\geq s$, $s_2\geq s$ and $s_1+s_2\geq s+s_0$, there exists $C>0$ such that for any $f\in H^{s_1}(\RR^d)$ and $g\in H^{s_2}(\RR^d)$, $fg\in H^s(\RR^d)$ and
		\[\norm{fg}_{H^s}\leq C \norm{f}_{H^{s_1}}\norm{g}_{H^{s_2}}.\]
		\item \label{L.product-Hs-2} For any $s\geq -s_0$, there exists $C>0$ such that for any $f\in H^{s}(\RR^d)$ and $g\in H^{s}(\RR^d)\cap H^{s_0}(\RR^d)$, $fg\in H^s(\RR^d)$ and
		\[\norm{fg}_{H^s}\leq C \norm{f}_{H^{s_0}}\norm{g}_{H^{s}} +C\left\langle  \norm{f}_{H^{s}}\norm{g}_{H^{s_0}}\right\rangle_{s>s_0}.\]
		\item \label{L.product-Hs-3} For any $s_1,\dots,s_n\in\RR$ such that $s_i\geq 0$ and $s_1+\dots+s_n\geq (n-1)s_0$, there exists $C>0$ such that for any $(f_1,\dots,f_n)\in H^{s_1}(\RR^d)\times \cdots\times H^{s_n}(\RR^d) $, $\prod_{i=1}^n  f_i\in L^2(\RR^d)$ and
		\[\norm{\prod_{i=1}^n  f_i }_{L^2}\leq C \prod_{i=1}^n \norm{f_i}_{H^{s_i}}.\]
	\end{enumerate}
\end{lemma}
Let us turn to product estimates in $H^{s,k}(\Omega)$ spaces.
\begin{lemma}\label{L.product-Hsk}
	Let $d\in\NN^\star$, $s_0>d/2$. Let  $s,k\in\NN$ such that  $s\geq s_0+\frac12$ and $1\leq k\leq s$. Then $H^{s,k}(\Omega)$ is a Banach algebra and there exists $C>0$ such that for any $F,G\in H^{s,k}(\Omega)$,
	\[\Norm{FG}_{H^{s,k}}\leq C\Norm{F}_{H^{s,k}}\Norm{G}_{H^{s,k}}.\]
	Moreover, if  $s\geq s_0+\frac32$ and $2\leq k\leq s$, then there exists $C'>0$ such that for any $F,G\in H^{s,k}(\Omega)$,
	\[\Norm{FG}_{H^{s,k}}\leq C'\Norm{F}_{H^{s,k}}\Norm{G}_{H^{s-1,k-1}}+C'\Norm{F}_{H^{s-1,k-1}}\Norm{G}_{H^{s,k}},\]
	and if $s\geq s_0+\frac32$ and $k=1$, then there exists $C''>0$ such that for any $F,G\in H^{s,k}(\Omega)$,
	\[\Norm{FG}_{H^{s,1}}\leq C''\Norm{F}_{H^{s,1}}\Norm{G}_{H^{s-1,1}}+C''\Norm{F}_{H^{s-1,1}}\Norm{G}_{H^{s,1}}.\]
\end{lemma}
\begin{proof}
	We set two multi-indices $\bbeta=(\bbeta_\bx,\bbeta_\varrho)\in\NN^{d+1}$ and $\bgamma=(\bgamma_\bx,\bgamma_\varrho)\in\NN^{d+1}$ being such that $|\bbeta|+|\bgamma|\leq s$ and $\bbeta_\varrho+\bgamma_\varrho\leq k$. Let us first assume furthermore that $\bgamma_\varrho\leq k-1$ and $|\bgamma|\leq s-1$. Then 
	\begin{align*}
		\Norm{(\partial^\bbeta F)(\partial^\bgamma G)}_{L^2(\Omega)}^2
		&\lesssim  \int_{\rho_0}^{\rho_1} \norm{\partial^\bbeta F(\cdot,\varrho)}_{H^{s-|\bbeta|}_\bx}^2\norm{\partial^\bgamma G(\cdot,\varrho)}_{H^{s-|\bgamma|-\frac12}_\bx}^2\dd \varrho\\
		&\lesssim \Norm{\partial^\bbeta F}_{H^{s-|\bbeta|,0}}^2\Norm{\partial^\bgamma G}_{H^{s-|\bgamma|,1}}^2\leq \Norm{ F}_{H^{s,k}}^2\Norm{ G}_{H^{s,k}}^2.
	\end{align*}
	where we used Lemma~\ref{L.product-Hs}(\ref{L.product-Hs-1}) with $(s,s_1,s_2)=(0,s-|\bbeta|,s-|\bgamma|-\frac12)$, and Lemma~\ref{L.embedding}. If $\bgamma_\varrho=k$ or $|\bgamma|=s$, and since $1\leq k\leq s$, we have $\bbeta_\varrho\leq k-1$ and $|\bbeta|\leq s-1$ and we may make use of the symmetric estimate. Hence the proof of the first statement follows from Leibniz rule.
	
	For the second statement, we assume first that $\max(\{\bbeta_\varrho,\bgamma_\varrho\})\leq k-1$ and $\max(\{|\bbeta|,|\bgamma|\})\leq s-1$. Then, using Lemma~\ref{L.product-Hs} with $(s,s_1,s_2)=(0,s-|\bbeta|-\frac12,s-|\bgamma|-1)$ (recall $s\geq s_0+\frac32$), and Lemma~\ref{L.embedding},
	\begin{align*}
		\Norm{(\partial^\bbeta F)(\partial^\bgamma G)}_{L^2(\Omega)}
		&\lesssim  \norm{\norm{\partial^\bbeta F(\cdot,\varrho)}_{H^{s-|\bbeta|-\frac12}_\bx}\norm{\partial^\bgamma G(\cdot,\varrho)}_{H^{s-|\bgamma|-1}_\bx} }_{L^2_\varrho}\\
		&\lesssim \Norm{ F}_{H^{s,\bbeta_\varrho+1}}\Norm{ G}_{H^{s-1,\bgamma_\varrho}}\leq \Norm{F}_{H^{s,k}}\Norm{G}_{H^{s-1,k-1}},
	\end{align*}
	Then if $\bbeta_\varrho=k$ or $|\bbeta|= s$, we have (since $s\geq k\geq 2$) $\bgamma_\varrho\leq k-2$ and $|\bgamma|\leq s-2$, and we infer
	\begin{align*}
		\Norm{(\partial^\bbeta F)(\partial^\bgamma G)}_{L^2(\Omega)}
		&\lesssim  \norm{ \norm{\partial^\bbeta F(\cdot,\varrho)}_{H^{s-|\bbeta|}_\bx}\norm{\partial^\bgamma G(\cdot,\varrho)}_{H^{s-|\bgamma|-\frac32}_\bx} }_{L^2_\varrho}\\
		&\lesssim  \Norm{ F}_{H^{s,\bbeta_\varrho}}\Norm{ G}_{H^{s-1,\bgamma_\varrho+1}}\leq \Norm{F}_{H^{s,k}}\Norm{G}_{H^{s-1,k-1}}.
	\end{align*}
	Of course we have the symmetrical result when $\bgamma_\varrho=k$ or $|\bgamma|= s$, which complete the proof.
	
	{Finally, for the last statement, we consider first the case  $\bbeta_\varrho=0$ and $\max(\{|\bbeta|,|\bgamma|\})\leq s-1$, and infer as above
		\[ \Norm{(\partial^\bbeta F)(\partial^\bgamma G)}_{L^2(\Omega)} \lesssim  \norm{\norm{\partial^\bbeta F(\cdot,\varrho)}_{H^{s-|\bbeta|-\frac12}_\bx}\norm{\partial^\bgamma G(\cdot,\varrho)}_{H^{s-|\bgamma|-1}_\bx} }_{L^2_\varrho}
		\lesssim \Norm{ F}_{H^{s,1}}\Norm{ G}_{H^{s-1,1}}.\]
		The case $\bbeta_\varrho=1$ (and hence $\bgamma_\varrho=0$) and $\max(\{|\bbeta|,|\bgamma|\})\leq s-1$ is treated symmetrically. Then if $|\bbeta|= s$ we have  $\bgamma_\varrho=|\bgamma|=0$, and we infer
		\[
		\Norm{(\partial^\bbeta F)(\partial^\bgamma G)}_{L^2(\Omega)}
		\lesssim  \norm{ \norm{\partial^\bbeta F(\cdot,\varrho)}_{H^{s-|\bbeta|}_\bx}\norm{\partial^\bgamma G(\cdot,\varrho)}_{H^{s-|\bgamma|-\frac32}_\bx} }_{L^2_\varrho}
		\lesssim \Norm{F}_{H^{s,1}}\Norm{G}_{H^{s-1,1}}.\]
		The case $|\bgamma|= s$ is treated symmetrically, and the proof is complete.}
\end{proof}

\subsection*{Composition estimates}
Let us recall the standard composition estimate in Sobolev spaces $H^s(\RR^d)$.
\begin{lemma}\label{L.composition-Hs}
	Let $d\in\NN^\star$, $s_0>d/2$. For any $\varphi\in \cC^\infty(\RR;\RR)$ such that $\varphi(0)=0$, and any $M>0$, there exists $C>0$ such that for any $f\in H^{s_0}(\RR^d) \cap H^s(\RR^d)$ with $\norm{f}_{H^{s_0}} \leq M$, one has $\varphi(f)\in H^s(\RR^d)$ and
	\[\norm{\varphi(f)}_{H^s} \leq C\norm{f}_{H^s}.\]
\end{lemma}
We now consider composition estimates in $H^{s,k}(\Omega)$.
\begin{lemma}\label{L.composition-Hsk}
	Let $d\in\NN^\star$, $s_0>d/2$. Let $s,k\in\NN$ with $s\geq s_0+\frac12$ and $1\leq k\leq s$, and $M>0$. There exists $C>0$ such that for any $\varphi\in W^{s,\infty}(\RR;W^{k,\infty}((\rho_0,\rho_1)))$ with $\varphi(0;\cdot)\equiv 0$, and any $F\in H^{s,k}(\Omega)$ such that $\Norm{F}_{H^{s,k}} \leq M$, then $\varphi\circ F:(\bx,\varrho)\mapsto\varphi(F(\bx,\varrho);\varrho)\in H^{s,k}(\Omega)$ and
	\[ \Norm{\varphi\circ F}_{H^{s,k}}\leq C\norm{\varphi}_{W^{s,\infty}( \RR ; W^{k,\infty}( (\rho_0,\rho_1) ) )}\Norm{F}_{H^{s,k}}.\]
	If moreover $s\geq s_0+\frac32$ and $2\leq k\leq s$, then there exists $C'>0$ such that for any $F\in H^{s,k}(\Omega)$ such that $\Norm{F}_{H^{s-1,k-1}} \leq M$, 
	\[ \Norm{\varphi\circ F}_{H^{s,k}}\leq C'\norm{\varphi}_{W^{s,\infty}(\RR;W^{k,\infty}((\rho_0,\rho_1)))}\Norm{F}_{H^{s,k}}.\]
\end{lemma}
\begin{proof}
	Let $\balpha=(\balpha_\bx,\balpha_\varrho)\in\NN^{d+1}\setminus\{{\bf 0}\}$ with $0\leq|\balpha|\leq s$ and $0\leq\balpha_\varrho\leq k$. We have by Fa\`a di Bruno's formula
	\[ \Norm{\partial^\balpha (\varphi\circ F)}_{L^2(\Omega)} \lesssim
	\sum\Norm{\big((\de_1^i\de_2^j\varphi)\circ F \big)\,  (\partial^{\balpha^{i,j}_1} F)\cdots (\partial^{\balpha^{i,j}_i} F)}_{L^2(\Omega)}, \]
	where $i,j\in\NN$ with $i+j\leq |\balpha|\leq s$, and the multi-indices $\balpha^{i,j}_\ell=(\balpha^{i,j}_{\ell,\bx},\balpha^{i,j}_{\ell,\varrho})\in\NN^{d+1}\setminus\{{\bf 0}\}$ satisfy $ \sum_{\ell=1}^i \balpha^{i,j}_{\ell,\bx}=\balpha_\bx $ and $ j+\sum_{\ell=1}^i \balpha^{i,j}_{\ell,\varrho}=\balpha_\varrho $. If $i=0$ then we have from the mean value theorem that for any $0\leq j\leq k$
	\[ \Norm{(\de_2^j\varphi)\circ F}_{L^2(\Omega)}=\Norm{(\de_2^j\varphi)\circ F-(\de_2^j\varphi)\circ 0}_{L^2(\Omega)}\leq \norm{\de_1\de_2^j\varphi}_{L^\infty(\RR\times (\rho_0,\rho_1))} \Norm{F}_{L^2(\Omega)}.\]
	The case $i=1$ is straightforward, and we now focus on the case $i\geq 2$. We assume without loss of generality that $|\balpha^{i,j}_{1,\varrho}|\geq |\balpha^{i,j}_{2,\varrho} |\geq \cdots \geq |\balpha^{i,j}_{i,\varrho} |$ and remark that for $\ell\neq 1$, $|\balpha^{i,j}_{\ell,\varrho}|\leq k-1$ (otherwise $|\balpha^{i,j}_{1,\varrho}|+|\balpha^{i,j}_{\ell,\varrho}|=2k>k\geq |\balpha_\varrho|$) and $ |\balpha^{i,j}_{\ell}  |\leq s-|\balpha^{1,j}_1|\leq s-1 $.
	Hence we have
	\begin{align*}\textstyle\Norm{\prod_{\ell=1}^i (\partial^{\balpha^{i,j}_\ell} F)}_{L^2(\Omega)}
		&\textstyle\lesssim \norm{\norm{\partial^{\balpha^{i,j}_1} F}_{H^{s-|\balpha^{i,j}_1|}_\bx} \big(\prod_{\ell=2}^i \norm{\partial^{\balpha^{i,j}_\ell} F}_{H^{s-|\balpha^{i,j}_2|-\frac12}_\bx}\big) }_{L^2_\varrho}\\
		&\textstyle\lesssim \Norm{ F}_{H^{s,\balpha^{i,j}_{1,\varrho}}} \big(\prod_{\ell=2}^i \Norm{ F}_{H^{s,\balpha^{i,j}_{\ell,\varrho}+1}} \big)\leq \Norm{ F}_{H^{s,k}}^i
	\end{align*}
	where we used Lemma~\ref{L.product-Hs}(\ref{L.product-Hs-3}) and $(i-1)(s-\frac12)\geq (i-1)s_0 $ and Lemma~\ref{L.embedding}. The first claim follows. 
	
	Now we assume additionally that $k\geq 2$ and $s\geq s_0+\frac32$. The cases $i\in\{0,1\}$ can be treated exactly as above and we deal only with the case $i\geq2$, ordering  $|\balpha^{i,j}_{1,\varrho}|\geq |\balpha^{i,j}_{2,\varrho}| \geq \cdots \geq |\balpha^{i,j}_{i,\varrho}| $ as above. Assume first that $|\balpha^{i,j}_{1,\varrho}|=k\geq 2$. Then  for all $\ell\neq 1$, $|\balpha^{i,j}_{\ell,\varrho}|=0$ and $|\balpha^{i,j}_{\ell}|\leq s-2$, and we conclude as before with
	\[\textstyle\Norm{\prod_{\ell=1}^i (\partial^{\balpha^{i,j}_\ell} F)}_{L^2(\Omega)}
	\lesssim \norm{\norm{\partial^{\balpha^{i,j}_1} F}_{H^{s-|\balpha^{i,j}_1|}_\bx} \big(\prod_{\ell=1}^i \norm{\partial^{\balpha^{i,j}_\ell} F}_{H^{s-|\balpha^{i,j}_\ell|-\frac32}_\bx} \big)}_{L^2_\varrho} \lesssim \Norm{ F}_{H^{s,k}}\Norm{ F}_{H^{s-1,1}}^{i-1}.
	\]
	Otherwise we have $|\balpha^{i,j}_{2,\varrho}|\leq |\balpha^{i,j}_{1,\varrho}|\leq k-1$ and $|\balpha^{i,j}_{2}|\leq s-|\balpha^{i,j}_{1}|\leq s-1$ and notice that for $\ell \geq 3$, $|\balpha^{i,j}_{\ell,\varrho}|\leq k-2$ (since otherwise we have the contradiction $|\balpha^{i,j}_{1,\varrho}|+ |\balpha^{i,j}_{2,\varrho}|+|\balpha^{i,j}_{3,\varrho}|\geq 3(k-1)\geq k+1\geq |\balpha_\varrho|+1$) and $|\balpha^{i,j}_{\ell}|\leq s-|\balpha^{i,j}_{1}|-|\balpha^{i,j}_{2}|\leq s-2$. Hence
	\begin{align*}\textstyle\Norm{\prod_{\ell=1}^i (\partial^{\balpha^{i,j}_\ell} F)}_{L^2(\Omega)}
		&\textstyle\lesssim  \norm{\norm{\partial^{\balpha^{i,j}_1} F}_{H^{s-|\balpha^{i,j}_1|-\frac12}_\bx}\norm{\partial^{\balpha^{i,j}_2} F}_{H^{s-|\balpha^{i,j}_2|-1}_\bx}\big(\prod_{\ell=3}^i \norm{\partial^{\balpha^{i,j}_\ell} F}_{H^{s-|\balpha^{i,j}_\ell|-\frac32}_\bx}\big)}_{L^2_\varrho}\\
		&\textstyle\lesssim \Norm{ F}_{H^{s,\balpha^{i,j}_{1,\varrho}+1}}\Norm{ F}_{H^{s-1,\balpha^{i,j}_{2,\varrho}}} \big(\prod_{\ell=3}^i\Norm{ F}_{H^{s-1,\balpha^{i,j}_{\ell,\varrho}+1}}\big) \lesssim \Norm{ F}_{H^{s,k}}\Norm{ F}_{H^{s-1,k-1}}^{i-1}.
	\end{align*}
	This concludes the proof.
\end{proof}

We shall apply the above to estimate quantities such as (but not restricted to) 
\[ \Phi:(\bx,\varrho)\in\Omega\mapsto\frac{h(\bx,\varrho)}{\bar h(\varrho)+h(\bx,\varrho)},\]
with $\bar h\in W^{k,\infty}((\rho_0,\rho_1))$ and  $h\in H^{s,k}(\RR^d)$ satisfying the condition $\inf_{(\bx,\varrho)\in\Omega}\bar h(\varrho)+h(\bx,\varrho)\geq h_\star>0$. Let us detail the result and its proof for this specific example.
\begin{lemma}\label{L.composition-Hsk-ex}
	Let $d\in\NN^\star$, $s_0>d/2$. Let $s,k\in\NN$ with $s\geq s_0+\frac12$ and $1\leq k\leq s$, and $M,\bar M,h_\star>0$. There exists $C>0$ 
	such that for any $\bar h\in  W^{k,\infty}((\rho_0,\rho_1))$ with $\norm{\bar h}_{W^{k,\infty}_\varrho}\leq \bar M$ and any $h\in H^{s,k}(\Omega)$ with $\Norm{h}_{H^{s,k}}\leq M$ and satisfying the condition $\inf_{(\bx,\varrho)\in\Omega}\bar h(\varrho)+h(\bx,\varrho)\geq h_\star$, then 
	\[ \Phi:(\bx,\varrho)\mapsto\frac{h(\bx,\varrho)}{\bar h(\varrho)+h(\bx,\varrho)} \in H^{s,k}(\Omega),\]
	and
	\[ \Norm{\Phi}_{H^{s,k}} \leq C \Norm{h}_{H^{s,k}}.\]
	If moreover $s> \frac d2+\frac32$ and $2\leq k\leq s$, then the above holds for any $h\in H^{s,k}(\Omega)$ with $\Norm{h}_{H^{s-1,k-1}}\leq M$.
\end{lemma}
\begin{proof}
	We can write $\Phi=\varphi\circ h$ with $\varphi(\cdot,\varrho)=f(\cdot,\bar h(\varrho))$ where $f\in \cC^\infty(\RR^2)$ is set such that $f(y,z)=\frac{y}{y+z}$ on  $\omega:=\{(y,z) \ : \ \ |y|\leq  \Norm{h}_{L^\infty(\Omega)} ,\ |z|\leq  \norm{\bar h}_{L^\infty((\rho_0,\rho_1))},\ y+z\geq h_\star \}$. We can construct $f$ as above such that the control of $\norm{\varphi}_{W^{s,\infty}(\RR;W^{k,\infty}((\rho_0,\rho_1)))} $ depends only on  $\Norm{h}_{L^\infty(\Omega)} $ (which is bounded appealing to Lemma~\ref{L.embedding}, if $h \in H^{s,k}$ with $s>\frac d2 + \frac 12, \, 1 \le k \le s$), $ \norm{\bar h}_{W^{k,\infty}((\rho_0,\rho_1))}$ and $h_\star>0$.
	The result is now a direct application of Lemma~\ref{L.composition-Hsk}.\end{proof}

\subsection*{Commutator estimates}
We now recall standard commutator estimates in $H^s(\RR^d)$.
\begin{lemma}\label{L.commutator-Hs}
	Let $d\in\NN^\star$, $s_0>d/2$ and $s\geq0$. 
	\begin{enumerate}
		\item \label{L.commutator-Hs-1} For any $s_1,s_2\in\RR$ such that $s_1\geq s$, $s_2\geq s-1$ and $s_1+s_2\geq s+s_0$, there exists $C>0$ such that for any $f\in H^{s_1}(\RR^d)$ and $g\in H^{s_2}(\RR^d)$, $[\Lambda^s,f]g:=\Lambda^s(fg)-f\Lambda^s g\in L^2(\RR^d)$ and
		\[\norm{[\Lambda^s,f]g}_{L^2}\leq C \norm{f}_{H^{s_1}}\norm{g}_{H^{s_2}}.\]
		\item \label{L.commutator-Hs-2} There exists $C>0$ such that for any $f\in L^\infty(\RR^d)$ such that $\nabla f \in H^{s-1}(\RR^d)\cap H^{s_0}(\RR^d)$ and for any $g\in H^{s-1}(\RR^d)$, one has  $[\Lambda^s,f]g\in L^2(\RR^d)$ and
		\[\norm{[\Lambda^s,f]g}_{L^2}\leq C \norm{\nabla f}_{H^{s_0}}\norm{g}_{H^{s-1}} +C\left\langle  \norm{\nabla f}_{H^{s-1}}\norm{g}_{H^{s_0}}\right\rangle_{s>s_0+1}.\]
		\item \label{L.commutator-Hs-3} There exists $C>0$ such that for any $f,g\in H^{s}(\RR^d)\cap H^{s_0+1}(\RR^d)$, the symmetric commutator $[\Lambda^s;f,g]:=\Lambda^s(fg)-f\Lambda^s g - g\Lambda^s f \in L^2(\RR^d)$ and
		\[\norm{[\Lambda^s;f,g]}_{L^2}\leq C \norm{f}_{H^{s_0+1}}\norm{g}_{H^{s-1}} +C \norm{f}_{H^{s-1}}\norm{g}_{H^{s_0+1}}  .\]
	\end{enumerate}
	{The validity of the above estimates persist when replacing the operator $\Lambda^s$ with the operator $\partial^\balpha$ with $\balpha\in\NN^d$ a multi-index such that $|\balpha|\leq s$.}
\end{lemma}
We conclude with commutator estimates in the spaces $H^{s,k}(\Omega)$.
\begin{lemma}\label{L.commutator-Hsk}
	Let $d\in\NN^\star$, $s_0>d/2$. Let $s\geq s_0+\frac32$ and $k\in\NN$ such that $2\leq k\leq s$. Then there exists $C>0$ such that for any $\balpha=(\balpha_\bx,\balpha_\varrho)\in\NN^{d+1}$ with $|\balpha|\leq s$ and $\balpha_\varrho\leq k$, one has
	\[\Norm{[\partial^\balpha,F]G}_{L^2(\Omega)}\leq C\Norm{F}_{H^{s,k}}\Norm{G}_{H^{s-1,\min(\{k,s-1\})}}.\]
\end{lemma}
\begin{proof}
	We set two multi-indices $\bbeta=(\bbeta_\bx,\bbeta_\varrho)\in\NN^{d+1}$ and $\bgamma=(\bgamma_\bx,\bgamma_\varrho)\in\NN^{d+1}$ with $\bbeta+\bgamma=\balpha$, and $|\bgamma|\leq s-1$. Assume first that $\bbeta_\varrho\leq k-1$ and $|\bbeta|\leq s-1$. Then 
	\begin{align*}
		\Norm{(\partial^\bbeta F)(\partial^\bgamma G)}_{L^2(\Omega)}
		&\lesssim  \norm{\norm{\partial^\bbeta F(\cdot,\varrho)}_{H^{s-|\bbeta|-\frac12}_\bx}\norm{\partial^\bgamma G(\cdot,\varrho)}_{H^{s-|\bgamma|-1}_\bx} }_{L^2_\varrho}\\
		&\lesssim \Norm{ F}_{H^{s,\bbeta_\varrho+1}}\Norm{ G}_{H^{s-1,\bgamma_\varrho}}\leq \Norm{F}_{H^{s,k}}\Norm{G}_{H^{s-1,\min(\{k,s-1\})}},
	\end{align*}
	where we used Lemma~\ref{L.product-Hs}(\ref{L.product-Hs-1}) with $(s,s_1,s_2)=(0,s-|\bbeta|-\frac12,s-|\bgamma|-1)$, and Lemma~\ref{L.embedding}. Otherwise $\bgamma_\varrho=0$ and $|\bgamma|\leq s-|\bbeta|\leq s-2$, and we have
	\begin{align*}
		\Norm{(\partial^\bbeta F)(\partial^\bgamma G)}_{L^2(\Omega)}
		&\lesssim  \norm{ \norm{\partial^\bbeta F(\cdot,\varrho)}_{H^{s-|\bbeta|}_\bx}\norm{\partial^\bgamma G(\cdot,\varrho)}_{H^{s-|\bgamma|-\frac32}_\bx} }_{L^2_\varrho}\\
		&\lesssim  \Norm{ F}_{H^{s,\bbeta_\varrho}}\Norm{ G}_{H^{s-1,1}} \leq \Norm{F}_{H^{s,k}}\Norm{G}_{H^{s-1,\min(\{k,s-1\})}}.
	\end{align*}
	The claim follows from decomposing $[\partial^\balpha,F]G$ as a sum of products as above.
\end{proof}
\begin{lemma}\label{L.commutator-Hsk-sym}
	Let $d\in\NN^\star$, $s_0>d/2$. Let $s\geq s_0+\frac52$ and $k\in\NN$ such that $2\leq k\leq s$. Then there exists $C>0$ such that for any $\balpha=(\balpha_\bx,\balpha_\varrho)\in\NN^{d+1}$ with $|\balpha|\leq s$ and $\balpha_\varrho\leq k$, one has
	\[\Norm{[\partial^\balpha;F,G]}_{L^2(\Omega)}\leq  C\Norm{F}_{H^{s-1,\min(\{k,s-1\})}}\Norm{G}_{H^{s-1,\min(\{k,s-1\})}}. \]
\end{lemma}
\begin{proof}
	We can decompose
	\[[\partial^\balpha;F,G]= \sum_{\bbeta+\bgamma=\balpha} (\partial^\bbeta F)(\partial^\bgamma G)\]
	with multi-indices $\bbeta=(\bbeta_\bx,\bbeta_\varrho)\in\NN^{d+1}$ and $\bgamma=(\bgamma_\bx,\bgamma_\varrho)\in\NN^{d+1}$ such that $|\bbeta|+|\bgamma|\leq s$ and $\bbeta_\varrho+\bgamma_\varrho\leq k$, and $1\leq |\bbeta|,|\bgamma|\leq s-1$. Assume furthermore that $\bbeta_\varrho\leq k-1$ and $|\bbeta|\leq s-2$. Then 
	\begin{align*}
		\Norm{(\partial^\bbeta F)(\partial^\bgamma G)}_{L^2(\Omega)}
		&\lesssim  \norm{ \norm{\partial^\bbeta F(\cdot,\varrho)}_{H^{s-|\bbeta|-\frac32}_\bx}\norm{\partial^\bgamma G(\cdot,\varrho)}_{H^{s-|\bgamma|-1}_\bx} }_{L^2_\varrho}\\
		&\lesssim  \Norm{ F}_{H^{s-1,\bbeta_\varrho+1}}\Norm{ G}_{H^{s-1,\bgamma_\varrho}}\leq  \Norm{F}_{H^{s-1,\min(\{k,s-1\})}}\Norm{G}_{H^{s-1,\min(\{k,s-1\})}},
	\end{align*}
	where we used Lemma~\ref{L.product-Hs}(\ref{L.product-Hs-1}) with $(s,s_1,s_2)=(0,s-|\bbeta|-\frac32,s-|\bgamma|-1)$, and Lemma~\ref{L.embedding}. 
	By symmetry, the result holds if $\bgamma_\varrho\leq k-1$ and $|\bgamma|\leq s-2$. Hence there remains to consider the situation where ($\bbeta_\varrho=k$ or $|\bbeta| =s-1$) and ($\bgamma_\varrho=k$ or $|\bgamma| =s-1$). Since $s> 2$  and $|\bbeta|+|\bgamma|\leq s$, we cannot have  $|\bbeta| =|\bgamma|=s-1$. In the same way, we cannot have $\bbeta_\varrho=\bgamma_\varrho=k$ since $k> 0$.  Furthermore , we cannot have $\bbeta_\varrho=k$  and $|\bgamma| =s-1$, since the former implies $|\bbeta|\geq \bbeta_\varrho=k\geq 2$ and the latter implies $|\bbeta|\leq 1$. Symmetrically, we cannot have $\bgamma_\varrho=k$  and $|\bbeta| =s-1$. This concludes the proof.
\end{proof}

\noindent{\bf Acknowledgments}
VD thanks Eric Blayo for revealing to him the work of Gent and McWilliams, Mahieddine Adim for his careful proofreading, as well as the Centre Henri Lebesgue, program ANR-11-LABX-0020-0. RB is partially supported by the GNAMPA group of INdAM. VD and RB thank Charlotte Perrin for identifying the relationship between Gent and McWilliams eddy-diffusivity contributions and the BD entropy, and the anonymous referee for pointing out many additional relevant references.



\end{document}